\documentclass[a4paper]{article}

\usepackage[T1]{fontenc}

\usepackage{amsmath}
\usepackage{amsfonts}
\usepackage{amssymb}
\usepackage{amsthm}

\usepackage{ifpdf}
\usepackage{ifthen}

\usepackage{enumerate}

\usepackage{aliascnt}
\ifpdf
\usepackage[colorlinks=true]{hyperref}
\else
\usepackage[draft=true]{hyperref}
\fi
\usepackage{cleveref}

\usepackage[all]{xy}

\usepackage{pifont}

\usepackage{tocbibind}

\ifpdf
\usepackage[pdftex]{graphicx,color}
\else
\usepackage{graphicx}
\fi

\usepackage{setspace}

\usepackage{makeidx}
\usepackage[refpage,intoc]{nomencl}

\swapnumbers

\newtheorem{thm}{Theorem}[subsection]
\crefname{thm}{theorem}{theorems}
\Crefname{thm}{Theorem}{Theorems}

\newaliascnt{lemma}{thm}
\newtheorem{lemma}[lemma]{Lemma}
\aliascntresetthe{lemma}

\newaliascnt{prop}{thm}
\newtheorem{prop}[prop]{Proposition}
\aliascntresetthe{prop}
\crefname{prop}{proposition}{propositions}
\Crefname{prop}{Proposition}{Propositions}

\newaliascnt{cor}{thm}
\newtheorem{cor}[cor]{Corollary}
\aliascntresetthe{cor}
\crefname{cor}{corollary}{corollaries}
\Crefname{cor}{Corollary}{Corollaries}

\theoremstyle{definition}

\newaliascnt{rem}{thm}
\newtheorem{rem}[rem]{Remark}
\aliascntresetthe{rem}
\crefname{rem}{remark}{remarks}
\Crefname{rem}{Remark}{Remarks}

\newaliascnt{example}{thm}
\newtheorem{example}[example]{Example}
\aliascntresetthe{example}

\newaliascnt{definition}{thm}
\newtheorem{definition}[definition]{Definition}
\aliascntresetthe{definition}

\newaliascnt{war}{thm}

\aliascntresetthe{war}
\crefname{war}{warning}{warnings}
\Crefname{war}{Warning}{Warnings}

\newaliascnt{block}{thm}
\newtheorem{block}[block]{}
\aliascntresetthe{block}
\crefname{block}{}{}
\Crefname{block}{}{}

\newaliascnt{conv}{thm}
\newtheorem{conv}[conv]{Convention}
\aliascntresetthe{conv}
\crefname{conv}{convention}{conventions}
\Crefname{conv}{Convention}{Conventions}

\numberwithin{equation}{section}
\crefformat{equation}{eq.~#2(#1)#3}
\Crefformat{equation}{Eq.~#2(#1)#3}

\crefname{subsection}{subsection}{subsections}
\Crefname{subsection}{Subsection}{Subsections}

\crefformat{block}{#2#1#3}
\Crefformat{block}{#2#1#3}

\crefformat{dummy}{#2#1#3}

\newcommand{\C}{\mathbb{C}}
\newcommand{\CP}{\mathbb{CP}}
\newcommand{\N}{\mathbb{N}}
\newcommand{\Z}{\mathbb{Z}}
\newcommand{\R}{\mathbb{R}}
\newcommand{\Q}{\mathbb{Q}}

\newcommand{\V}{\mathcal{V}}
\newcommand{\A}{\mathcal{A}}
\newcommand{\E}{\mathcal{E}}
\newcommand{\Nd}{\mathcal{N}}

\newcommand{\F}{\mathcal{F}}

\renewcommand{\S}{\mathcal{S}}

\newcommand{\X}{\tilde X}
\newcommand{\Stop}{\mathcal{S}_{\mathrm{top}}}
\newcommand{\Zmin}{Z_{\mathrm{min}}}
\newcommand{\Fnb}{F^{\mathrm{nb}}}

\newcommand{\Fcn}{F^{\mathrm{cn}}}
\newcommand{\bFcn}{F^{\mathrm{cn}}}
\newcommand{\Fcnm}{F^{\mathrm{cn}-}}

\renewcommand{\O}{\mathcal{O}}
\newcommand{\T}{\mathcal{T}}

\newcommand{\rk}{\mathop{\mathrm{rk}}}

\newcommand{\eu}{\mathop{\mathrm{eu}}}

\newcommand{\Vol}{\mathop{\rm Vol}\nolimits}
\newcommand{\convx}{\mathop{\rm conv}\nolimits}
\newcommand{\Gr}{\mathop{\rm Gr}\nolimits}

\newcommand{\supp}{\mathop{\rm supp}\nolimits}
\newcommand{\Hom}{\mathop{\rm Hom}\nolimits}
\newcommand{\wt}{\mathop{\rm wt}\nolimits}

\renewcommand{\div}{\mathop{\rm div}\nolimits}

\renewcommand{\mod}{\mathop{\rm mod}\nolimits}
\newcommand{\Sp}{\mathop{\rm Sp}\nolimits}
\newcommand{\pc}{\mathop{\rm pc}\nolimits}
\newcommand{\id}{\mathop{\rm id}\nolimits}
\renewcommand{\Re}{\mathop{\rm Re}\nolimits}

\newcommand{\scan}{\mathop{\sigma_{\rm can}}\nolimits}
\newcommand{\bscan}{\mathop{\bar\sigma_{\rm can}}\nolimits}

\newcommand{\spinc}{{\rm spin}^{\rm c}}
\newcommand{\Spin}{{\rm Spin}}
\newcommand{\Spinc}{{\rm Spin}^{\rm c}}
\newcommand{\BSpinc}{{\rm BSpin}^{\rm c}}
\newcommand{\KS}{\mathop{\rm KS}\nolimits}
\newcommand{\SO}{\rm SO}
\newcommand{\BSO}{\rm BSO}
\newcommand{\U}{{\rm U}}
\newcommand{\BU}{\rm BU}

\newcommand{\swun}{\mathop{\rm \bf sw}\nolimits}
\newcommand{\sw}{\mathop{\rm \bf sw}\nolimits^0}

\newcommand{\SW}{\mathop{\rm \bf SW}\nolimits^0}

\newcommand{\set}[2]{\left\{ #1 \,\middle\vert\, #2 \right\}}
\newcommand{\gen}[2]{\left\langle #1 \middle| #2 \right\rangle}

\renewcommand{\epsilon}{\varepsilon}

\newcounter{dummy}
\renewcommand{\thedummy}{\roman{dummy}}
\crefformat{dummy}{#2(#1)#3}

\newcommand{\fa}[2]{\forall #1 :\, #2}

\newcommand{\mynd}[4]
{
    \begin{figure}[#2]\begin{center}
        \ifpdf
            \input{#1.pdf_t}
        \else
            \input{#1.pstex_t}
        \fi
		\ifthenelse{\equal{#3}{}}
        {
		}
		{
            \caption{#3\label{#4}}
        }
    \end{center}
    \end{figure}
}

\hypersetup{colorlinks=false}

\author{Baldur Sigur\dh sson}
\title{
The geometric genus and Seiberg--Witten invariant of Newton nondegenerate
surface singularities}
\date{\today}

\makeindex
\makenomenclature

\begin{document}
\thispagestyle{empty}
\setlength{\parindent}{0cm}
\begin{center}
{\huge \bf
\begin{spacing}{1}
The geometric genus and  \\
Seiberg--Witten invariant\\
of Newton nondegenerate  \\
surface singularities    \\
\end{spacing}}

\vspace{.5cm}

{\large
PhD thesis\\
Baldur Sigur\dh sson}
\end{center}

\vspace{1.8cm}

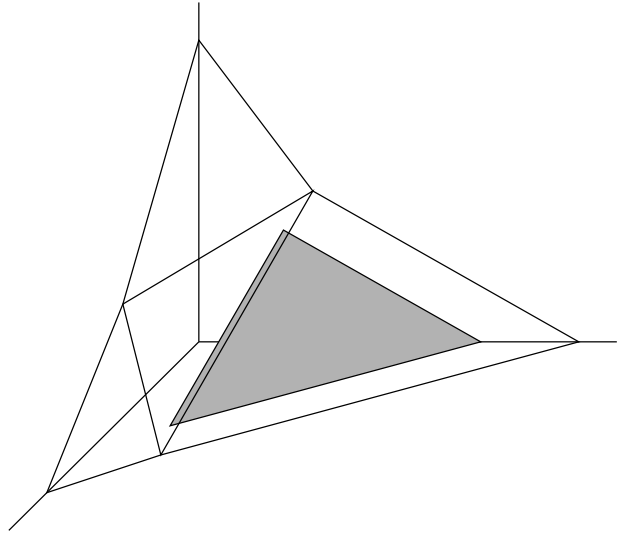
\mynd{fp}{h}{}{fig:fp}

\begin{figure}[b]
Adviser: N\'emethi Andr\'as

\vspace{.1cm}
Submitted to
Central European University\\
Department of Mathematics and its Applications\\
October 2015
\end{figure}
\setlength{\parindent}{15pt}

\newpage
\thispagestyle{empty}
\phantom{asdf}
\newpage

\thispagestyle{empty}
\begin{abstract}
Given a normal surface singularity $(X,0)$,
its link, $M$ is a closed differentiable
three dimensional manifold which carries much analytic information.
For example, a germ of a normal space is smooth if (and only if)
its link is the three sphere $S^3$ \cite{Mumford}
(it is even sufficient to assume that $\pi_1(M) = 1$).
The geometric genus $p_g$ is an analytic invariant of $(X,0)$
which, in general, cannot be recovered from the link.
However, whether $p_g = 0$ can be determined from the link \cite{Artin_iso}.
The same holds for the statement $p_g = 1$, assuming that $(X,0)$ is
Gorenstein \cite{Lauf_minell}.
It is an interesting question to ask whether, under suitable
analytic and topological conditions, the geometric genus (or other analytic
invariants) can be recovered from the link.
The Casson invariant conjecture \cite{Neu_Wa_Ca}
predicts that $p_g$ can be identified
using the Casson invariant in the case when
$(X,0)$ is a complete intersection and $M$ has
trivial first homology with integral coefficients
(the original statement identifies the
signature of a Milnor fiber rather then $p_g$, but in this case these
are equivalent data \cite{Laufer_mu,Wahl_smooth}).
The Seiberg--Witten invariant conjecture predicts that the
geometric genus of a Gorenstein
singularity, whose link has trivial first homology with
rational coefficients, can be calculated as
a normalized Seiberg--Witten invariant of the link.
The first conjecture is still open, but counterexamples
have been found for the second one.
We prove here the Seiberg--Witten invariant conjecture for hypersurface
singularities given by a function with Newton nondegenerate principal part.
We provide a theory of computation sequences and of the way they bound
the geometric genus.
Newton nondegenerate singularities can be resolved explicitly
by Oka's algorithm, and we exploit
the combinatorial interplay between the resolution graph and the
Newton diagram to show that in each step of the computation sequence
we construct, the given bound is sharp. Our method recovers
the geometric genus of $(X,0)$ explicitly from the link, assuming that 
$(X,0)$ is indeed Newton nondegenerate with a rational homology sphere link.
Assuming some additional information about the Newton diagram, we recover
part of the spectrum, as well as the Poincar\'e series associated with
the Newton filtration.
Finally, we show that the normalized Seiberg--Witten invariant associated
with the canonical $\spinc$ structure on the link coincides with
our identification of the geometric genus.
\end{abstract}

\newpage

\tableofcontents
\listoffigures
\vspace{0.5cm}
\begin{block}
The figure on the front page shows the Newton diagram of the singularity
given by the equation
$x_1^4 + x_1^3 x_2^2 + x_2^{10} + x_1^2 x_3^3 + x_2^3 x_3^4 + x_3^8 = 0$.
The gray polygon is $\Fcn_i$ (see \cref{def:cone_face})
for a certain $i$ in case \ref{rt:Newton} (see \cref{def:rt}).
\end{block}
\newpage

\section{Introduction}

This text was written in 2015, in partial fulfillment of the requirements
for the degree of doctor of philosophy in mathematics at Central European
University in Budapest, under the supervision of N\'emethi Andr\'as.

\subsection{Content}

In \cref{s:general} we recall some results on two dimensional singularities
and fix notation. These include
a formula for the geometric genus in terms of the Poincar\'e series
and a similar formula for the normalized Seiberg--Witten invariant of the link
in terms of the zeta function,
a general theory of computation sequences,
the polynomial part and periodic constant of a power series in one variable,
a short review of the spectrum of hypersurface singularities,
as well as a result of Saito on part of the spectrum.
In the last subsection we give a detailed presentation of our results
as well as an outline of the proofs.

In \cref{s:newton_diag} we recall the definition of Newton nondegeneracy
for a hypersurface singularity, and the construction of its Newton diagram.
We recall Oka's algorithm
for isloated singualarities of surfaces in $\C^3$
which provides
the graph of a resolution of the singularity from the Newton diagram and
discuss conditions of minimality and convenience.
Next we recall the Newton filtration and its associated Poincar\'e series.
In the last section we recall
Braun and N\'emethi's classification of Newton diagrams giving rise
to rational homology sphere links
which is crucial to the proof in \cref{s:SW}.

In all the following sections, we will assume that $(X,0)$ is a
hypersurface singularity, given by a function with Newton nondegenerate
principal part, with a rational homology sphere link. Furthermore,
$G$ is the resolution graph produced by Oka's algorithm from the
Newton diagram of this function.

In \cref{s:affine} we fix some notation regarding polygons in two dimensional
real affine space, and give a result on counting integral points in such
polygons.

In \cref{s:seq}, we construct three computation sequences on $G$
and prove a formula
which says that the intersection numbers along these sequences count the
integral points under the Newton diagram, or in the positive octant of $\R^3$.

In \cref{s:pg} we apply the formula from the previous section to prove that
the computation sequences constructed calculate the geometric genus, as well
as part of the spectrum and the Poincar\'e series associated with the Newton
filtration.  In particular, this gives a simple topological
identification of the geometric genus for two dimensional Newton nondegenerate
hypersurface singularities.

In \cref{s:SW}, we prove that one of the computation sequences constructed
in \cref{s:seq} calculates the normalized Seiberg--Witten invariant
for the canonical $\spinc$ structure on the link. As a corollary, we
prove the Seiberg--Witten invariant conjecture for $(X,0)$.

\subsection{Acknowledgements}

I would like to thank my adviser, N\'emethi Andr\'as, for his great
support and encouragement, and the many things he has taught me.
I would also like to thank Central European University and
Alfr\'ed R\'enyi Institute of Mathematics, particularly Stipsicz Andr\'as,
for providing me with the opportunity to stay in Budapest to study
mathematics.
I would also like to thank Patrick Popescu-Pampu for his numerous helpful
remarks, which I believe improved the text considerably.
Finally, I would like to thank my colleagues, friends and
family, whose moral support has been indispensable to my work.

\subsection{Notation}

The
\emph{content} \index{content}
of an integral vector $a\in\Z^N$ is the greatest
common divisor of its coordinates.
A \emph{primitive} vector is a vector
whose content is $1$.
If $p,q\in\Z^N$, then we say that
the segment $[p,q]$ is primitive if $q-p$ is a primitive vector.
If we consider $\Z^N$ as an affine space, and
$\ell:\Z^N\to\Z$ is an affine function, then its \emph{content}
is the index of its image as a coset in $\Z$.
Equivalently, the content $c$ of $\ell$ is the largest $c\in\Z$
for which there exists an
affine function $\tilde \ell:\Z^n\to\Z$ and a constant $b\in\Z$ so that
$\ell = c\tilde\ell + b$. An affine function is \emph{primitive} if its content
is $1$.

\newpage
\section{General theory and statement of results} \label{s:general}

In this section we will recall some facts about 
singularities and fix some notation. We will always assume that
$(X,0)$ \nomenclature[X]{$(X,0)$}{Germ of a singular space}
is a germ of a normal complex surface singularity, embedded in some
$(\C^N,0)$. Furthermore, when choosing a representative $X$ of the germ $(X,0)$,
we assume $X$ to be a contractible Stein space given as the intersection
of a closed analytic set and a suitably small ball around the origin, and
that $X$ is smooth outside the origin.

\subsection{The link}
\index{link}

In this section we denote by $S_r^{d-1}\subset\R^d$ the sphere with radius
$r$ around the origin in $\R^d$, by $B_r^d \subset\R^d$ the ball
with radius $r$ and by $\bar B_r^d$ its closure.
For the definition of
\emph{plumbing graphs}, \index{plumbing graph}
we refer to
\cite{Neu_plumb,Mumford,Nem_Szil,P-P_cfrac}.
Recall that each vertex $v$ of a plumbing graph is labelled by two integers,
the
\emph{selfintersection number} \index{selfintersection number}
$-b_v$ \nomenclature[b]{$-b_v$}{Selfintersection number}
and the \emph{genus} $g_v$.
Furthermore, denoting the vertex set of the graph by
$\V$, \nomenclature[V]{$\V$}{Vertex set of a resolution or plumbing graph}
then there
is an associated $|\V|\times |\V|$
\emph{intersection matrix} \index{intersection matrix}
$I$ \nomenclature[I]{$I$}{Intersection matrix}
with $I_{v,v} = -b_v$ and $I_{v,w}$ the number of edges between $v$ and $w$
if $v\neq w$.

\begin{definition}\label{def:link}
Let $(X,0)$ be a germ of an isolated surface singularity.
Its 
\emph{link} \index{link}
is the three dimensional manifold
$M = X\cap S^{2N-1}_r$ \nomenclature[M]{$M$}{Link of a singularity}
where
we assume given some embedding $(X,0) \to (\C^N,0)$ and the radius
$r>0$ is sufficiently
small. As a differentiable manifold, $M$ does not depend on the embedding 
$(X,0) \hookrightarrow (\C^N,0)$, or $r$ (see e.g. \cite{Looijenga}).
\end{definition}

The topology (or embedded topology) of a singularity is completely encoded
in its link (or the embedding $M\hookrightarrow S^{2N-1}_r$ of the link). 

\begin{prop}[\cite{Milnor_hyp,Looijenga}]
Let $(X,0)$ be a singularity embedded into $(\C^N,0)$ for some $N>0$
and let $r>0$ be small enough. Then the pair
$(\bar B_r^{2N}, X\cap \bar B_r^{2N})$ is homeomorphic to the cone
over the pair $(S^{2N-1},M)$. \qed
\end{prop}

\begin{block}
In \cite{Mumford}, Mumford proved that the germ of a normal two dimensional
space is smooth if and only if the link is simply connected. He also showed
that the link can always be described by a
\emph{plumbing graph}. \index{plumbing graph}
These graphs
were studied by Neumann in \cite{Neu_plumb} where he gave a calculus for
determining whether two graphs yield the same manifold. Furthermore,
every graph is equivalent to a unique
minimal graph \index{minimal graph}
which is easily determined
from the original graph.
A plumbing graph for the link may be obtained from a resolution as described
in \cref{ss:res}.
\end{block}

\begin{prop}[Grauert \cite{Gr_exzept}] \label{prop:exzept}
Let $M$ be the three dimensional manifold obtained from the plumbing graph $G$.
Then $M$ is the link of some singularity if and only if $G$ is
connected and the associated intersection matrix is negative definite.
\qed
\end{prop}

\begin{prop}[Mumford \cite{Mumford}] \label{prop:H}
Let $M$ be the three dimensional
manifold obtained from the plumbing graph $G$ and
assume that the associated intersection matrix is negative definite.
Let $g = \sum_{v\in\V} g_v$ be the sum of genera of the vertices of
$G$ and define $c$ as the first Betti number of the
topological realisation of the graph $G$, that is, number of independent
loops. Then $H_1(M,\Z)$ has rank $c+2g$ and torsion the cokernel
of the linear map given by the intersection matrix.
In particular, we have $H_1(M,\Q) = 0$ if and only if $G$ is a tree
and $g_v = 0$ for all vertices $v$. \qed
\end{prop}

\begin{definition}
A closed three dimensional manifold $M$ is called a
\emph{rational homology sphere} \index{homology sphere!rational}
(\emph{integral homology sphere}) \index{homology sphere!integral}
if $H_i(M,\Q) \cong H_i(S^3,\Q)$ ($H_i(M,\Z) \cong H_i(S^3,\Z)$).
By Poincar\'e duality, this is equivalent to
$H_1(M,\Q) = 0$ ($H_1(M,\Z) = 0$).
\end{definition}

\subsection{Resolutions of surface singularities} \label{ss:res}

\begin{definition} \label{def:res}
Let $(X,0)$ be a normal isolated singularity. A
\emph{resolution} \index{resolution}
of $X$ is a holomorphic manifold
$\X$, \nomenclature[X]{$\X$}{Resolution of $X$}
together with a proper surjective map
$\pi:\X \to X$ \nomenclature[p]{$\pi$}{Resolution $\pi:\X\to X$}
so that
$E = \pi^{-1}(0)$ \nomenclature[E]{$E$}{Exceptional divisor of a resolution}
is a divisor in $\X$ and
the induced map $\X\setminus E \to X\setminus\{0\}$ is biholomorphic.
We refer to $E$ as the
\emph{exceptional divisor} \index{exceptional divisor}
of the resolution $\pi$.
We say that $\pi$ is a
\emph{good} resolution \index{good resolution}
if $E\subset \X$ is a
\emph{normal crossing divisor}, \index{normal crossing divisor}
that is, a union of smooth submanifolds
intersecting transversally, with no triple intersections.
We will always assume this condition.
Write $E = \cup_{v\in\V} E_v$, where
$E_v$ \nomenclature[Ev]{$E_v$}{Irreducible component of exceptional divisor}
are the irreducible components
of $E$.
Denote by $g_v$ the genus of (the normalisation of) the curve $E_v$
and by $-b_v$ the Euler number of the normal bundle of $E_v$ as a
submanifold of $\X$.
\end{definition}

\begin{definition} \label{def:res_graph}
Let $\pi:(\X,E)\to(X,0)$ be a (good) resolution as above. The
\emph{resolution graph} \index{resolution graph}
$G$ \nomenclature[G]{$G$}{Resolution or plumbing graph}
associated with $\pi$ is the graph with vertex set $\V$ and
$|E_v\cap E_w|$ edges between $v$ and $w$ if $v\neq w$ and no loops.
It is decorated with the
\emph{selfintersection numbers} \index{selfintersection number}
$-b_v$ 
and
\emph{genera} $g_v$ for $v\in\V$. We denote by
$\delta_v$ \nomenclature[d]{$\delta_v$}{Degree of a vertex}
the
\emph{degree} \index{degree}
of a vertex $G$, that is, $\delta_v = \sum_{w\neq v} |E_v\cap E_w|$.
\end{definition}

\begin{prop}[Mumford \cite{Mumford}]
Let $M$ be the link of a singularity admitting a resolution with resolution
graph $G$. Then $M$ is the plumbed manifold obtained from the plumbing
graph $G$.\qed
\end{prop}

\begin{prop}[Zariski's main theorem]
If $G$ is the graph of a resolution of a normal singularity, then
$G$ is connected.
\end{prop}
\begin{proof}
This follows from the fact that $E$ is a connected variety,
see e.g. \cite{Hartshorne_AG}, Corollary 11.4.
\end{proof}

\begin{block}
Given an embedding of $(X,0)$ into some smooth space $(\C^N,0)$, we may
take as a representative for the germ an intersection with a \emph{closed}
ball of sufficiently small radius. Then, the
resolution $\X$ is given as a manifold with boundary and
$\partial \X = M$. In particular, one can consider the perfect pairing
$H_2(\X,\Z)\otimes H_2(\X,M,\Z) \to \Z$ which induces a symmetric form
$(\cdot,\cdot):H_2(\X,\Z)^{\otimes2}\to\Z$.

The exceptional divisor $E$ is a strong homotopy retract of $\X$. In
particular, $H_2(\X,\Z) = \Z\gen{E_v}{v\in\V}$ and 
$H_2(\X,M,\Z) = \Hom(H_2(\X,\Z),\Z)$ is free.
If $v\neq w$, then $(E_v, E_w) = |E_v\cap E_w|$. Further,
$E_v^2 = (E_v, E_v)$ is the Euler number of the normal bundle
of the submanifold $E_v \subset \X$.
The intersection form is negative definite, in particular,
nondegenerate \cite{Mumford}. This means that the natural map
$H_2(\X,\Z) \to H_2(\X,M,\Z)$ may be viewed as an inclusion with
finite cokernel.
In particular, we may view $H_2(\X,M,\Z)$ as a lattice in
$H_2(\X,\Z)\otimes\Q$, containing $H_2(\X,\Z)$ with finite index.
\end{block}

\begin{definition} \label{def:intersection_form}
Let
$L = H_2(\X,\Z) = \Z\gen{E_v}{v\in\V}$
\nomenclature[L]{$L$}{Intersection lattice}
and
$L' = H_2(\X,M,\Z) = \Hom(L,\Z)$.
\nomenclature[Lp]{$L'$}{Dual intersection lattice}
We refer to these as the \emph{lattice} and the \emph{dual lattice}
associated with the resolution $\pi$. They are endowed with a partial
order by setting $l_1\geq l_2$ if and only if $l_1 - l_2$ is an effective
divisor.
The form
$(\cdot,\cdot):L\otimes L\to\Z$
\nomenclature{$(\cdot,\cdot)$}{Intersection form}
defined above is the
\emph{intersection form}. \index{intersection form}
We extend the intersection form to
$L_\Q = L\otimes\Q$ and $L_\R=L\otimes\R$
by linearity. Elements of $L$ (or $L_\Q$, $L_\R$)
will be referred to as \emph{cycles} with integral (rational, real)
coefficients. We set
$H = L'/L$. \nomenclature[H]{$H$}{Cokernel of the intersection form}
The intersection form is encoded in the
\emph{intersection matrix} \index{intersection matrix}
$I = ( (E_{v}, E_{w}) )_{v,w\in \V}$.
\nomenclature[Ip]{$I$}{Intersection matrix}
This matrix is invertible over $\Q$, and we write
$I^{-1} = (I_{v,w}^{-1})$.
\nomenclature[Iivw]{$I_{v,w}^{-1}$}{Entries of the inverted intersection
matrix}
\end{definition}

\begin{rem} \label{rem:Lp}
By the above discussion, it is clear that we have an identification
$L' = \set{l\in L_\Q}{\fa{l'\in L}{(l,l')\in\Z}}$.
Furthermore, one obtains the short exact sequence
\[
  0 \to L'/L \to H_1(M,\Z) \to H_1(E) \to 0
\]
from the long exact sequence of the pair $(\X,M)$,
which gives a canonical isomorphism between $H$ and the torsion submodule
of $H_1(M)$.
\end{rem}

\begin{definition} \label{def:Z_K}
The
\emph{canonical cycle} \index{canonical cycle}
$K\in L'$ \nomenclature[K]{$K$}{Canonical cycle}
is the unique cycle satisfying the
\emph{adjunction equalities} \index{adjunction equalities}
$(K,E_v) = -E_v^2 + 2g_v - 2$.
We define the
\emph{anticanonical cycle} \index{anticanonical cycle}
as
$Z_K = -K$. \nomenclature[ZK]{$Z_K$}{Anticanonical cycle}
We say that $G$ is
\emph{numerically Gorenstein} \index{numerically Gorenstein}
if $K \in L$.
\end{definition}

\begin{rem}
\begin{list}{(\thedummy)}
{
  \usecounter{dummy}
  \setlength{\leftmargin}{0pt}
  \setlength{\itemsep}{0pt}
  \setlength{\itemindent}{4.5pt}
}

\item
The nondegeneracy of the intersection form guarantees the existence of
$Z_K$ as a cycle with rational coefficients. By \cref{rem:Lp} we have
$Z_K\in L'$. For
hypersurface \index{hypersurface}
singularities (more generally, for
Gorenstein \index{Gorenstein}
singularities) we have, in fact, $Z_K\in L$.
Indeed, $K$ is numerically equivalent to the divisor defined by
any meromorphic differential form on $\X$. In the case of a hypersurface
singularity (or, more generally, a Gorenstein singularity),
there exists a meromorphic $2$-form on $\X$ whose divisor
is exactly $K$. For details, see e.g. \cite{Durfee_smoothings,Nemethi_FL}.

\item
This definition of the canonical cycle assumes that all components $E_v$
are smooth. If this is not the case, the correct formula also contains
a term counting the ``number of nodes and cusps'' on $E_v$, see e.g.
\cite{Nemethi_FL}.

\item
An isolated singularity $(X,0)$ is said to be
\emph{Gorenstein} \index{Gorenstein}
if the canonical line bundle $\Omega^2_{X\setminus\{0\}}$ in
a punctured neighbourhood around $0$ is trivial. Gorenstein singularities
are numerically Gorenstein \index{numerically Gorenstein}
\cite{Durfee_smoothings,Nemethi_FL}
and
hypersurfaces \index{hypersurface}
(more generally,
complete intersections) \index{complete intersection}
are Gorenstein
\cite{Nemethi_FL}.
Similarly, $(X,0)$ is said to be
$\Q$-Gorenstein \index{Gorenstein!$\Q$}
if some tensor power of $\Omega^2_{X\setminus\{0\}}$ is trivial.
\end{list}
\end{rem}

\begin{definition} \label{def:dual_c}
The
\emph{dual cycles} \index{dual cycles}
$E_v^* \in L'$, $v\in\V$, \nomenclature[Es]{$E^*_v$}{Dual cycles}
are defined by the linear
equations
$(E_v^*,E_w) = -\delta_{v,w}$, where $\delta_{v,w}$ is the Kronecker delta.
These exist and are well defined since the intersection matrix $I$ is
invertible over $\Q$.
In fact, we have $E_v^* = \sum_{w\in\V} -I^{-1}_{v,w} E_w$.
It follows that the family $(E_v^*)_{v\in\V}$ is a basis of
$L'$. In particular, we have $E_v^*\in L$ for all $v\in\V$ if and only if
$M$ is an integral homology sphere.
\index{homology sphere!integral}
\end{definition}

\begin{definition}
For a cycle $Z = \sum_{v\in\V} m_v E_v \in L$, write
$m_v(Z) = m_v$. \nomenclature[m]{$m_v(Z)$}{Coefficient of a cycle $Z$}
\end{definition}

\begin{lemma} \label{lem:Es_pos}
The entries $m_w^{\phantom*}(E_v^*) = -I_{v,w}^{-1}$ are positive.
\end{lemma}
\begin{proof}
Write
$E_v^* = Z_1 - Z_2$, where $m_v(Z_i) \geq 0$ for all $v$ and $i=1,2$, and
$Z_1, Z_2$ have disjoint supports (the \emph{support} of a cycle is
$\supp(Z) = \set{v\in\V}{m_v(Z) \neq 0}$). Since
$(-Z_2, E_v) \leq (Z, E_v) \leq 0$ for all $v\in \supp(Z_2)$, we find
$Z_2^2 \geq 0$, hence $Z_2 = 0$ by negative definiteness and so
$Z = Z_1$. We must show $\supp(Z_1) = \V$. Since $Z \neq 0$, if there
is a $v\in \V\setminus\supp(Z)$, we may assume that there is such a $v$
having a neighbour in $\supp(Z)$. This would give
$(Z, E_v) = \sum_{u\in\V_v\cap\supp(Z)} m_u(Z) > 0$ contradicting our
assumptions.
\end{proof}

\subsection{The topological semigroup} \label{ss:Lipman}

Throughout this subsection we assume given a good resolution $\pi:\X\to X$
as described in the previous subsection. We also assume that the link $M$
is a rational
homology sphere. \index{homology sphere!rational}

\begin{definition} \label{def:Stop}
The
\emph{Lipman cone} \index{Lipman cone}
is the set
\[
  \Stop = \set{Z\in L}{\fa{ v\in \V}{ (Z,E_v) \leq 0}}.
\nomenclature[Stop]{$\Stop$}{Lipman cone, topological semigroup}
\]
We also define
\[
  \Stop' = \set{Z\in L'}{\fa{ v\in \V}{ (Z,E_v) \leq 0}}.
\nomenclature[Stopp]{$\Stop'$}{Lipman cone, topological semigroup}
\]
\end{definition}

\begin{rem}
We have $\Stop' = \N\gen{E_v^*}{v\in\V}$ and $\Stop = \Stop' \cap L$.
\end{rem}

\begin{prop} \label{prop:gStop}
Let $g\in\O_{X,0}$ and define $Z\in L$ by setting $m_v(Z)$ equal to the
divisorial valuation of $\pi^*g$ along $E_v\subset \tilde X$. Then
$Z\in \Stop$.
\end{prop}
\begin{proof}
We have $(g) = \sum_{v\in\V} m_v(Z) E_v + S$ where $S$ is a divisor, none of
whose components are supported on $E$. In particular, we have $(E_v,S) \geq 0$
for all $v\in \V$. Furthermore, $(g)$ is linearly
equivalent to $0$ in the divisor group, which gives
$(E_v,(g)) = 0$ for all $v$. Thus, $(E_v,Z) = -(E_v,S) \leq 0$.
\end{proof}

\begin{definition}
Let $Z_i = \sum_v m_{v,i} E_v \in L'$, $i=1,2$. Define their
\emph{meet} \index{meet}
as $Z_1\wedge Z_2 = \sum_v \min\{m_{v,1},m_{v,2}\} E_v$.
\nomenclature{$\wedge$}{Meet operation}
\end{definition}

\begin{prop}[Artin \cite{Artin_iso}] \label{prop:meet}
The Lipman cone is closed under addition, and therefore makes up a semigroup.
The same holds for $\Stop'$.
Furthermore, if $Z_1,Z_2 \in \Stop'$, then $Z_1\wedge Z_2 \in \Stop'$.
\end{prop}
\begin{proof}
The first statement is clear, since $\Stop\subset L$ and
$\Stop'\subset L'$ are given by inequalities,
and are therefore each given as the set of integral points in a real convex
cone.
For the second statement, write $Z_i = \sum_v m_{v,i}$ and set
$m_v = \min\{m_{v,1},m_{v,2} \}$. Assuming
$Z_1, Z_2\in\Stop'$, and, say, $m_v = m_{v,1}$, we get
\[
\begin{split}
  (Z_1\wedge Z_2,E_v) &= m_{v,1} E_v^2 + \sum_{w\neq v} m_w (E_v, E_w) \\
   &\leq m_{v,1} E_v^2 + \sum_{w\neq v} m_{w,1} (E_v, E_w)
    =    (Z_1,E_v) \leq 0.
\end{split}
\]
\end{proof}

\begin{definition} \label{def:Zmin}
By \cref{lem:Es_pos}, the elements in $\Stop$ have positive entries.
Therefore, the partially ordered set $\Stop\setminus\{0\}$ has minimal
elements. Furthermor, by \cref{lem:Es_pos}, the meet $Z_1\wedge Z_2$ of two
elements $Z_1, Z_2\in\Stop\setminus\{0\}$ is again nonzero. Thus, the set
$\Stop\setminus\{0\}$ contains a unique minimal element. We denote this
element by
$\Zmin$ \nomenclature[Zmin]{$\Zmin$}{Artin's minimal cycle}
and call it
\emph{Artin's minimal cycle}, \index{Artin's minimal cycle}
or, the
\emph{minimal cycle}. \index{minimal cycle}
This element is often referred to as the
\emph{fundamental cycle}. \index{fundamental cycle}
\end{definition}

\subsection{Topological zeta and counting functions} \label{ss:ZQ}

\begin{block} \label{bl:series}
We will make use
of the set
$\Z[[t^L]] = \set{\sum_{l\in L}a_l t^l}{a_l \in \Z}$.
\nomenclature[ZtL]{$\Z[[t^L]]$}{Set of power series}
It is a group under addition, and has a partially defined multiplication.
More precisely, if $A(t) = \sum a_l t^l$ and $B(t) = \sum b_l t^l$ are 
elements of $\Z[[t^L]]$, then $A(t) \cdot B(t)$ is defined
if the sum $c_l = \sum_{l_1+l_2 = l} a_{l_1} b_{l_2}$
is finite for all $l\in L$, in which case we define
$A(t) \cdot B(t) = \sum_{l\in L} c_l t^l$.
In particular, $\Z[[t^L]]$ is a module over the ring of
Laurent polynomials \index{Laurent polynomials}
$\Z[t_1^{\pm 1}, \ldots, t_s^{\pm 1}]$ where $s = |\V|$.
A simple exercise also shows that if $A(t) = \sum a_lt^l$ is supported in
the
Lipman cone \index{Lipman cone}
(that is, $a_l = 0$ for $l\notin\Stop$) then
$A(t) \cdot \sum_{l\not\geq 0} t^l$ is well defined.

In precisely the same way, one obtains the set
$\Z[[t^{L'}]] = \set{\sum_{l\in L'} a_l t^l}{a_l\in\Z}$ which naturally
contains $\Z[[t^L]]$, and is contained in
$\Z[[t_1^{\pm1/d}, \ldots, t_s^{\pm1/d}]]$, where $d = |H|$.

One may modify this definition by introducing coefficients from any ring
$R$, thus obtaining $R[[t^L]]$.

For a discussion of these sets and some rings contained in them,
see e.g. \cite{Nem_Poinc}.
\end{block}

\begin{rem}
If $C\subset L_\R$ is a strictly convex cone (i.e. contains no nontrival
linear space) and $A(t),B(t) \in \Z[[t^{L'}]]$ as above, with
$a_{l'} = b_{l'} = 0$ if $l'\notin C$, then $A(t) \cdot B(t)$ is
well defined. As is easily seen, the set of such series thus form a
local ring, with maximal ideal the set of series $A(t)$ with $a_0 = 0$.
In particular, if $l'\in C$, and $l'\neq 0$ then the element $1-t^{l'}$
is invertible, in fact we have $(1-t^{l'})^{-1} = \sum_{k=0}^\infty t^{kl'}$.
Since this is independent of the cone $C$, we will assume this formula
without referring to $C$.
\end{rem}

\begin{definition}[\cite{CDGZ,Nem_Poinc}] \label{def:zeta}
For $l\in L'$, denote by $[l]\in H = L'/L$ the associated residue class.
Denote by
$\hat H = \Hom(H,\C)$ \nomenclature[Hh]{$\hat H$}{Pontrjagin dual of $H$}
the Pontrjagin dual of the group $H$.
The intersection product induces an isomorphism
$\theta:H\to\hat H$, $[l] \mapsto e^{2\pi i(l,\cdot)}$.
The
\emph{equivariant zeta function} \index{equivariant zeta function}
associated with the resolution graph $G$
is
\begin{equation} \label{eq:zeta}
  Z(t) = \prod_{v\in\V} (1 - [E_v^*] t^{E_v^*})^{\delta_v-2}
       \in \Z[H][[t^{L'}]].
\end{equation} \nomenclature[Z(t)]{$Z(t)$}{Equivariant zeta function}
The natural bijection $\Z[H][[t^{L'}]] \leftrightarrow \Z[[t^{L'}]][H]$
induces well defined series $Z_h(t) \in \Z[[t^{L'}]]$ for each $h\in H$
so that $Z(t) = \sum_{h\in H} Z_h(t) h$. It is clear that the series
$Z_h(t)$ is supported on the coset of $L$ in $L'$ corresponding to $h$, that
is, the coefficient of $l'\in L'$ in $Z_h(t)$ vanishes if $[l'] \neq h$.
In particular, we have
$Z_0(t) \in \Z[[t^L]]$,\nomenclature[Z0(t)]{$Z_0(t)$}{Zeta function}
where $0$ denotes the trivial
element of $H$. We call $Z_0(t)$ the
\emph{zeta function} \index{zeta function}
associated with
the graph $G$. Denote by
$z_{l'}\in\Z$
\nomenclature[zl]{$z_{l'}$}{Coefficient of (equivariant) zeta function}
the coefficients of $Z(t)$, i.e.
$Z(t) = \sum_{l'\in L'} z_{l'} [l'] t^{l'}$. Thus, we have
$Z_0(t) = \sum_{l\in L} z_{l} t^{l}$.

The
\emph{equivariant counting function} \index{equivariant counting function}
associated with $G$ is the series
$Q(t) = \sum_{l'\in L'} q_{l'} [l'] t^{l'} \in \Z[H][[t^{L'}]]$,
\nomenclature[Q(t)]{$Q(t)$}{Equivariant counting function}
where
$q_{l'} = \sum\set{z_{l'+l}}{l\in L,\,l\not\geq 0}$.
\nomenclature[ql]{$q_{l'}$}{Coefficient of (equivariant) counting function}
This yields a
decomposition $Q(t) = \sum_{h\in H} Q_h(t) h$ where $Q_h\in\Z[[t^{L'}]]$ as
above. In particular,
$Q_0(t) \in \Z[[t^L]]$. \nomenclature[Q0(t)]{$Q_0(t)$}{Counting function}
The series $Q_0(t)$ is called the
\emph{counting function} \index{counting function}
associated with $G$.
\end{definition}

\begin{rem}
\begin{list}{(\thedummy)}
{
  \usecounter{dummy}
  \setlength{\leftmargin}{0pt}
  \setlength{\itemsep}{0pt}
  \setlength{\itemindent}{4.5pt}
}

\item
The zeta function is supported on the Lipman cone, that is, writing
$Z_0(t) = \sum_{l\in L} z_l t^l$ we have $z_l = 0$ if $l\notin\Stop$.

\item
It is not immediately obvious why the series above is given the name
'zeta function'. This may be motivated by similarities to the formula
of A'Campo \cite{ACampo}.

\item
The counting function is a topological analogue of the Hilbert series.
This is made clear by \cref{prop:H_from_P} and the relationship between
the Poincar\'e series (see \cref{def:Hilb_Poinc}) and the zeta function
described in \cite{Nem_Poinc}. See also \cref{rem:SW}\cref{it:SW_MI}.
Its title is inspired with its connections with Ehrhart theory
\cite{LaszNem_Ehr}.

\item
In the more general 'equivariant setting', one studies invariants of the
universal abelian cover of a singularity along with its action of
the group $H$, see e.g. \cite{NeumWa_SpQ,Neu_AbCovQH,Nem_Nico_SW,Laszlo_th}.
This explains why the zeta function and counting function are considered
equivariant.

\end{list}
\end{rem}

\subsection{The geometric genus}

\begin{definition} \label{def:pg}
Let $(X,0)$ be a normal surface singularity.
The
\emph{geometric genus} \index{geometric genus}
of $(X,0)$ is defined as
$p_g = h^1(\X,\O_{\X})$, \nomenclature[pg]{$p_g$}{Geometric genus}
where $\X\to X$ is a resolution.
\end{definition}

The geometric genus of $(X,0)$ is defined in terms of a resolution. Using
the fact that any resolution is obtained by blowing up the minimal resolution,
as well as Lemma 3.3 from \cite{Lauf_rat}, one finds that $p_g$ is independent
of the resolution. This fact also follows from the following formula of Laufer:

\begin{prop}[Laufer \cite{Lauf_rat}, Theorem 3.4]
We have
\begin{equation} \label{eq:Laufer_sq_int}
  p_g = \dim_\C\frac{H^0      (X\setminus 0,\Omega^2_{X\setminus 0})}
                    {H^0_{L^2}(X\setminus 0,\Omega^2_{X\setminus 0})},
\end{equation}
where $H^0(X\setminus 0,\Omega^2_{X\setminus 0})$ is the set of germs of
holomorphic two forms defined around the origin, and 
$H^0_{L^2}(X\setminus 0,\Omega^2_{X\setminus 0})$ is the subset of
square integrable forms. \qed
\end{prop}

\begin{block}
Assume that we have a
resolution \index{resolution}
$\pi:\X\to X$ as in \ref{ss:res} and
take $\omega \in H^0(X\setminus 0,\Omega^2_{X\setminus 0})$. 
By Laufer \cite{Lauf_rat}, $\omega$ is square integrable if and only if
$\pi^*(\omega)$ extends to a holomorphic form on $\X$.
\end{block}

\begin{definition} \label{def:Hilb_Poinc}
Assume given a resolution $\pi:\X\to X$, with notation as in \ref{ss:res}.
The
\emph{divisorial filtration} \index{divisorial filtration}
is a multiindex filtration of $\O_{X,0}$
by ideals, given by
\[
  \F(l) = \set{f\in\O_{X,0}}{\div(f) \geq l}
        = \pi_* H^0(\X,\O_{\X}(-l)), \quad l\in L.
\] \nomenclature[F(l)]{$\F(l)$}{Divisorial ideal}
For $l\in L$ we set
$h_l = \dim_\C \O_{X,0}/\F(l)$
\nomenclature[hl]{$h_l$}{Coefficient of Hilbert series}
and define the
\emph{Hilbert series} \index{Hilbert series}
as
\[
  H(t) = \sum_{l\in L} h_l t^l \in \Z[[t^L]].
\nomenclature[H(t)]{$H(t)$}{Hilbert series}
\]
The
\emph{Poincar\'e series} \index{Poincar\'e series}
is defined as
\[
  P(t) = \sum_{l\in L} p_l t^l =  -H(t) \prod_{v\in\V} (1-t_v^{-1}).
\]
\nomenclature[P(t)]{$P(t)$}{Poincar\'e series}
\nomenclature[pl]{$p_l$}{Coefficient of Poincar\'e series}
\end{definition}

\begin{rem}
Classically, a Hilbert series is the generating function of the
numbers $A/I_n$, where $(I_n)_{n\in\N}$ is a filtration of the algebra
$A$. Similarly, if $A = \oplus_{n=0}^\infty A_n$ is a graded algebra
(e.g. $A_n = I_n/I_{n+1}$), then its associated Poincar\'e series
is the generating function for the numbers $\dim_\C A_n$.
This coincides with our definition in the case when $\rk L = 1$.
\end{rem}

\begin{prop}[N\'emethi \cite{Nem_Poinc}]
The Poincar\'e series is supported on the Lipman cone, that is, if
$l\notin\Stop$ then $p_l = 0$. \qed
\end{prop}

\begin{block}
The Poincar\'e series is obtained by a simple formula from the Hilbert
series. There are, however, nonzero elements in $\Z[[t^L]]$ whose product
with $1-t_v^{-1}$ is defined and equals zero.
This means that, in principle, one cannot use this formula
to determine $H$ from $P$.
The following proposition guarantees that
one may nonetheless determine $H$ from $P$. The two series therefore
provide equivalent data.
\end{block}

\begin{prop}[N\'emethi \cite{Nem_Poinc}] \label{prop:H_from_P}
Let $H$ and $P$ be as in \cref{def:Hilb_Poinc}. Then, for any $l\in L$,
we have
\[
  h_l = \sum_{\substack{l'\in L\\l'\not\geq l}} p_{l'}.
\]
Equivalently, we have
$H(t) = \left(\sum_{l\not\geq 0} t^l\right) \cdot P(t)$. \qed
\end{prop}

\begin{prop}[N\'emethi \cite{Nem_Poinc,Nem_coh_splice}] \label{prop:pg_hl}
\index{geometric genus}
Assume the notation in \ref{ss:res} and \ref{ss:Lipman} and let $H$ and $P$
be as in \ref{def:Hilb_Poinc}.
Then, if $l\in L$ and $(l,E_v) \leq (Z_K,E_v)$ for all $v\in \V$,
then
\begin{equation} \label{eq:pg_hl}
  h_l = p_g + \frac{(Z_K - l,l)}{2}.
\end{equation}
In particular, if $(X,0)$ is numerically Gorenstein, then 
$h_{Z_K} = p_g$.
\qed
\end{prop}

\begin{rem}
The condition $l\in L$ and $(l,E_v) \leq (Z_K,E_v)$ for all $v\in \V$
is equivalent to $l \in (Z_K+\Stop')\cap L$. The formula \cref{eq:pg_hl}
holds for any $l\in L'$ satisfying the same conditions, once the term
$h_l$ is defined for such $l$. These are the coefficients of the
\emph{equivariant Hilbert series}, which will not be discussed here.
\end{rem}

\begin{block}
Combining \cref{prop:H_from_P} and \cref{prop:pg_hl}, one finds that the
geometric genus \index{geometric genus}
can be calculated once the Poincar\'e series is known.
In particular, if one finds a formula for the Poincar\'e series given
in terms of the link $M$, one automatically obtains a topological
identification of $p_g$. Although this is indeed impossible in general,
there are
certain cases where the Poincar\'e series, or just $p_g$, can be described
by topological invariants. As an example, we have the following result:
\end{block}

\begin{prop}[N\'emethi \cite{Nem_coh_splice,Nem_Poinc}]\label{prop:SWIC_known}
Let $(X,0)$ be a
\emph{splice quotient} \index{splice!quotient}
singularity \cite{NeumWa_SpQ}.
Then $P(t) = Z_0(t) \in \Z[[t^L]]$.
\end{prop}

\begin{rem}
Rational,\index{rational}
minimally elliptic \index{elliptic!minimally}
and
weighted homogeneous \index{weighted homogeneous}
are examples of splice quotient singularities \cite{Okuma_certain,Neu_AbCovQH}.
\end{rem}

\subsection{The Seiberg--Witten invariants}

We will now discuss the
\emph{Seiberg--Witten invariants} \index{Seiberg--Witten invariant}
$\sw_M(\sigma)\in\Q$
associated with any three dimensional manifold $M$ with a
$\spinc$ structure \index{spinc structure@$\spinc$ structure}
$\sigma$.
The definition of these numbers is quite involved and we will only touch
the surface of the theory here. For details, see \cite{Lim} and references
therein. There are, however various identifications of the
Seiberg--Witten invariants. In \cite{MengTaubes}, Meng and Taubes proved that
in the case $H_1(M,\Q) \neq 0$, the Seiberg--Witten invariants are equivalent
to Milnor torsion. Nicolaescu then proved \cite{Nico_SW_tor} that in the
case of a
rational homology sphere, \index{homology sphere!rational}
the Seiberg--Witten invariants are given
by the
Casson--Walker invariant \index{Casson--Walker invariant}
and
Reidemeister--Turaev torsion. \index{Reidemeister--Turaev torsion}
In this case, $\sw_M(\sigma)$ is also given as the
normalized Euler characteristic \index{normalized Euler characteristic}
of either Ozsv\'ath and Szab\'o's
Heegaard--Floer homology \index{Heegaard--Floer homology}
\cite{Rustamov},
or N\'emethi's
lattice homology \index{lattice homology}
associated with $\sigma$,
see \cref{ss:lattice}.

As in \ref{ss:res}, we use the notation $H = H_1(M,\Z)$.

\begin{block}
We start with a short review of
$\spinc$ structures.
For more details, see e.g. \cite{Nico_notes, Nem_Nico_SW}.
For each $n\geq 0$
we have the group $\Spinc(n)$, along with a $\U(1)$ bundle
$\Spinc(n) \to \SO(n)$. This is (for $n\geq 0$) the $\U(1)$ bundle
corresponding to the nontrivial element in $H^2(\SO(n),\Z) = \Z/2\Z$.
Let $X$ be a CW complex, and
let $E\to X$ be an oriented real vector bundle of rank $n$ obtained via a map
$\rho:X\to \BSO(n)$. A
\emph{$\spinc$ structure} \index{spinc structure@$\spinc$ structure}
on $E$ is a lifting
$X\to \BSpinc(n)$ of $\rho$. Since $\ker(\Spinc(n) \to \SO(n)) = \U(1)$, the
difference
of two $\spinc$ structures is a $\U(1)$ bundle, which is zero if and only the
two structures coincide. The set
$\Spinc(E)$
\nomenclature[SpincE]{$\Spinc(E)$}{Set of $\spinc$ structures on the vector
bundle $E$}
of $\spinc$ structures on
$E$ is therefore a torsor over $H^2(X,\Z) = [X,\BU(1)]$, unless it is empty.
A $\spinc$ structure on
an oriented
manifold $M$ is by definition a $\spinc$ structure on
its tangent bundle, their set is denoted by
$\Spinc(M)$.
\nomenclature[SpincM]{$\Spinc(M)$}{Set of $\spinc$ structures on
the manifold $M$}
Denote the action by
$H^2(X,\Z)\times\Spinc(E) \ni (h,\sigma) \mapsto h\sigma \in \Spinc(E)$.

The map $\U(n) \to \SO(2n)$ factors through $\Spinc(2n)$. A complex structure
on a vector bundle of even rank therefore induces a $\spinc$ structure.
In particular, if $E$ has a complex structure, then $\Spinc(E) \neq \emptyset$.

Now, assume that $M$ is the boundary of a complex surface $\X$. We want to
construct the
\emph{canonical $\spinc$ structure}
\index{spinc structure@$\spinc$ structure!canonical}
$\scan\in\Spin(M)$
\nomenclature[sc]{$\scan$}{Canonical $\spinc$ structure on link}
on $M$.
Note first that since $\X$ is a complex manifold, its tangent bundle has
a complex structure which induces $\bscan \in \Spinc(\X)$.
Now, the tangent bundle of $\X$ splits on $M$ as $T\X|_M = \R\oplus TM$,
where we denote simply by $\R$ the trivial line bundle. Here, the first
summand is generated by an outwards pointing vector field. This yields
a lift $M\to\BSO(3)$ of the structure map defining $T\X|_M$, and this map
defines the tangent bundle of $M$. We therefore have lifts
of $M\to\BSO(4)$ to $\BSO(3)$ as well as $\BSpinc(4)$. Since
$\Spinc(3) = \Spinc(4)\times_{\SO(4)}\SO(3)$, this defines a lift
$M\to\BSpinc(3)$ of the structure map of the tangent space $TM$.
This is the canonical $\spinc$ structure $\scan$ on $M$.
By the above statements,
we get a bijection $H = H^2(M,\Z) \leftrightarrow \Spinc(M)$
given by $h \leftrightarrow h\scan$.
\end{block}

\begin{block}
Let $M$ be a compact oriented three dimensional differentiable manifold and
choose a
$\spinc$ structure \index{spinc structure@$\spinc$ structure}
$\sigma$ on $M$. We will assume throughout that the first
Betti number of $M$ is zero, that is, $H_1(M,\Q) = 0$.
Choose a Riemannian metric
$g$ and a closed two form $\eta$ on $M$. Assuming that $g$ and $\eta$ are
chosen sufficiently generic, one obtains a space of
\emph{monopoles}, \index{monopoles}
whose signed count we call the
\emph{unnormalized Seiberg--Witten invariant}
\index{Seiberg--Witten invariant!unnormalized}
and denote by
$\swun_M(\sigma,g,\eta)$.
This number depends on the choice of $g$ and $\eta$. 
The \emph{Kreck--Stolz invariant}
$\KS_M(\sigma,g,\eta)\in\Q$ \nomenclature[KS]{$\KS_M(\sigma,g,\eta)$}
{Kreck--Stolz invariant}
is another number defined by this data
\cite{Lim,Nico_SW_tor}. The
\emph{normalized Seiberg--Witten invariants}
\index{Seiberg--Witten invariant!normalized}
$\sw_M(\sigma)$ \nomenclature[sw]{$\sw_M(\sigma)$}{Seiberg--Witten invariant}
are defined as follows:
\end{block}

\begin{prop}[Lim \cite{Lim}]
The number 
\[
  \sw_M(\sigma) = \swun_M(\sigma,g,\eta) +\KS_M(\sigma,g,\eta)
\]
is independent of the choice of $g$ and $\eta$. \qed
\end{prop}

\begin{rem}
Lim also obtained results in the case when the first Betti number is greater
or equal to $1$. We will not discuss these results here, since our results
concern rational homology spheres only.
\end{rem}

\begin{block}
Let $M$ be a
rational homology sphere \index{homology sphere!rational}
with a $\spinc$ structure $\sigma$.
Denote by
$\lambda(M)$ \nomenclature[l(M)]{$\lambda(M)$}
{Casson--Walker--Lescop invariant of $M$}
the
\emph{Casson--Walker--Lescop invariant}\index{Casson--Walker--Lescop invariant}
of $M$, normalized as in \cite{Lescop}. Denote by
\[
  \T_{M,\sigma} = \sum_{h\in H} \T_{M,\sigma}(h) h \in\Q[H]
\] \nomenclature[Tms]{$\T_{M,\sigma}$}{Reidemeister--Turaev torsion}
the
\emph{Reidemeister--Turaev torsion} \index{Reidemeister--Turaev torsion}
defined in
\cite{Turaev_MRL,Turaev_book}.
The \emph{normalized} (or \emph{modified})
\emph{Reidemeister--Turaev torsion}
\index{Reidemeister--Turaev torsion!normalized}
is defined as
\[
  \T^0_{M,\sigma} = \sum_{h\in H}
      \left(\T_{M,\sigma}(h) - \frac{\lambda(M)}{|H|}\right)h \in\Q[H].
\]
\nomenclature[T0ms]{$\T^0_{M,\sigma}$}{Normalized Reidemeister--Turaev torsion}
These invariants are discussed in \cite{Nem_Nico_SW}.
\end{block}

\begin{rem} \label{rem:CWL}
The
\emph{Casson}, \index{Casson invariant}
\emph{Casson--Walker} \index{Casson--Walker invariant}
and
\emph{Casson--Walker--Lescop} \index{Casson--Walker--Lescop invariant}
invariants are successive generalizations. Casson introduced an integral
invariant
$\lambda_C(M)$ \nomenclature[lC(M)]{$\lambda_C(M)$}{Casson invariant of $M$}
for $M$ an integral homology sphere. For $M$ a
rational homology sphere, Walker defined
$\lambda_{CW}(M)$ \nomenclature[lCW(M)]
{$\lambda_{CW}(M)$}{Casson--Walker invariant of $M$}
satisfying
$\lambda_{CW}(M) = 2\lambda_C(M)$ if $M$ is an integral homology sphere.
In \cite{Lescop}, Lescop defined an invariant
$\lambda_{CWL}(M)$ \nomenclature[lCWL(M)]{$\lambda_{CWL}(M)$}
{Casson--Walker--Lescop invariant of $M$}
for any
closed oriented three dimensional manifold, satisfying
$\lambda_{CWL}(M) = \frac{|H_1(M,\Z)|}{2} \lambda_{CW}(M)$ whenever
$M$ is a rational homology sphere. We will follow the notation of
Lescop, that is, $\lambda = \lambda_{CWL}$.
\end{rem}

\begin{prop}[Nicolaescu \cite{Nico_SW_tor}]
Let $M$ be a
rational homology sphere \index{homology sphere!rational}
with a
$\spinc$ structure \index{spinc structure@$\spinc$ structure}
$\sigma$ and
set $\SW_{M,\sigma} = \sum_{h\in H} \sw(M,h\sigma) h \in \Q[H]$. Then
$\SW_{M,\sigma} = \T^0_{M,\sigma}$.
\end{prop}

\begin{rem}
Since we will only deal with
rational homology spheres, \index{homology sphere!rational}
we do not state the
corresponding statements in \cite{Nico_SW_tor} about three dimensional
manifolds with
nontrivial rational first homology.
\end{rem}

We will now describe the identification of the
normalized Seiberg--Witten invariants
\index{Seiberg--Witten invariant!normalized}
which we will use to prove the main theorem in \cref{s:SW}. 
Recall that the coefficients $q_{l'}$ were introduced in
\cref{def:zeta}.

\begin{prop}[N\'emethi \cite{Nem_SW,thebook}] \label{prop:Q_SW}
Assume the notation in \cref{ss:res} and \cref{ss:Lipman} and that $M$ is
a rational homology sphere. Take any $l'\in L'$ satisfying
$(l',E_v) \leq (Z_K,E_v)$ for all $v\in \V$. Then
\begin{equation} \label{eq:prop:Q_SW}
  q_{l'}
    = \sw_M([l']\scan) - \frac{(-Z_K + 2l')^2 + |\V|}{8}.
\end{equation}
\end{prop}

\begin{rem} \label{rem:SW}
\begin{list}{(\thedummy)}
{
  \usecounter{dummy}
  \setlength{\leftmargin}{0pt}
  \setlength{\itemsep}{0pt}
  \setlength{\itemindent}{4.5pt}
}

\item
In \cite{Laszlo_th}, L\'aszl\'o develops a general theory of multivariable
power series \index{power series}
and defines a notion of a
periodic constant \index{periodic constant}
(see also \cref{ss:pol_part}).
In his language, 
\cref{eq:prop:Q_SW} means that for $h\in H$, the periodic constant of $Z_h(t)$
is the number
\[
  \sw_M(h\scan) - \frac{(-Z_K + 2r_h)^2 + |\V|}{8},
\]
where $r_h$ is the unique element in $L'$ with $[r_h] = h$ and
$0\leq m_v(r_h) < 1$ for all $v\in \V$.

\item \label{it:SW_MI}
Although neither side of \cref{eq:prop:Q_SW} is generally easy to compute,
it shows that the normalized Seiberg--Witten invariant behaves with
respect to the zeta function as the geometric genus does with respect
to the Poincar\'e series, see \cref{prop:H_from_P,prop:pg_hl}.
In particular, N\'emethi's main identity $Z_0 = P$ \cite{Nem_Poinc} would
imply
the Seiberg--Witten invariant conjecture which is discussed in \cref{ss:SWIC}.

\item
In \cref{s:SW}, the left hand side of \cref{eq:prop:Q_SW} is calculated
in terms of a Newton diagram, given some nondegeneracy conditions (these
are defined in \cref{s:newton_diag}).

\end{list}
\end{rem}

\subsection{The Seiberg--Witten invariant conjecture} \label{ss:SWIC}

In this subsection we give a very brief account of the
Seiberg--Witten invariant conjecture
\index{Seiberg--Witten invariant conjecture}
of N\'emethi and Nicolaescu.

\begin{block}
In \cite{Nem_Nico_SW}, N\'emethi and Nicolaescu conjectured a topological
upper bound on the
geometric genus \index{geometric genus}
of a normal surface singularity,
whose link is a
rational homology sphere \index{homology sphere!rational}
in terms of the
normalized Seiberg--Witten invariant
\index{Seiberg--Witten invariant!normalized}
of the
link, \index{link}
and the
resolution graph. \index{resolution graph}
More precisely, the
\emph{Seiberg--Witten invariant conjecture (SWIC)}
\nomenclature[SWIC]{SWIC}{Seiberg--Witten Invariant Conjecture}
says that
\begin{equation} \label{eq:SWIC}
  \sw_{M}(\scan) - \frac{Z_K^2+|\V|}{8} \geq p_g,
\end{equation}
with equality if the singularity is
$\Q$-Gorenstein \index{Gorenstein!$\Q$}
(in particular,
Gorenstein). \index{Gorenstein}
If the singularity is a
complete intersection \index{complete intersection}
and the link is an
integral homology sphere, \index{homology sphere!integral}
then the
conjecture is equivalent with the
\emph{Casson invariant conjecture (CIC)} \index{Casson invariant conjecture}
\nomenclature[CIC]{CIC}{Casson Invariant Conjecture}
of Neumann and Wahl \cite{Neu_Wa_Ca}.
Although counterexamples have been found to the SWIC (see below), it is
still an interesting question to ask, under which conditions does
the SWIC hold?
\end{block}

\begin{example}
\begin{list}{(\thedummy)}
{
  \usecounter{dummy}
  \setlength{\leftmargin}{0pt}
  \setlength{\itemsep}{0pt}
  \setlength{\itemindent}{4.5pt}
}

\item
Neumann and Wahl proved the
CIC \index{Casson invariant conjecture}
for
weighted homogeneous \index{weighted homogeneous}
singularities
(equivalently, singularities of \emph{Brieskorn--Pham} type
\cite{Neu_AbCovQH}),
suspensions \index{suspension}
of plane curves
and certain
complete intersections \index{complete intersection}
in $\C^4$ \cite{Neu_Wa_Ca}.
They also note that in the case of Brieskorn--Pham
hypersurface \index{hypersurface}
singularities, that is, singularities given by an equation of
the form $x^p+y^q+z^r = 0$, the conjecture follows from work of
Fintushel and Stern \cite{Fin_Ste}.

The existence of complete intersection singularities with integral
homology sphere link, other than the ones listed above, is an interesting
open problem.

\item
N\'emethi and Nicolaescu proved the SWIC
for certain
rational \index{rational}
and
minimally elliptic \index{elliptic!minimally}
singularities
 \cite{Nem_Nico_SW}, for singularities with a
good $\C^*$ action \index{good $\C^*$ action}
\cite{Nem_Nico_SWII} and for
suspensions \index{suspension}
of irreducible
plane curve \index{plane curve}
singularities \cite{Nem_Nico_SWIII}.

\item
N\'emethi and Okuma proved the
CIC \index{Casson invariant conjecture}
for singularities of
splice \index{splice}
type
\cite{Nem_Ok_CIC}, as well as the SWIC for
splice quotients \index{splice!quotient}
\cite{Nem_Ok_SWIC}
(see \cite{NeumWa_Sp,NeumWa_SpQ} for definitions).

\item
Using
superisolated \index{superisolated}
singularities, Luengo Velasco, Melle Hern\'andez and
N\'emethi constructed counterexamples to the
SWIC \index{Seiberg--Witten invariant conjecture}
\cite{LVMHN_super}.
More precisely,
they constructed
hypersurface \index{hypersurface}
singularities (in particular,
Gorenstein) \index{Gorenstein}
for which \cref{eq:SWIC} does not hold.

\item
In \cite{Nem_Bal}, N\'emethi and the author formulated a different
topological characterization for the geometric genus which was
proved for superisolated singularities and Newton nondegenerate
singularities. In the latter case, the SWIC is proved
in \cref{s:SW}.

\end{list}
\end{example}

\subsection{Computation sequences}

In this section we will discuss
computation sequences \index{computation sequence}
and an upper bound
on the
geometric genus \index{geometric genus}
obtained by such sequences.

In \cite{Lauf_rat}, Laufer constructs a computation sequence to the
minimal cycle, \index{minimal cycle}
yielding a relatively simple procedure for its computation.
Furthermore, a numerical condition is given for
rationality, \index{rational}
i.e. whether $p_g = 0$.
In \cite{Lauf_minell}, a similar computation sequence
plays an essential role in
the topological characterisation of
minimally elliptic \index{elliptic!minimally}
singularities.
This idea was generalized by Yau in \cite{Yau_two} and by N\'emethi
in \cite{Nem_WE} for more general
elliptic \index{elliptic} singularities, where a bound of the type
\cref{eq:comp_seq} is given and equality characterized
by triviality of certain line bundles.

One of the main results
in \cite{Nem_Bal} is the existence of a computation sequence to the
anticanonical cycle \index{anticanonical cycle}
obtained
directly from the
resolution graph \index{resolution graph}
yielding equality in \cref{eq:comp_seq}
in the case of
Newton nondegenerate \index{Newton nondegenerate}
singularities,
thus giving a topological identification of the
geometric genus. \index{geometric genus}
This result, as well as some improvements, is described in
\cref{s:seq,s:pg}.

\begin{definition} \label{def:comp_seq}
Assume given a
resolution graph \index{resolution graph}
$G$ for a singularity $(X,0)$.
Let $Z \in L$ be an effective cycle. A
\emph{computation sequence} \index{computation sequence}
for $Z$
is a sequence $Z_0, \ldots, Z_k$ so that $Z_0 = 0$, $Z_k = Z$ and for
each $i$ we have a $v(i) \in \V$ so that $Z_{i+1} = Z_i + E_{v(i)}$.
Given such a computation sequence, its
\emph{continuation to infinity} \index{continuation to infinity}
is the sequence $(Z_i)_{i=0}^\infty$ recursively defined by
$Z_{i+1} = Z_i + E_{v(i)}$ where we extend $v$ to $\N$ by
$v(i') = v(i)$ if $i' \equiv i \, (\mod k)$.
\end{definition}

The main elements of the  statement and proof of the following theorem
are already present in \cite{Lauf_minell,Nem_WE}.
A different formulation can also be found in
\cite{Lattice_coh} Proposition 6.2.2.

\begin{thm} \label{thm:comp_seq}
Assume that each component $E_v$ of the exceptional divisor is a rational
curve.
Let $Z\in L$ be an effective cycle and $(Z_i)_{i=0}^k$ a
computation sequence \index{computation sequence}
for $Z$. Then
\begin{equation} \label{eq:comp_seq}
  h_Z \leq \sum_{i=0}^{k-1} \max\{ 0, (-Z_i,E_{v(i)})+1 \}
\end{equation}
and we have an equality if and only if the natural maps
$H^0(\X,\O_{\X}(-Z_i)) \to H^0(E_{v(i)},\O_{E_{v(i)}}(-Z_i))$
are surjective for all $i$.
\end{thm}

\begin{rem} \label{rem:comp_seq_pg}
If $G$ is
numerically Gorenstein, \index{Gorenstein!numerically}
then we have $h_{Z_K} = p_g$ by
\cref{prop:pg_hl}. Therefore, \cref{eq:comp_seq} gives a topological
bound on the
geometric genus \index{geometric genus}
in this case.
\end{rem}

\begin{proof}[Proof of \cref{thm:comp_seq}]
For any $i$, we have a short exact sequence
\[
\xymatrix{
  0 \ar[r] & \O_{\X}(-Z_{i+1}) \ar[r] & \O_{\X}(-Z_i) \ar[r] &
             \O_{E_{v(i)}}(-Z_i) \ar[r] & 0
}
\]
which yields the long exact sequence
\[
\xymatrix{
  0 \ar[r] & H^0(\X,\O_{\X}(-Z_{i+1})) \ar[r] & H^0(\X,\O_{\X}(-Z_i)) \ar[r] &
             H^0(\X,{\O_{E_{v(i)}}}(-Z_i)) 
    \ar`r/8pt[rd]`[lll]`[dlll]`/8pt[dll] [dll]\\
           & H^1(\X,\O_{\X}(-Z_{i+1})) \ar[r] & H^1(\X,\O_{\X}(-Z_i)) \ar[r] &
             H^1(\X,\O_{E_{v(i)}}(-Z_i)) \ar[r] & 0.
}
\]
Denote by $\beta_i$ the connection homomorphism of this sequence. We get
\[
\begin{split}
  h_Z =& \dim_\C \frac{H^0(\X,\O_{\X})}{H^0(\X,\O_{\X}(-Z))}\\
      =& \sum_{i=0}^{k-1}
	      \dim_\C \frac{H^0(\X,\O_{\X}(-Z_i)}{H^0(\X,\O_{\X}(-Z_{i+1}))}\\
      =& \sum_{i=0}^{k-1}
	      \dim_\C H^0(\X,\O_{E_{v(i)}}(-Z_i)) - \rk \beta_i.
\end{split}
\]
The statement now follows, since, on one hand,
$E_{v(i)}\cong\CP^1$ and the degree of the line bundle
$\O_{E_{v(i)}}(-Z_i)$ is $(-Z_i,E_{v(i)})$, and on the other hand,
the surjectivity condition is equivalent to $\rk \beta_i = 0$.
\end{proof}

\begin{rem} \label{rem:subseq}
Assume that for some $i$ we have $(Z_i,E_{v(i)}) > 0$. Then the group
$H^0(E_{v(i)},\O_{E_{v(i)}}(-Z_i))$ vanishes and the surjectivity condition
in \cref{thm:comp_seq} holds automatically. Furthermore,
the $i$\textsuperscript{th} summand in \cref{eq:comp_seq} vanishes.
Assume given a
subsequence \index{computation sequence!subsequence}
$i_1,\ldots, i_s$ of $0,\ldots, k-1$ so that if
$0\leq i\leq k-1$ and $i\neq i_r$ for all $r$, then $(Z_i,E_{v(i)}) > 0$.
Then \cref{thm:comp_seq} can be phrased entirely in terms of this subsequence,
that is, the sum on the right hand side of \cref{eq:comp_seq} can be taken
over the subsequence $(i_r)$ only, and the surjectivity condition is only
needed for the $i_r$\textsuperscript{th} terms.
\end{rem}

\subsection{Lattice cohomology and path lattice cohomology} \label{ss:lattice}

In \cite{Lattice_coh}, N\'emethi introduced
\emph{lattice cohomology} \index{lattice cohomology}
as well as the related
\emph{path lattice cohomology}. \index{lattice cohomology!path}
In general, lattice cohomology is associated with any
$\spinc$ structure \index{spinc structure@$\spinc$ structure}
on the link and the results given here have generalizations to this setting.
For simplicity, we will assume that $G$ is the
resolution graph \index{resolution graph}
of a
numerically Gorenstein \index{Gorenstein!numerically}
singularity $(X,0)$,
and we will only consider invariants
associated with the canonical $\spinc$ structure
$\scan$. \index{spinc structure@$\spinc$ structure!canonical}

In \cref{prop:eu_SW}, a connection between lattice cohomology
and the Seiberg--Witten invariants is obtained.
\Cref{prop:eu_path}, on the other hand, gives a connection with computation
sequences.

\begin{block}
Let $L = \Z\gen{E_v}{v\in\V}$ be the lattice associated with a resolution graph
$G$ as in \cref{ss:res}. We give $L_\R = L\otimes \R$ the structure of a
CW complex by taking as cells the cubes
$\square_{l,I} = \set{l+\sum_{v\in I} t_v E_v}{\fa{v\in I}{0<t_v<1}}$,
\nomenclature{$\square_{l,I}$}{Cube}
where $l\in L$ and $I \subset \V$.
Let
$\mathcal{Q}$ \nomenclature[Q]{$\mathcal{Q}$}{Set of cubes}
be the set of these cubes.
For $l\in L$ we set
$\chi(l) = (-l,l-Z_K)/2$. 
\nomenclature[chl]{$\chi(l)$}{Analytic euler characteristic extended to $L$}
Note that if $l$ is an effective
cycle, then $\chi(l)$ is the Euler characteristic of the
structure sheaf of the scheme defined by the ideal sheaf $\O_{\X}(-l)$.
The
\emph{weight function} \index{weight function}
$w$ \nomenclature[w]{$w$}{Weight function}
is defined on $\mathcal{Q}$ by setting
\[
  w(\square_{l,I}) = \max\set{\chi(l+\sum_{v\in I'} E_v)}{I'\subset I}.
\]
In this way, if $n\in\Z$, then
the set $S_n = \cup\set{\square_{l,I}}{w(\square_{l,I}) \leq n}$
is a subcomplex of $L_\R$.
Note that, by negative definiteness, $\chi$ is bounded from below on $L$
and the subcomplexes $S_n$ are finite.
\end{block}

\begin{definition}
Let $A \subset L_\R$ be a subcomplex. The $q^{\mathrm{th}}$
\emph{lattice cohomology} \index{lattice cohomology}
of the pair $(A,w)$ is defined as
\[
  \mathbb{H}^q(A,w) = \bigoplus_{n\in\Z} H^q(A\cap S_n,\Z),
\] 
for $q\geq 0$
We also set
$\mathbb{H}^*(A,w) = \oplus_{q\geq 0} \mathbb{H}^q(A,w)$.
\nomenclature[H*Aw]{$\mathbb{H}^*(A,w)$}{Lattice cohomology}
For any $q$ and $n$, the inclusion $A\cap S_n \subset A\cap S_{n+1}$
induces a map on cohomology which we denote by $U$. This
gives $\mathbb{H}^*(A,w)$ the structure of a $\Z[U]$ module.
Similarly, we get \emph{reduced lattice cohomology}
$\mathbb{H}^*_{\mathrm{red}}(A,w)$ by replacing cohomology
\nomenclature[H*redAw]{$\mathbb{H}^*_{\mathrm{red}}(A,w)$}
{Reduced lattice cohomology}
$H^*$ by reduced cohomology $\tilde H^*$.
\end{definition}

\begin{definition}
Let $G$ be a plumbing graph representing the link of a surface singularity
with a canonical cycle $K$.
The associated \emph{lattice cohomology} is
$\mathbb{H}^*(G,K) = \mathbb{H}^*(L_\R,w)$.

Assume that $G$ is numerically Gorenstein (i.e. $Z_K\in L$) and
let $(Z_i)_{i=0}^k$ be a computation sequence to $Z_K$. Let $\gamma$
be the subcomplex of $L_\R$ consisting of the zero dimensional cubes
$\{Z_i\}$ for $0\leq i\leq k$  and one dimensional cubes with
vertices $Z_i, Z_{i+1}$ for $0\leq i < k$.
We say that $\gamma$ is the \emph{path} associated to the computation sequence
$(Z_i)_{i=0}^k$.
We define the
\emph{path lattice cohomology} associated with the computation sequence
$(Z_i)_{i=0}^k$ as $\mathbb{H}^*(\gamma,w)$.
\end{definition}

\begin{rem}
In general, N\'emethi defines $\mathbb{H}^*(G,k)$ for any
\emph{characteristic element} $k \in L'$, an example of which is
the canonical cycle $K$. We restrict ourselves to this element since
none of our result concern characteristic elements other than $K$.
\end{rem}

For $l_1\leq l_2$ define the rectangle
\[
  R(l_1,l_2)
  = \bigcup \set{\square_{l,I}}{l_1\leq l \leq l+\sum_{v\in I}E_v \leq l_2}.
\]
\begin{prop}[N\'emethi \cite{Lattice_coh}] \label{prop:lch}
The inclusion
$R(0,Z_K) \cap S_n \subset S_n$ \index{anticanonical cycle}
is a homotopy equivalence
for all $n$. Furthermore, the complex $S_n$ is contractible if $n>0$.
\end{prop}

\begin{cor}
The group $\mathbb{H}_{\mathrm{red}}^*(L_\R,w)$ is finitely generated.
\end{cor}

\begin{definition}
Set $m = \min \chi$ and assume that $\mathbb{H}_{\mathrm{red}}^*(A,w)$
has finite rank. The
\emph{normalized Euler characteristic} \index{normalized Euler characteristic}
of lattice cohomology is defined as
\[
\begin{split}
  \eu(\mathbb{H}^*(A,w))
    &= -m + \sum_{q=0}^\infty (-1)^q \rk \mathbb{H}^q_{\mathrm{red}} \\
  \eu(\mathbb{H}^0(A,w))
    &= -m + \rk \mathbb{H}^0_{\mathrm{red}}
\end{split}
\]
\nomenclature[eu]{$\eu(\mathbb{H}^*(A,w))$}{Normalized Euler characteristic}
\end{definition}

\begin{prop}[N\'emethi \cite{Nem_SW}] \label{prop:eu_SW}
We have
\[
  \eu(\mathbb{H}^*(L_\R,w)) = \sw_M(\scan) - \frac{Z_K^2+|\V|}{8}.
\]
\qed
\end{prop}

\begin{prop}[N\'emethi \cite{Lattice_coh}] \label{prop:eu_path}
Assume that $G$ is numerically Gorenstein.
Let $(Z_i)_{i=0}^k$ be a
computation sequence \index{computation sequence}
for $Z_K\in L$ and let
$\gamma$ be the associated path. Then
\[
    \eu(\mathbb{H}^*(\gamma,w))
	=    \sum_{i=0}^{k-1} \max\{0,(-Z_i,E_{v(i)})+1\}.
\]
\end{prop}

Combining this result with \cref{thm:comp_seq,rem:comp_seq_pg} we have
the following
\begin{cor}
Assume that $G$ is numerically Gorenstein.
We have $p_g \leq \min_\gamma \eu(\mathbb{H}^*(\gamma,w))$, where $\gamma$
runs through complexes associated to any
computation sequence \index{computation sequence}
to
$Z_K$ \index{anticanonical cycle}
as in the proposition above.
\qed
\end{cor}

\subsection{On power series in one variable} \label{ss:pol_part}

In this subsection we recall some facts about
power series \index{power series}
in one variable
and define the
polynomial part \index{polynomial part}
of the power series expansion of a rational
function.
Here, as well as in the sequel, we will identify a rational function
with its Taylor expansion at the origin.
In particular, we will identify the localization
$\C[t]_{(t)}$ with a subring of the ring of power series $\C[[t]]$.
Furthermore, we will generalize these definitions
to rational Puiseux series and prove a formula for these invariants
for special series constructed from simplicial cones.

N\'emethi and Okuma introduced the
periodic constant \index{periodic constant}
of a rational function \cite{Nem_Ok_CIC,Okuma_pgspq}.
Braun and N\'emethi used the polynomial part of a rational function
in their work
\cite{SW_surg}. See also \cite{Laszlo_th} for a discussion
and generalization of these invariants.

Recall that a
\emph{quasipolynomial} \index{quasipolynomial}
is a function $\Z\to\C$ of the form
$t \to \sum_{i=0}^d c_i(t) t^i$, where $c_i:\Z\to\C$ are
periodic functions. \index{periodic function}

\begin{prop}
Let $P \in \C[t]_{(t)}$ be a rational function, regular at the origin, and
consider its expansion at the origin
$P(t) = \sum_{i=0}^\infty a_i t^i$ with $a_i \in \C$.
Then there exists a
quasipolynomial \index{quasipolynomial}
function $i \mapsto a'_i$ so that
for $i$ large enough we have $a_i = a'_i$.
\qed
\end{prop}

\begin{definition} \label{def:pol_part}
Let $P$, $a_i$ and $a'_i$ be as in the proposition above. The
\emph{negative part} \index{negative part}
of $P$ is
$P^{\mathrm{neg}}(t) = \sum_{i=0}^\infty a'_i t^i$.
\nomenclature[Pneg]{$P^{\mathrm{neg}}(t)$}{Negative part of $P(t)$}
The
\emph{polynomial part} \index{polynomial part}
of $P$ is
$P^{\mathrm{pol}}(t) = P(t) - P^{\mathrm{neg}}(t)$.
\nomenclature[Ppol]{$P^{\mathrm{pol}}(t)$}{Polynomial part of $P(t)$}
The
\emph{periodic constant} \index{periodic constant}
of $P(t)$ is the number
\nomenclature[pc]{$\pc P(t)$}{Periodic constant of $P(t)$}
$\pc P(t) = P^{\mathrm{pol}}(1)$.
\end{definition}

\begin{lemma} \label{lem:pol_part_add}
The
polynomial part \index{polynomial part}
is additive. More precisely, if $P,Q\in\C[t]_{(t)}$,
then $(P+Q)^{\mathrm{pol}} = P^{\mathrm{pol}} + Q^{\mathrm{pol}}$.
\end{lemma}
\begin{proof}
It is clear from definition that 
$(P+Q)^{\mathrm{neg}} = P^{\mathrm{neg}} + Q^{\mathrm{neg}}$.
The lemma follows.
\end{proof}

\begin{rem}
\begin{list}{(\thedummy)}
{
  \usecounter{dummy}
  \setlength{\leftmargin}{0pt}
  \setlength{\itemsep}{0pt}
  \setlength{\itemindent}{4.5pt}
}
\item
We may write $P(t) = p(t)/q(t)$ with $p(t), q(t) \in \C[t]$ and
$\gcd(p(t),q(t)) = 1$. Using the
Euclidean algorithm, we can write $p(t) = h(t) q(t) + r(t)$ with
$h(t), r(t) \in \C[t]$ and $\deg r(t) < \deg q(t)$ and furthermore, this
presentation is unique. It is a simple exercise to show that
$P^{\mathrm{neg}}(t) = r(t)/q(t)$ and $P^{\mathrm{pol}}(t) = h(t)$. In fact,
$P(t) = P^{\mathrm{neg}}(t) + P^{\mathrm{pol}}(t)$ is the unique presentation
of $P(t)$ as a sum of a polynomial and a fraction of negative degree.

\item
One finds easily that $\C[t]_{(t)} = \C[t] \oplus N$ where
\[
  N = \set{p\in\C[t]_{(t)}}{\lim_{t\to\infty} p(t) = 0},
\]
and that the
polynomial \index{polynomial part}
and
negative parts \index{negative part}
are the projections to these summands.
The additivity property \cref{lem:pol_part_add} follows immediately from
this observation.

\end{list}
\end{rem}

\begin{block}
Denote by $\C[[t^{1/\infty}]] = \cup_{n\in\Z_{>0}} \C[[t^{1/n}]]$ the
\emph{ring of Puiseux series}. \index{Puiseux series}
Thus, for any Puiseux series $P(t)$, there is
an $N>0$ so that $P'(t) = P(t^N) \in \C[[t]]$. We will say that $P(t)$ is
\emph{rational} \index{Puiseux series!rational}
if $P'(t)$ is rational for such a choice of $N$.
The statements and
definition above apply to this situation without much alteration.
In particular, if $P(t) \in \C[[t^{1/\infty}]]$ is rational, and $N$ is as
above, then we set
$P^{\mathrm{pol}}(t) = P'^{\mathrm{pol}}(t^{1/N})$ and
$P^{\mathrm{neg}}(t) = P'^{\mathrm{neg}}(t^{1/N})$.
Since $P'^{\mathrm{pol}}(t)$ is a polynomial, we find that
$P^{\mathrm{pol}}(t)$ is a finite
expression, that is, $P^{\mathrm{pol}}(t)$ is a
\emph{Puiseux polynomial}, \index{Puiseux polynomial}
$P^{\mathrm{pol}}(t) \in \C[t^{1/\infty}] = \cup_{n\in\Z_{>0}} \C[t^{1/n}]$.

Similarly as above, we have the field of
\emph{Laurent--Puiseux series} \index{Laurent--Puiseux series}
and the ring of
\emph{Laurent--Puiseux polynomials}
\index{Laurent--Puiseux polynomials}
\[
  \C((t^{1/\infty})) = \cup_{n\in\Z_{>0}} \C((t^{1/n})),\quad
  \C[t^{\pm1/\infty}] = \cup_{n\in\Z_{>0}} \C[t^{\pm1/n}].
\]
\end{block}

\begin{lemma} \label{lem:cone_pol_part}
Let $C\subset\R^{n+1}$ be a finitely generated rational strictly convex cone
of dimension $q$
and $\ell:\Z^{n+1}\to\Q$ a rational linear function so that
$\ker \ell \cap C = \{ 0 \}$. Define a series
\[
  P_C(t) = \sum_{p\in C\cap\Z^{n+1}} t^{\ell(p)} \in \C[[t^{1/\infty}]].
\]
Then, $P_C(t)$ is a
rational \index{Puiseux series!rational}
Puiseux series and $P_C^{\mathrm{pol}}(t) = 0$ and
\[
  ((1-t)P_C(t))^{\mathrm{pol}}
     = (-1)^{q-1}\sum_{p\in S\cap\Z^{n+1}} t^{1-\ell(p)}
\]
where $S = \set{p\in C^\circ\cap\Z^{n+1}}{\ell(p) \leq 1}$.
\index{polynomial part}
\end{lemma}
\begin{proof}
There is an $r\in\Q$ so that $\ell(\Z^{n+1}) = r\Z$. Then
$D = \ell^{-1}(r) \cap C$ is a rational polyhedron of dimension $q-1$,
denote by $D^\circ$ its
relative interior. Define $a:\Z\to \Z$ by setting
\[
  a(k) = 
\begin{cases}
  |kD\cap\Z^{n+1}|                 &  k\geq 0, \\
  (-1)^{q-1}|kD^\circ\cap\Z^{n+1}| &  k< 0.
\end{cases}
\]
By classical results, see e.g. \cite{Stan} and references therein,
the function $a$ is a quasipolynomial.
Furthermore, it is clear by construction that
$P_C(t) = \sum_{k=0}^\infty a(k) t^{kr}$, and so $P_C^{\mathrm{pol}}(t) = 0$.
We also get
\[
  ((1-t)P_C(t))^{\mathrm{pol}}
     = (-t P_C(t))^{\mathrm{pol}}
     =
     \left(
       -\sum_{0\leq k} a(k) t^{kr + 1}
     \right)^{\mathrm{pol}}
     =
        \sum_{-1\leq kr < 0} a(k) t^{kr + 1}
\]
which proves the lemma.
\end{proof}

\begin{example}
\begin{list}{(\thedummy)}
{
  \usecounter{dummy}
  \setlength{\leftmargin}{0pt}
  \setlength{\itemsep}{0pt}
  \setlength{\itemindent}{4.5pt}
}

\item
Let $A = \oplus_{i=0}^\infty A_k$ be the coordinate ring of an affine surface
$X \subset \C^N$ with a
good $\C^*$ action \index{good $\C^*$ action}
as in \cite{Pinkham} (with the
origin a fixed point) and
$P(t) = \sum_{i=0}^\infty \dim_\C A_i t^i$ the associated
Poincar\'e series. \index{Poincar\'e series}
Assume that the singularity at the origin has a rational homology sphere
link.
In \cite{Pinkham}, Pinkham shows that $P(t)$ and $p_g$
can be described in therms of the link at the origin.
Using the notation introduced in this subsection, his results yield
\index{geometric genus} \index{periodic constant}
$p_g = \pc P(t)$.

\item
Let $G$ be a negative definite graph as in \cref{ss:res} and
assume that $G$ satisfies the
\emph{semigroup condition} \index{semigroup condition}
and the
\emph{congruence condition} \index{congruence condition}
described in \cite{NeumWa_SpQ} (or, equivalently, the
end curve condition, \index{end curve condition}
see \cite{NeumWa_End,Okuma_End}). Neumann and Wahl \cite{NeumWa_SpQ}
constructed a
singularity (more precisely, a set of singularities forming an equisingular
deformation) whose topological
type is given by $G$. Such a singularity is called a
\emph{splice quotient singularity}. \index{splice!quotient}
If $v\in\V$ and $G_i$ are the components of the complement of $v$ in
$G$, then these subgraphs satisfy the same conditions. Okuma showed
\cite{Okuma_pgspq} that the
geometric genus \index{geometric genus}
of a
splice quotient singularity is the sum of the
geometric genera of splice quotient singularities with graphs $G_i$
plus an error term, which is the
periodic constant \index{periodic constant}
associated with a series
in one variable obtained from $G$ and $v$. More precisely, this series
has two descriptions. On one hand it is the
Poincar\'e series \index{Poincar\'e series}
associated
with the graded ring associated with the
divisorial filtration \index{divisorial filtration}
on the local ring
of the singularity given by the divisor $E_v$. On the other hand, it is
the function $Z_0^v(t_v)$, obtained from
the topological
zeta function \index{zeta function}
$Z_0(t)$ (see \cref{ss:Lipman})
by the restriction $t_w = 1$ for $w\neq v$.

\item
In \cite{SW_surg}, Braun and N\'emethi obtain a surgery formula for the
normalized Seiberg--Witten invariant
\index{Seiberg--Witten invariant!normalized}
associated with the graphs $G$, $G_i$ in
the previous example, but with no assumption on the graph other than
negative definiteness. In place of the
geometric genus, \index{geometric genus}
the formula contains
a normalized version of the
Seiberg--Witten invariant of the canonical
$\spinc$ structure \index{spinc structure@$\spinc$ structure!canonical}
on the associated three dimensional manifold. The error
term is, as in the previous example, the
periodic constant \index{periodic constant}
of $Z^v_0(t_v)$.

\end{list}
\end{example}

\subsection{The spectrum} \label{ss:spec}

In this section we will recall some facts about the
spectrum, \index{spectrum}
a numerical
invariant coming from Hodge theory. Its construction would require a
lengthier treatment than is possible here, so we only mention the main results
required. The most important fact we need about the spectrum is
\cref{prop:Saito_exp}, which allows us to calculate part of the spectrum
from the Newton diagram. In \cref{s:pg}, we will show how to recover this
part of the spectrum directly, given only the knowledge of the resolution
graph, as well as the divisor of the function $x_1 x_2 x_3$.
In our applications, we will always assume that $n=2$.

\begin{block}
We start with a very small account of the results leading to the mixed
Hodge structure on the cohomology groups of the
Milnor fiber. \index{Milnor fiber}
Mixed Hodge structures \index{mixed Hodge structure}
were introduced by Deligne in \cite{DelHII,DelHIII}
where he constructs a mixed Hodge structure on the cohomology groups
of arbitrary algebraic varieties, generalizing the
Hodge decomposition \index{Hodge decomposition}
on K\"ahler manifolds \cite{Hodge}. Previously, Griffiths, Schmid
\cite{GriffPerI,GriffPerII,GriffPerIII,Schmid} and others had studied
variations of Hodge structures arising from deformations of complex
manifolds, as well as the case of flat families, possibly with singular fibers.
For these,
a limit of the Hodge structures appears (in a suitable sense), but this
must be viewed as a mixed Hodge structure, rather than a pure Hodge
structure. In \cite{Steenb_lim}, Steenbrink considers the same problem from
a different viewpoint and constructs a mixed Hodge complex calculating
this limit. In \cite{Steenb_Oslo}, these results are combined with
others to construct a limit mixed Hodge structure on the cohomology groups
of the Milnor fiber of an isolated hypersurface singularity.
\end{block}

\begin{block}
Let $n\in\N$ and 
$f:(\C^{n+1},0) \to (\C,0)$ be a singular map germ defining an isolated
hypersurface \index{hypersurface}
singularity $(X,0)$ in $\C^{n+1}$. Assume that $Y\subset \C^{n+1}$
is a subset yielding a \emph{good representative} of the
Milnor fibration, \index{Milnor fibration}
where $D\subset \C$ is some small disc (see e.g. \cite{Looijenga}). 
Setting $D^* = D\setminus\{0\}$ and $Y^* = Y\setminus f^{-1}(0)$ we obtain
a locally trivial fiber bundle $Y^*\to D^*$ whose fiber is the Milnor fiber.
For a $t\in D^*$, denote by
$m_t:Y_t \to Y_t$
the
geometric monodromy, \index{geometric monodromy} \index{monodromy}
and
$T_t:H^n(Y_t,\C) \to H^n(Y_t,\C)$ the induced map on cohomology,
the
algebraic monodromy. \index{algebraic monodromy}
We can assume that together, these form a diffeomorphism $m^*:Y^*\to Y^*$.

Take $\tilde D^*$ as the universal covering space of the punctured disc,
and set $Y_\infty = Y^*\times_{D^*} \tilde D^*$. Concretely, we may take
$\tilde D^*$ as an upper half plane, with the covering map given by the
function $u\mapsto r e^{2\pi i u}$, where $r$ is the radius 
of $D^*$. We obtain canonical
maps $k:Y_\infty \to Y$ and $f_\infty:Y_\infty \to D^*$, as well as monodromy
transformations $m_\infty:Y_\infty\to Y_\infty$
and $\tilde m^*: \tilde D^* \to \tilde D^*$
satisfying
$f_\infty m_\infty = \tilde m^* f_\infty$ and $km_\infty = m^* k$.
This is summarized in the following diagram.
\begin{equation}
\begin{gathered}
\xymatrix
{
                                                                &
  Y^*                 \ar@{_{(}->}[d]   \ar@(dl,ul)^{m^*}       &
                                                                \\
  Y_\infty   \ar[r]^k \ar[d]_{f_\infty} \ar@(dl,ul)^{m_\infty}  &
  Y                   \ar[d]_f                                  &
  X          \ar@{_{(}->}[l]_i \ar[d]                           \\
  \tilde D^*\ar[r] \ar@(dl,ul)^{\tilde m^*_{\hphantom{\infty}}} &
  D                                                             &
  \{ 0\}     \ar@{_{(}->}[l]
}
\end{gathered}
\end{equation}
The space $\tilde D^*$ is a half plane in $\C$, and so, in particular,
it is contractible. Therefore, the space $Y_\infty \cong Y_t\times \tilde D^*$
has the same homotopy type as any Milnor fiber $Y_t$.  Furthermore, this
homotopy equivalence is determined uniquely modulo the monodromy.
\end{block}

\begin{prop}[Monodromy theorem \cite{Landman,Schmid}] \label{prop:monodromy}
\index{Monodromy theorem}
The eigenvalues of the
monodromy \index{monodromy}
operator
$T_\infty:H^n(Y_\infty,\C)\to H^n(Y_\infty,\C)$ are roots of unity, that is,
there is an $N>0$ so that $T_\infty^N$ is unipotent. Furthermore, for such
an $N$, we have $(T^N_\infty-\id)^{n+1} = 0$.
\end{prop}

\begin{block} \label{block:MHS}
In \cite{Steenb_Oslo}, Steenbrink constructs a
mixed Hodge structure \index{mixed Hodge structure}
on the
\emph{vanishing cohomology} \index{vanishing cohomology}
$H^n(Y_\infty)$.
This means that on $H^n(Y_\infty,\Q)$ one has the
\emph{weight filtration} \index{weight filtration}
 \[
  0 = W_0 H^n(Y_\infty,\Q) \subset \ldots \subset W_{2n} H^n(Y_\infty,\Q)
  = H^n(Y_\infty,\Q)
\] \nomenclature[WkH]{$W_\cdot H^n(Y_\infty,\Q)$}{Weight filtration}
and on $H^n(Y_\infty,\C)$, the
\emph{Hodge filtration} \index{Hodge filtration}
\[
  H^n(Y_\infty,\C) = F^0 H^n(Y_\infty,\C) \supset \ldots \supset
  F^{n+1} H^n(Y_\infty,\C) = 0.
\] \nomenclature[FkH]{$F_\cdot H^n(Y_\infty,\C)$}{Hodge filtration}
Furthermore, these susbspaces are invariant under the semisimple part
of the monodromy operator $T_\infty$.
The filtrations induce graded objects
\[
  \Gr_k^W H^n(Y_\infty,\Q) = 
    \frac{ W_k H^n(Y_\infty,\Q) }{ W_{k-1} H^n(Y_\infty,\Q) },\quad
  \Gr^p_F H^n(Y_\infty,\C) = 
    \frac{ F^p H^n(Y_\infty,\C) }{ F^{p+1} H^n(Y_\infty,\C) }.
\]
Furthermore, as a subquotient of $H^n(Y_\infty,\C)$, the space
$\Gr_k^W H^n(Y_\infty,\Q) \otimes \C$ inherits the Hodge filtration making
it a Hodge structure of weight $k$.
For each of these spaces, we denote by
$(\cdot)_\lambda$ \nomenclature{$(\cdot)_\lambda$}{Eigenspace}
the eigenspace of the semisimple part of $T_\infty$ with eigenvalue $\lambda$.
\end{block}

\begin{definition} \label{def:spec}
The \emph{spectrum} of an isolated hypersurface singularity defined by
$f\in\O_{\C^{n+1},0}$ is the element
\begin{equation}\label{eq:def_spec}
  \Sp(f,0) =
  \sum_\lambda \sum_{p=0}^n
                      \dim_\C (\Gr^p_F H^n(Y_\infty,\C))_\lambda
                      \left(\frac{\log \lambda}{2\pi i} + n - p \right)
  \in \Z[\Q].
\end{equation} \nomenclature[Sp]{$\Sp(f,0)$}{Spectrum}
where we choose
$-1<\Re\left(\frac{\log \lambda}{2\pi i}\right)\leq 0$, and $(a)$
denotes the element corresponding to $a\in\Q$ in the group ring $\Z[\Q]$.
By \cref{prop:monodromy}, we then have $\frac{\log \lambda}{2\pi i} \in \Q$.
For any subset $I\subset \Q$ we define
$\Sp_I(f,0) = \pi_I(\Sp(f,0))$,
\nomenclature[SpI]{$\Sp_I(f,0)$}{Part of spectrum}
where $\pi_I: \Z[\Q] \to \Z[\Q]$ is the projection sending $(a)$ to 
$(a)$ if $a\in I$, but to $0$ if $a\notin I$. For simplicity, we also set
$\Sp_{\leq 0}(f,0) = \Sp_{]-\infty,0]}(f,0)$.

Since the coefficients in \cref{eq:def_spec} are nonnegative integers, we
may also write $\Sp(f,0) = \sum_{j=1}^\mu (l_j)$ where 
$l_1, \ldots, l_\mu\in\Q$ satisfy $l_1\leq\ldots\leq l_\mu$.
Here, $\mu = \dim_\C(Y_\infty,\C)$ is the
Milnor number. \index{Milnor number}
\nomenclature[m]{$\mu$}{Milnor number}
\end{definition}

\begin{block} \label{block:Hodge_numbers}
The
mixed Hodge structure \index{mixed Hodge structure}
on the
vanishing cohomology \index{vanishing cohomology}
induces 
\emph{Hodge numbers} \index{Hodge numbers}
\begin{equation} \label{eq:hpq}
  h^{p,q} = \dim_\C \Gr_F^p \Gr^W_{p+q} H^n(Y_\infty,\C).
\end{equation} \nomenclature[hpq]{$h^{p,q}$}{Hodge numbers}
The action of the semisimple part
of the
monodromy, \index{monodromy}
induces 
\emph{equivariant Hodge numbers} \index{Hodge numbers!equivariant}
\begin{equation} \label{eq:hpql}
  h^{p,q}_\lambda = \dim_\C
    \left( \Gr_F^p \Gr^W_{p+q} H^n(Y_\infty,\C) \right)_\lambda
\end{equation}\nomenclature[hpql]{$h^{p,q}_\lambda$}{Equivariant Hodge numbers}
for $\lambda \in \C$.
Equivalently to the definition above, we now have
\[
  \Sp(f,0) = \sum_{p,q,\lambda} h^{p,q}_\lambda
                      \left(\frac{\log \lambda}{2\pi i} + n - p \right)
           = \sum_{\alpha\in\Q} \sum_{q\in\Z}
		     h^{n+\lfloor -\alpha \rfloor,q}_{\exp(2\pi i \alpha)} (\alpha).
\]
\end{block}

\begin{prop}[\cite{ScherkSteenb} (7.3)] \label{prop:spec_prop}
We have the following properties.
\begin{enumerate}
\item \label{it:spec_prop_sym}
The spectrum is symmetric around $\frac{n-1}{2}$. More precisely, we have
$\Sp(f,0) = \iota \Sp(f,0)$, where $\iota:\Z[\Q] \to \Z[\Q]$ is the group 
automorphism sending $(a)$ to $(n-1-a)$ for $a\in\Q$.

\item \label{it:spec_prop_int}
The spectrum is contained in the interval $]-1,n[$. More precisely,
for every monomial
$(a)$ in the sum of \cref{eq:def_spec} with nonzero coefficient we have
$-1<a<n$.
\end{enumerate}
\qed
\end{prop}

\begin{prop}[Saito \cite{Saito_exp}] \label{prop:Saito_exp}
Let $f\in\O_{\C^{n+1},0}$ define an isolated
hypersurface \index{hypersurface}
singularity $(X,0)\subset(\C^{n+1},0)$.
Assume furthermore that $f$ has
Newton nondegenerate \index{Newton nondegenerate}
principal part (for
definitions of diagrams and nondegeneracy, see \cref{ss:diag_nondeg}, for
the
Newton filtration, \index{Newton filtration}
(see \cref{def:Newt}).
The part of the
spectrum \index{spectrum}
lying in $]-1,0]$ is given by the Newton weight
function of monomials containing all three variables which are under the
Newton diagram. \index{Newton diagram}
That is, we have
$\Sp_{\leq 0}(f,0) = \sum_{p\in\Z^{n+1}_{>0} \cap \Gamma_-(f)} (\ell_f(x^p)-1)$.
\end{prop}

\begin{rem}
\begin{list}{(\thedummy)}
{
  \usecounter{dummy}
  \setlength{\leftmargin}{0pt}
  \setlength{\itemsep}{0pt}
  \setlength{\itemindent}{4.5pt}
}

\item
The spectrum \index{spectrum}
is an invariant that depends only on the
Hodge filtration. \index{Hodge filtration}
A stronger invariant, the
\emph{spectral pairs} \index{spectral pairs}
\cite{Steenb_Oslo}, take the
weight filtration \index{weight filtration}
into account as well. In fact, the spectral pairs encode the same data as
the equivariant Hodge numbers. We will, however, not make any use of the
spectral pairs.

\item
In \cite{Nem_2spec}, N\'emethi shows that the mod 2 spectral pairs
are equivalent to the real Seifert form. In particular, the residue
mod 2 of the spectrum is determined by the topology of the embedding
$(X,0)\subset(\C^{n+1},0)$. This information, along with $\Sp_{\leq 0}(f,0)$
then determines the full spectrum $\Sp(f,0)$ by
\cref{prop:spec_prop}\ref{it:spec_prop_int}.

\end{list}
\end{rem}

\subsection{Statement of results}

Assume that $f\in\O_{\C^3,0} = \C\{x_1,x_2,x_3\}$ is the germ of a
holomorphic function 
in three variables defining an isolated
hypersurface \index{hypersurface}
singularity $(X,0)$ at the origin.
Assume, furthermore, that $f$ has
Newton nondegenerate \index{Newton nondegenerate}
principal part
(see \cref{ss:diag_nondeg}) and that
that the link $M$ of $(X,0)$ is a
rational homology sphere. \index{homology sphere!rational}
We denote by $p_g$ the
geometric genus \index{geometric genus}
of $(X,0)$.

The following theorem gathers the main results of the thesis.

\begin{thm} \label{thm:main}
\begin{list}{}
{
  \setlength{\leftmargin}{0pt}
  \setlength{\itemsep}{0pt}
  \setlength{\itemindent}{4.5pt}
}

\item[\bf\ref{rt:article}]
There exists a
computation sequence \index{computation sequence}
$(Z^{\ref{rt:article}}_i)_{i=0}^k$ for
$Z_K$
on the
minimal good resolution \index{minimal good resolution}
graph of $M$ satisfying
\begin{equation} \label{eq:thm_I}
  p_g = \sum_{i=0}^k \max\{0,(-Z^{\ref{rt:article}}_i,E_{v(i)})+1\}
      = \sw_M(\scan) - \frac{Z_K^2+|\V|}{8}.
\end{equation}
\index{Seiberg--Witten invariant!normalized}
\index{spinc structure@$\spinc$ structure!canonical}
\index{geometric genus}
Furthermore, this computation sequence can be easily calculated using the 
minimal resolution graph, see \cref{def:comp_seq_constr}.

\item[\bf\ref{rt:Newton}]
Assuming that the
Newton diagram \index{Newton diagram}
$\Gamma(f)$ is
convenient \index{convenient}
(see
\cref{def:convenient}) and that $G$ is the
resolution graph \index{resolution graph}
obtained from
Oka's algorithm \index{Oka's algorithm}
using this diagram (see \cref{block:oka_alg}),
there exists a
computation sequence \index{computation sequence}
$(Z^{\ref{rt:Newton}}_i)_{i=0}^k$ for
$\wt(f)$ (see \cref{def:wtdiv})
satisfying
\[
  P_X^\A(t) = \sum_{i=0}^\infty \max\{0, (-Z^{\ref{rt:Newton}}_i,E_{v(i)})+1\}
                                t^{r_i}
\]
where $P_X^\A(t)$ is the
Poincar\'e series \index{Poincar\'e series}
associated to the
Newton filtration \index{Newton filtration}
on $\O_{X,0}$ (see \cref{ss:Newt_filt}),
$(Z^{\ref{rt:Newton}}_{i=0})^\infty$ is the
continuation \index{continuation to infinity}
of
$(Z^{\ref{rt:Newton}})$ to infinity as in \cref{def:comp_seq} and
we set $r_i = m_{v(i)}(Z^{\ref{rt:Newton}}_i)$ for each $i\geq 0$.

The part of the
spectrum \index{spectrum}
$\Sp_{\leq 0}(f,0)$ is obtained
from this series by the equality
\[
  \Sp_{\leq 0}(f,0) = P_X^{\A,\mathrm{pol}}(t^{-1}),
\]
where we identify the ring of
Laurent--Puiseux polynomials \index{Laurent--Puiseux polynomials}
$\Z[t^{\pm1/\infty}]$ with the group ring $\Z[\Q]$.

Furthermore, this computation sequence $(Z^{\ref{rt:Newton}}_i)_{i=0}^k$
can be easily calculated, assuming only the
knowledge of $G$, the resolution graph, and the cycle $\wt(x_1x_2x_3)$
(see \cref{def:comp_seq_constr}).

\item[\bf\ref{rt:spec}]
Assuming that the
Newton diagram \index{Newton diagram}
$\Gamma(f)$ is
convenient \index{convenient}
(see
\cref{def:convenient}) and that $G$ is the resolution graph obtained from
Oka's algorithm \index{Oka's algorithm}
using this diagram (see \cref{block:oka_alg}),
there exists a
computation sequence \index{computation sequence}
$(Z^{\ref{rt:spec}}_i)_{i=0}^k$ for
$x(Z_K-E)$ (see \cref{prop:x_exist})
satisfying
\[
  \Sp_{\leq 0}(f,0) = \sum_{i=0}^{k-1}
                           \max \{0,(-Z^{\ref{rt:spec}}_i,E_{v(i)})+1 \} (r_i)
						   \in \Z[\Q],
\index{spectrum}
\]
where, for each $i$ we set
\[
  r_i = \frac{m_{v(i)}(Z^{\ref{rt:spec}}_i) + \wt_{v(i)}(x_1 x_2 x_3)}
             {\wt_{v(i)}(f)}.
\]
Furthermore, this sequence can be easily computed, assuming only the
knowledge of $G$, the resolution graph, and of the cycle $\wt(x_1x_2x_3)$
(see \cref{def:comp_seq_constr}).
\end{list}
\end{thm}
\begin{proof}
See \cref{thm:pg_as,thm:Newt,thm:qZ_ident}.
\end{proof}

\begin{block}
The computation sequences in the above theorem are constructed in
\cref{ss:alg}. In \cref{ss:pt_count} we define a sequence of sets
$P_i \subset\Z^3$ and prove the equation
$|P_i| = \max\{ 0, (-Z_i,E_{v(i)}) \}$ in cases \ref{rt:article},
\ref{rt:spec}, as well as a similar equation in case \ref{rt:Newton}.
In each case, this result is obtained in two steps.
In the technical \cref{lem:pts_faces} we show that for each $i$, the
set $P_i$ is given as the set of integral points in a diluted polygon
in an affine plane with a lattice.
We then apply \cref{thm:point_count} to relate the cardinality
of $P_i$ with the intersection number $(-Z_i,E_{v(i)})$.

In \cref{ss:pgspec} we apply a statement proved by Ebeling and Gusein-Zade
\cite{Ebeling_Gusein-Zade} to prove
\[
  |P_i| \leq \dim_\C \frac{H^0(\X,\O_{\X}(-Z_i)}{H^0(\X,\O_{\X}(-Z_{i+1}))}
\]
in cases \ref{rt:article} and \ref{rt:spec}. Along with \cref{thm:comp_seq},
this gives the above formula for the geometric genus. Furthermore, Saito's
result \cref{prop:Saito_exp} provides the formula for $\Sp_{\leq 0}(f,0)$.

In \cref{ss:PNfilt}, we give a formula for the Poincar\'e series
associated with the Newton filtration. For this, we use the fact
that the sets $P_i$ provide a partition of $\N^3$ which is
proved in \cref{s:seq}.
Our formula for the
Poincar\'e series is then proved using \cref{lem:PNseries}.

The equality to the right in \cref{eq:thm_I} is proved in \cref{s:SW}.
There, we relate the coefficients of the counting function with the sets
$P_i$. More precisely, we show that for each $i$ we have
\[
q_{Z_{i+1}} - q_{Z_i} = |P_i|.
\]
This rhymes with the results in the previous section, which can be
written $h_{Z_{i+1}} - h_{Z_i} = |P_i|$. The proof is technical and
is split into cases, each requiring attention to many details.
\end{block}

\begin{rem}
\begin{list}{(\thedummy)}
{
  \usecounter{dummy}
  \setlength{\leftmargin}{0pt}
  \setlength{\itemsep}{0pt}
  \setlength{\itemindent}{4.5pt}
}

\item
In \cref{ss:cyc_polyt} we explain the choice of minimal and convenient
diagrams in \cref{thm:main}.

\item
The result
$p_g = \sum_{i=0}^k \max\{0,(-Z^{\ref{rt:article}}_i,E_{v(i)})+1\}$
\index{geometric genus}
in \cref{rt:article} in the theorem above can be found in a joint article
of N\'emethi and the author \cite{Nem_Bal}.

\item
If $f$ is a weighted homogeneous polynomial, then the resolution
graph of $(X,0)$ has a unique node, say $n$. By construction, we also
have $(Z_i,E_{v(i)}) > 0$ unless $v(i) = n$
(see \cref{ss:Laufer_seq,ss:alg}). Using
\cref{lem:x_interpol}, if $v(i) = n$ and $m = m_n(Z_i)$, then one obtains
\[
  \max\{ 0, (-Z_i,E_{v(i)})+1 \}
  = \max \left\{ 0, mb_n
    - \sum_i \left\lceil \frac{\beta_i m}{\alpha_i} \right\rceil + 1 \right\}
\]
where $(\alpha_i/\beta_i)_i$ are the Seifert invariants of the link.
The first equality in \cref{eq:thm_I} therefore follows from
Theorem 5.7 of \cite{Pinkham} (note that we use here the assumption that
the central curve is rational). Furthermore, the second equality follows
from the already proved Seiberg--Witten invariant conjecture for
weighted homogeneous singularities \cite{Nem_Nico_SWII}.

\end{list}
\end{rem}

\newpage
\section{Newton diagrams and nondegeneracy} \label{s:newton_diag}

In this section we will recall the definition of a
Newton diagram \index{Newton diagram}
associated
with a function $f\in\O_{\C^{n+1},0}$, the nondegeneracy condition and some
important properties of singularities defined by nondegenerate functions.

In what follows, $f$ is a function germ around the origin in $\C^{n+1}$
and $(X,0)$ is the germ of the zero set of $f$. We will assume that
$X$ has an isolated singularity at the origin (see \ref{lem:isol}) and
that $f$ has
Newton nondegenerate \index{Newton nondegenerate}
principal part (\cref{def:nondeg}).
We will also assume that $n=2$, except for in
\cref{ss:diag_nondeg,ss:Newt_filt}
In the sections following this one, we will also assume that the link is a
rational homology sphere. \index{homology sphere!rational}

\subsection{Diagrams and nondegeneracy} \label{ss:diag_nondeg}

Let $f\in\O_{\C^{n+1},0}$ and
write $f = \sum_{p\in\N^3} a_p x^p$. We define the \emph{support}
of $f$ as $\supp(f) = \set{p\in\N^{n+1}}{a_p \neq 0}$. The
\emph{Newton polyhedron} \index{Newton polyhedron}
$\Gamma_+(f)$ \nomenclature[G+(f)]{$\Gamma_+(f)$}{Newton polyhedron}
of $f$ is the convex closure of
$\cup_{p\in\supp(f)} p+\R^{n+1}_{\geq 0}$.
The \emph{Newton diagram} \index{Newton diagram}
$\Gamma(f)$ \nomenclature[G(f)]{$\Gamma(f)$}{Newton diagram}
of $f$
is the union of compact two dimensional faces of the Newton polyhedron.
Here, a face $F \subset \Gamma_+(f)$ means the minimal set of any linear
function $\R^{n+1}\to\R$. We also denote by
$\Gamma_-(f)$ \nomenclature[G-(f)]{$\Gamma_-(f)$}{Under Newton diagram}
the union of segments
joining the origin in $\R^{n+1}$ with $\Gamma(f)$.

\begin{definition} \label{def:nondeg}
Let $F\subset \Gamma(f)$ be a compact face of the
Newton polyhedron \index{Newton polyhedron}
and define
$f_F(x) = \sum_{p\in F} a_p x^p$. \nomenclature[fF]{$f_F$}{Principal part}
We say that $f$ is
\emph{nondegenerate with respect to} $F$ if the set of
equations
$\frac{\partial}{\partial x_i} f_F = 0$ has no solution in $(\C^*)^3$.
We say that $f$ has \emph{Newton nondegenerate principal part}
if $f$ is nondegenerate with respect to every nonempty face of $\Gamma(f)$.
\end{definition}

\begin{rem}
The condition of nondegeneracy is equivalent to the zeroset of
$f_F$ in $(\C^*)^3$ being smooth for each face $F\subset\Gamma(f)$.
\end{rem}

\begin{prop}[Koushnirenko \cite{Kouchnirenko} Th\'eor\`eme 6.1]
For any $S\subset\N^{n+1}$, the function
$f = \sum_{p\in S} a_p x^p \in \O_{\C^{n+1},0}$ has Newton nondegenerate
principal part for a generic choice of coefficients $a_p\in\C^*, p\in S$.
\end{prop}

The main statement in the following lemma can be found in Kouchnirenko's
article \cite{Kouchnirenko} as Remarque 1.13(ii).
The case $n=2$ is an observation made by Braun and N\'emethi
\cite{Newton_nondeg}.

\begin{lemma} \label{lem:isol}
Let $f\in\O_{\C^{n+1},0}$ define a singularity $(X,0)$ and assume that $f$ has
Newton nondegenerate \index{Newton nondegenerate}
principal part. Then $X$ has an isolated
singularity at $0$ if and only if for any
$I \subset \{ 1, 2, \ldots, n+1 \}$ we have
\[
  \left|
    \set{i}{\supp(f)-e_i \cap \R^I_{\geq0} \neq \emptyset}
  \right|
    \geq |I|.
\]
where
\[
  \R^I_{\geq 0}
     = \set{(p_1,\ldots, p_{n+1})\in\R^{n+1}}{\fa{i\notin I}{p_i = 0}}
\]
and $e_1, \ldots, e_{n+1}$ is the natural basis of $\R^{n+1}$.

In the case $n=2$, this is equivalent to the following condition.
The diagram $\Gamma(f)$ contains a point on each
coordinate hyperplane and a point at distance at most $1$ from each
coordinate axis.
\qed
\end{lemma}

\begin{definition}\label{def:convenient}
We say that $f \in \O_{\C^3,0}$ is
\emph{convenient} \index{convenient}
(\emph{commode} in French) if
any of the following equivalent conditions is fulfilled.
\begin{itemize}
\item
$\supp(f)$ contains an element of each coordinate axis.
\item
The set $\R_{\geq0}^3 \setminus \Gamma_+(f)$ is bounded.
\item
$\R_{\geq 0} \Gamma(f) = \R_{\geq 0}^3$.
\end{itemize}
\end{definition}

\subsection{Oka's algorithm} \label{ss:Okas_alg}

A singularity $(X,0) \subset (\C^3,0)$ given by $f\in\O_{\C^3,}$
with Newton nondegenerate principal part as defined above has an
explicit resolution.
This is obtained through a modification of $\C^3$ constructed
from the Newton diagram $\Gamma(f)$.
Such modifications have been applied in any dimension,
see e.g. \cite{Var_zeta}, as
well as \cite{AGZVII} Chapter 8 and references therein.
Oka \cite{Oka} proved that in the case of surfaces, a relatively simple
algorithm
can be applied to the diagram to compute the graph associated to the
resolution of $(X,0)$ and that this does not depend on choices made
during the construction.
In this section we describe this algorithm and give related definitions.

We will use the notation from \ref{ss:res} for this
resolution. \index{resolution}

Recall that for integers $b_1, \ldots, b_s$ we have the
\emph{negative continued fraction} \index{continued fraction}
\[
  [b_1, \ldots, b_s] = b_1 - \frac{1}{b_2 - \ddots}.
\]
Further, the string $b_1, \ldots, b_s$ is referred to as the
\emph{negative continued fraction expansion} of the rational number above.
If we require $b_j \geq 2$ for $j \geq 2$, then the expansion is unique.
As we will never make use of positive continued fraction, we will often
simply say continued fraction. See \cite{P-P_cfrac} for a detailed
discussion of continued fractions and how they relate to the topology
of surface singularities, as well as \cite{BHPVdV} III.5 and
\cite{Laufer_book} Chapter II.

The statements in the following definition are not difficult to prove.
\begin{definition}
Let $A$ be a free abelian group of finite rank and take distinct primitive
elements $a,b\in A$.
\begin{itemize}
\item
The
\emph{determinant} \index{determinant}
$\alpha(a,b)$ \nomenclature[aab]{$\alpha(a,b)$}{Determinant of vectors $a,b$}
of $a,b$ is the 
greatest common divisor of maximal minors of the matrix whose rows are
given by the coordinate vectors of $a$ and $b$ with respect to some basis
of $A$.
\item
If $\alpha(a,b)>1$, then we define
the
\emph{denominator} \index{denominator}
$\beta(a,b)$ \nomenclature[bab]{$\beta(a,b)$}{Numerator of vectors $a,b$}
of $a,b$ as the unique integer
$0\leq \beta(a,b) < \alpha(a,b)$ for which $\beta(a,b) a + b$ has
content $\alpha(a,b)$.
\item
If $\alpha(a,b)=1$, we choose the denominator to be $\beta(a,b) = 1$ or
$\beta(a,b) = 0$.
\item
If $\alpha(a,b) > 1$, then
the
\emph{selfintersection numbers} \index{selfintersection number}
associated with $a,b$ are defined as 
$-b_1, \ldots, -b_s$, where $b_1, \ldots, b_s$ is the
continued fraction \index{continued fraction}
expansion of $\alpha(a,b)/\beta(a,b)$.
\item
The
\emph{canonical primitive sequence} \index{canonical primitive sequence}
associated with $a,b$ is the unique
sequence $a_1, \ldots, a_s \in A$ satisfying 
$a_{i-1} - b_i a_i + a_{i+1} = 0$
for $i=1, \ldots, s$, where $a_0 = a$ and $a_{s+1} = b$.
\item
If $\alpha(a,b) = 1$ and we choose $\beta(a,b) = 1$, then the
selfintersection numbers \index{selfintersection number}
associated with $a,b$ consist of a single $-1$, and the 
canonical primitive sequence \index{canonical primitive sequence}
is $a_1 = a+b$. If we choose $\beta(a,b) = 0$,
then both sequences are empty.
\end{itemize}
We refer to $\alpha(a,b)/\beta(a,b)$ as the \emph{fraction} associated with
$a,b\in A$.
\end{definition}

\begin{rem} \label{rem:alpbet}
We have $\gcd(\alpha(a,b),\beta(a,b)) = 1$. Thus, the fraction associated
with $a,b$ determines the
determinant \index{determinant}
and the
denominator. \index{denominator}
\end{rem}

\begin{block} \label{block:abc_comp}
If $a,b\in \Z^3$, then the determinant, denominator and canonical primitive
sequence can be calculated as follows. We must first make sure that
both vectors are primitive. Then, $\alpha(a,b)$ is obtained as
the content of the cross product $a\times b$. The denominator $\beta(a,b)$
is the unique number $0\leq \beta(a,b) < \alpha(a,b)$ for which
$\beta(a,b) a + b$ has content $\alpha(a,b)$. If the $i^{\mathrm{th}}$
coordinate $a_i$ of $a$ has no common factors with $\alpha(a,b)$,
then we have $\beta(a,b) = - b a^{-1}$, where the inverse is taken
in the ring $\Z/\alpha(a,b)\Z$. Otherwise, one can check manually
any number from $0$ to $\alpha(a,b)$ (using a computer is quite helpful).
Finally, the canonical primitive sequence is determined recursively
by $a_0 = a$, $a_1 = (\beta(a,b) a + b)/\alpha(a,b)$ and
$a_{i+1} = -a_{i-1} + b_i a_i$.
\end{block}

\begin{block} \label{block:oka_alg}
\index{Oka's algorithm}
We are now ready to construct the graph $G$. First, let
$\Nd^*$ \nomenclature[N*]{$\Nd^*$}{Index set for two dimensional faces}
be a
set indexing the two dimensional faces of $\Gamma_+(f)$, that is,
let $\set{F_n}{n\in\Nd^*}$ be the set of two dimensional faces of
$\Gamma_+(f)$. Let
$\Nd$ \nomenclature[N]{$\Nd$}{Index set for compact two dimensional faces}
be the subset of $\Nd^*$ corresponding to
compact faces.
For each $n\in\Nd^*$ let
$\ell_n$ \nomenclature[ln]{$\ell_n$}{Support function of Newton polyhedron}
be the unique integral primitive linear
function on $\R^3$ having
$F_n$ \nomenclature[Fn]{$F_n$}{Compact face of Newton polyhedron}
as its minimal set on $\Gamma_+(f)$.
For any $n\in\Nd$ and $n'\in\Nd^*$, let
$t_{n,n'}$ \nomenclature[tnn']{$t_{n,n'}$}{Length of segment $F_n\cap F_{n'}$}
be the one dimensional
combinatorial volume of $F_n\cap F_{n'}$. This is the same as the
number of components of $F_n\cap F_{n'}\setminus \Z^3$.
We also define
$\alpha_{n,n'} = \alpha(\ell_n,\ell_{n'})$
\nomenclature[ann']{$\alpha_{n,n'}$}{Determinant of $\ell_n$ and $\ell_{n'}$}
and
$\beta_{n,n'} = \beta(\ell_n,\ell_{n'})$,
\nomenclature[bnn']{$\beta_{n,n'}$}{Denominator of $\ell_n$ and $\ell_{n'}$}
where, if $\alpha_{n,n'} = 1$,
we choose $\beta_{n,n'} = 0$ if $n'\in\Nd$, but
$\beta_{n,n'} = 1$ if $n' \in \Nd^*\setminus\Nd$.

The graph
$G^*$ \nomenclature[G*]{$G^*$}{Extended graph}
is obtained as follows. First take
$\Nd^*$ as vertex set. Then, for any $n\in\Nd$ and $n'\in\Nd^*$,
add $t_{n,n'}$ copies of the
bamboo \index{bamboo}
depicted in \cref{fig:Bamboo}.

Let $\ell_{v_1}, \ldots, \ell_{v_s}$ be the
canonical primitive sequence \index{canonical primitive sequence}
associated with $\ell_n, \ell_{n'}$. We then have elements $\ell_v$
associated with all vertices of the graph $G^*$.
Let
$\V^*$ \nomenclature[V*]{$\V^*$}{Vertex set of $G^*$}
be the set of vertices of $G^*$ and for $v\in\V^*$,
let
$\V^*_v$ \nomenclature[Vv]{$\V^*_v$}{Set of neighbours in $G^*$}
be the set of neighbours of $v$ in $G^*$.

\mynd{Bamboo}{ht}{A bamboo.}{fig:Bamboo}

Let $\V$ be the set of vertices not in $\Nd^*\setminus \Nd$. Then,
define $G$ as the subgraph of $G^*$ generated by the vertex set $\V$.
The vertices $v_1, \ldots, v_s$ (as in \cref{fig:Bamboo})
are labelled with the
selfintersection numbers \index{selfintersection number}
associated with
$\ell_n, \ell_{n'}$, taken as (primitive) elements of $\Hom(\Z^3,\Z)$.
For $n\in\Nd$ we define the selfintersection number $-b_n$ as the unique
solution of the equation
\begin{equation} \label{eq:nbr_sum}
  -b_n \ell_n + \sum_{u\in\V^*_n} \ell_u = 0.
\end{equation}
Thus, for every $v\in\V$ we have a selfintersection number $-b_v$.
Furthermore, by the definition of $b_v$ for $v\in\V\setminus\Nd$,
\cref{eq:nbr_sum} holds with $n$ replaced by $v$.

For $G$ to be a
plumbing graph, \index{plumbing graph}
we must provide genera $[g_v]$ for
all $v\in \V$. For $n\in\Nd$, let $g_n$ be the number of integral points
in the relative interior of the polygon $F_n$. All other vertices get
genus $0$.
\end{block}

\begin{definition} \label{def:coord}
In addition to the linear functions $\ell_v$ for $v\in\V^*$ defined above,
let
$\ell_c$ \nomenclature[l1]{$\ell_1$}{Standard coordinate function in $\R^3$}
be the standard coordinate functions for $c=1,2,3$, that is,
$\ell_c(p) = p_c$ for $p = (p_1, p_2, p_3) \in \R^3$.
\end{definition}

\begin{definition}
For $n\in\Nd$, let $\Nd^*_n = \set{n'\in\Nd^*}{t_{n,n'} > 0}$ and
$\Nd_n = \Nd_n^*\cap\Nd$.
If $n\in \Nd$, $n'\in\Nd^*_n$ and $\beta_{n,n'} \neq 0$,
let $u_{n,n'} = v_1$ as in \cref{fig:Bamboo}.
If $\beta_{n,n'} = 0$, let
$u_{n,n'} = n'$.
\nomenclature[unn']{$u_{n,n'}$}{Neighbour of $n$ on branch containing $n'$}

\end{definition}

\begin{rem} \label{rem:Oka}
\begin{list}{(\thedummy)}
{
  \usecounter{dummy}
  \setlength{\leftmargin}{0pt}
  \setlength{\itemsep}{0pt}
  \setlength{\itemindent}{4.5pt}
}
\item \label{it:Oka_nodes}
Note that $\beta_{n,n'} = 0$ can only happen if $n'\in\Nd$, thus we always
have $u_{n,n'} \in \V$. In particular, we have
$\V^*_n = \V^{\phantom{*}}_n$ for $n\in\Nd$.
Furthermore, this convention guarantees that the set $\Nd$ is precisely
the set of nodes in the graph $G$, that is, the set of vertices
of degree at least $3$.

\item
If $t_{n,n'} > 1$, we must choose a neighbour $u_{n,n'}$ out of a set
of $t_{n,n'}$ elements. By construction, however, the functional
$\ell_u$ is well defined, for any such choice.
The numbers $m_u(\psi(l))$ (see \cref{lem:Sl_struct}) and
$m_u(Z)$, where $x(Z) = Z$ (see \cref{ss:Laufer_seq}), are also well defined
in this case.

\item
An implementation of the algorithm \cref{block:oka_alg} is available at
\cite{Oka.gp}.

\item \label{it:rem_Oka_n1}
The definition of the graph $G$ is specific to hypersurfaces in $\C^3$.
For $f\in\O_{\C^n+1,0}$, however, we can index the $n$ dimensional compact
faces of $\Gamma_+(f)$ by a set $\Nd$ and obtain primitive linear support
functions $\ell_n:\R^{n+1}\to\R$ as above.

\end{list}
\end{rem}

\begin{rem}
For $n\in\Nd$, the existence of $b_n$ is not obvious, but can be seen as
follows. Let $H$ be the hyperplane in $\R^3$ defined by $\ell_n = m$, where
$m$ is the value of $\ell_n$ on $F_n$. It follows from the definition
of the
canonical primitive sequence \index{canonical primitive sequence}
that for any $u\in\V_n$, the affine
function $\ell_u|_H$ is in fact primitive, and its minimal set on $F_n$
is $F_n\cap F_{n'}$, where $u$ is assumed to lie on a bamboo connecting
$n$ and $n'\in\Nd^*$. We now see that there is a natural correspondence
between the neighbours $u\in\V_n$ and the primitive segments of the
boundary $\partial F_n$. It is simple to show that under these conditions,
the sum $\sum_{u} \ell_u|_H$ is constant (see e.g. proof of
\cref{thm:point_count}).
Since $\ell_n$ is by definition
also constant on $H$, the existence of $b_n$ follows.
Furthermore, since $\ell_n$ is primitive, we have $b_n\in\Z$.
Finally, since all $\ell_v$ are positive on the open positive quadrant,
we must have $b_n > 0$.
\end{rem}

Recall that for $Z\in L$, we denote by $m_v(Z)$ the coefficient of $E_v$,
that is, $Z = \sum_{v\in\V} m_v(Z) E_v$.
\begin{lemma} \label{lem:mod_ZKE}
Let $n\in\Nd$ and $n'\in\Nd_n$. Then, for $u= u_{n,n'}$ we have
$\alpha_{n,n'}m_u(Z_K-E) = \beta_{n,n'}m_n(Z_K-E) + m_{n'}(Z_K-E)$.
Similarly, if $n'\in\Nd^*_n\setminus\Nd$, then
$\alpha_{n,n'}m_u(Z_K-E) = \beta_{n,n'}m_n(Z_K-E) - 1$.
\end{lemma}
\begin{proof}
This follows from the more general \cref{lem:mod}, since
$(Z_K-E,E_v) = 0$ if $\delta_v = 2$.
\end{proof}

\begin{prop} \label{prop:isol}
Take $f\in\O_{\C^3,0}$ as above with
Newton nondegenerate \index{Newton nondegenerate}
principal part defining
an isolated
hypersurface \index{hypersurface}
singularity $(X,0)$. The
link \index{link}
of $(X,0)$ is a
rational homology sphere \index{homology sphere!rational}
if and only if $\Gamma(f) \cap \Z_{>0}^3 = \emptyset$.
\end{prop}
\begin{proof}
Let $g$ and $c$ be as in \cref{prop:H}. We see immediately that $g = 0$
if and only if for each $n\in \Nd$, the face $F_n$ contains no integral
points in its relative interior.

If $c \neq 0$, then we must have at least one of
the following possibilities: there are $n_1, n_2\in\Nd$ with $t_{n_1,n_2} > 1$,
or, there are $n_1, \ldots, n_s \in \Nd$ so that
$t_{n_i,n_{i+1}} \neq 0$ for $i=1, \ldots, s$ (where we set $n_{s+1} = n_1$)
and $\cap_{i=1}^s F_{n_i}$ is a zero dimensional face of $\Gamma(f)$. 
In the first case, $F_{n_1} \cap F_{n_2} \subset \Gamma(f)$ contains an
integral point with positive coordinates and in the second case the point
in $\cap_{i=1}^s F_{n_i}$ is such a point.

From this we see that $g = c = 0$ if and only if any integral
point $p\in \Gamma(f) \cap \Z^3$ lies on the boundary $\partial \Gamma(f)$.
But every segment of the boundary of $\Gamma(f)$
which is not contained in some coordinate hyperplane has the form
$[(a,0,b),(0,1,c)]$ for some $a,b,c \in \N$ modulo permutation of coordinates,
and so all integral points on the boundary of $\Gamma(f)$ lie on some
coordinate hyperplane. The proposition follows.

Alternatively, by Saito's result \cref{prop:Saito_exp}, the multiplicity of
$0$ in the spectrum is precisely the number of integral points in $\Gamma(f)$
with positive coordinates (we use here the convention given in
\cref{def:spec}, in \cref{prop:Saito_exp}, the `exponents' are normalized
between $0$ and $n+1$). But $0$ has nonzero multiplicity in the spectrum
if and only if the monodromy has $1$ as an eigenvalue, which is equivalent
to $M$ having nontrivial rational homology (see e.g. \cite{Nem_sti} 3.6).
\end{proof}

We end this subsection with the following result which can greatly simplify
calculations.

\begin{prop} \label{prop:content}
Let $[p_1, p_2] \subset F_n$ be an edge of one of the faces of the
Newton diagram $\Gamma(f)$, thus, $[p_1, p_2] = F_n\cap F_{n'}$ for some
$n'\in\Nd_n^*$.
Let $q_1, q_2\in \partial F_n\cap\Z^3$ so that
$[p_1,q_1]$ and $[p_2, q_2]$ are the primitive segments adjacent
to $[p_1, p_2]$ in $\partial F_n$
and set $\alpha_1 = \ell_{n'}(q_1-p_1)$ and $\alpha_2 = \ell_{n'}(q_2 - p_2)$.
If $p_1$ is a
regular vertex \index{regular vertex}
of $F_n$, then
$\alpha_{n,n'} = \alpha_1$ and $\alpha_1 | \alpha_2$
(see \cref{def:reg_sing} for regular vertices).
\end{prop}

\begin{proof}
A simple calculation shows that $\alpha_{n,n'}$ can be identified as
the content
of the affine function $\ell_{n'}|_H$, where $H$ is the affine plane
containing $F_n$, that is, the smallest
positive integer $c$ for which there is an integer $0\leq r < c$ and
an integral functional $\ell:H\to\R$ so that
$\ell_{n'}|_H = c\ell+r$. It follows that there are
$a_1, a_2 \in \N$ so that $\alpha_1 = a_1 \alpha_{n,n'}$ and
$\alpha_2 = a_2 \alpha_{n,n'}$.
Since $p_1$ is a regular vertex of $F_n$, the points $p_1, p_2, q_1$
form an integral affine basis for
$H$,
hence $a_1 = 1$, and so $\alpha_1 = \alpha_{n,n'}$ and
$\alpha_2 = a_2 \alpha_{n,n'}$.
\end{proof}

\begin{rem}
Assume that $(X,0)$ is as in \cref{prop:isol} and that the link of
$(X,0)$ is a
rational homology sphere. \index{homology sphere!rational}
Then, by \cref{cor:reg_vx}, any edge of a face $F_n$ of the Newton diagram
contains a
regular vertex \index{regular vertex}
of $F_n$ as an endpoint.
\end{rem}

\begin{example}
Let
$f(x_1,x_2,x_3)
= x_1^4 + x_1^3x_2^2 + x_2^{10} + x_1^2x_3^3 + x_2^3x_3^4 + x_3^8$.
A simple calculation shows that the Newton polyhedron is given by the
inequalities
\[
  \langle (11, 5, 7),\cdot \rangle \geq  43, \quad
  \langle ( 6, 3, 4),\cdot \rangle \geq  24, \quad
  \langle (32,12,21),\cdot \rangle \geq 120, \quad
  \langle (15, 8, 6),\cdot \rangle \geq  48,
\]
as a subset of the positive octant. These vectors are the normal vectors
to the faces of the diagram.
In this case, the set $\Nd$
contains four elements, and the set $\Nd^*$ contains three elements,
one corresponding to each coordinate hyperplane.
On the left hand side of \cref{fig:ex}, we see the Newton diagram
$\Gamma(f)$.
On the right hand side, filled circles represent compact faces of
the Newton polyhedron and crosses represent noncompact ones. The segments
joining $n$ and $n'$ in this picture represent $t_{n,n'}$.
\mynd{ex}{ht}{A Newton diagram and its dual graph in the plane.}{fig:ex}
By calculation, we obtain the plumbing graph shown in \cref{fig:ex_graph},
with additional vertices corresponding to the elements of $\Nd^*\setminus\Nd$.

We carry out the calculations described in \cref{block:abc_comp} for
the pairs of vectors $(11,5,7), (6,3,4)$ and $(11,5,7), (15,8,6)$
and $(32,12,21),(0,0,1)$.

For the first pair, we have $(11,5,7)\times(6,3,4) = (-1,-2,3)$ which is
a primitive vector, hence $\alpha = 1$. By convention, we take
$\beta = 0$, yielding an empty canonical primitive sequence.

For the second pair, we have $(11,5,7)\times(15,8,6) = (-26,39,13)$,
thus $\alpha = 13$. Since $\gcd(13,11) = 1$, we
obtain $\beta = -11^{-1}\cdot 15 = 1$ in
$\Z/13\Z$, since $15 \equiv -11 \,(\mod 13)$. Since $13/1 = [13]$, the
canonical primitive sequence has length $1$ and consists of the vector
$((11,5,7)+(15,8,6))/13 = (2,1,1)$.

For the third pair, we have $(32,12,21)\times(0,0,1) = (12,-32,0)$, yielding
$\alpha = 4$. Furthermore, we have $\beta = -21^{-1}\cdot 1 = 3$
in $\Z/4\Z$. We have $4/3 = [2,2,2]$, and so the canonical primitive sequence
is given as $(24, 9, 16)$, $(16, 6, 11)$, $(8, 3, 6)$.

\mynd{ex_graph}{ht}{A plumbing graph obtained by Oka's algorithm.}{fig:ex_graph}
\end{example}

\subsection{On minimality} \label{ss:minimality}

In this subsection, we will recall some results on
minimality of plumbing \index{minimal graph}
graphs on one hand, and of
Newton diagrams \index{Newton diagram}
on the other. In 
\cite{Kouchnirenko}, Kouchnirenko introduces the condition of
convenience, \index{convenient}
(see \cref{def:convenient}), the assumption of
which can be of great convenience,
but does not actually reduce the generality when working with isolated
singularities. This is because for a given $f\in\O_{\C^3,0}$ with
Newton nondegenerate \index{Newton nondegenerate}
principal part, defining an isolated singularity,
the function $f+\sum_{i=1}^3 x_i^d$, for $d$ large enough,
defines an analytically equivalent singularity, and the
Newton diagram of the new function is convenient.

\begin{block} \label{block:Neu_min}
In \cite{Neu_plumb}, Neumann showed that if $M$ is an oriented three
dimensional manifold
which can be represented by a
plumbing graph, \index{plumbing graph}
then there is a unique
\emph{minimal plumbing graph} \index{minimal graph}
representing $M$. If the
intersection matrix \index{intersection matrix}
associated with
the plumbing graph $G$ is negative definite, then minimality, in this sense,
means that $G$ contains no vertex $v$ with $\delta_v \leq 2$ and $E_v^2 = -1$.
A minimal representative can be obtained by blowing down $-1$ curves
whenever possible.
\end{block}

\begin{block} \label{block:Newton_min}
In \cite{Newton_nondeg}, Braun and N\'emethi provide generators for
an equivalence relation on the set of Newton diagrams. If two functions
with Newton nondegenerate principal part have equivalent diagrams, then
they define topologically equivalent singularities. In fact, they can be
connected by a topologically constant deformation.
The generators can be described as follows. Let $P\subset\N^3$
so that $f = \sum_{p\in P} a_p x^p$ defines an isolated singularity
for generic coefficients $(a_p)_{p\in P}$. Equivalently, we assume
that $P$ contains a point at distance at least $1$ from each coordinate axis,
see \cref{lem:isol}. If $p\in\N^3$ so that
$\Gamma(f) \subset \Gamma(f+x^p)$, then these diagrams are equivalent.
If the two diagrams are not equal, then there are two possibilities.
One is that there is a face $F\subset\Gamma(f+x^p)$ so that
$\Gamma(f+x^p) = \Gamma(f)\cup F$ and
$\Gamma(f)\cap F \subset\partial\Gamma(f)$. The other is that
there is a face $F\subset\Gamma(f+x^p)$ and $F\cap\Gamma(f)$ is a
(two dimensional) face in $\Gamma(f)$.
A
Newton diagram \index{Newton diagram}
is
\emph{minimal} \index{Newton diagram!minimal}
if it is a minimal element of its equivalence class with respect
to inclusion.
In general, minimal diagrams are not
convenient. \index{convenient}
They do, however, have the advantage that
if one applies
Oka's algorithm \index{Oka's algorithm}
on a minimal diagram as in \cref{block:oka_alg},
then the output is a
minimal plumbing graph. \index{minimal graph}
\end{block}

The following proposition essentially repeats some of the results of
\cite{Newton_nondeg}:

\begin{prop} \label{prop:min}
Let $f\in\O_{\C^3,0}$ have
Newton nondegenerate \index{Newton nondegenerate}
principal part, defining an
isolated singularity at $0$ with a
rational homology sphere \index{homology sphere!rational}
link, \index{link}
which is not an $A_n$ singularity. Let $G$ be the
resolution graph \index{resolution graph}
constructed
in \cref{block:oka_alg} from $\Gamma(f)$.
\begin{enumerate}
\item \label{it:min_corr}
There is a bijective correspondence between
nodes \index{nodes}
$n\in\Nd$ in $G$
and two dimensional faces $F_n\subset \Gamma(f)$ and for each $n\in\Nd$
there is a bijection between neighbours $u\in \V_n$ of $n$ and primitive
segments of the boundary $\partial F_n$ of $F_n$. In particular,
$\Vol_1(\partial F_n) = \delta_n$.

\item \label{it:min_min}
If $\Gamma(f)$ is a
minimal Newton diagram \index{minimal Newton diagram}
then $G$ is a
minimal plumbing graph. \index{minimal graph}
\end{enumerate}
\end{prop}
\begin{proof}
\ref{it:min_corr} follows directly from construction. For \ref{it:min_min},
however, one must prove that if $n\in\Nd$ and $n'\in\Nd^*\setminus \Nd$
with $t_{n,n'} \geq 1$, then $\alpha_{n,n'} > 1$. This is proved
in \cite{Newton_nondeg} Proposition 3.3.11.
\end{proof}

\subsection{Association of cycles and polyhedrons} \label{ss:cyc_polyt}

In this section we will describe two methods of associating a cycle
to a function. On the other hand, we will associate a Newton polyhedron
to any cycle which will allow us to use the geometry of the
Newton diagram to prove properties of the
computation sequences \index{computation sequence}
defined in \cref{s:seq}.

\begin{definition} \label{def:wtdiv}
Let $g\in\O_{\C^3,0}$ and denote by
$\bar g$ the corresponding element in $\O_{X,0} = \O_{\C^3,0}/(f)$.
\begin{itemize}
\item
For any $v\in\V^*$ let
$\wt_v(g) = \min_{p\in\supp(g)} \ell_v(p)$
if
$g\neq 0$, otherwise set $\wt_v(g) = \infty$. Further,
let $\wt(g) = \sum_{v\in\V} \wt_v(g) E_v \in L$.
\item
For any $v\in\V^*$ let $\wt_v(\bar g) = \max_{h\in\O_{\C^3,0}} \wt_v(g+hf)$
if $\bar g \neq 0$, otherwise let $\wt_v(\bar g) = \infty$.
Further, let
$\wt(\bar g) = \sum_{v\in\V} \wt_v(\bar g) E_v \in L$.
\nomenclature[wt]{$\wt(g)$}{Weight of $g$}
\item
For any $v\in\V$, let $\div_v$ be the divisorial valuation on $\O_{X,0}$
associated with the exceptional divisor $E_v$. Furthermore,
let $\div(g) = \div(\bar g) = \sum_{v\in\V} \div_v(g) E_v \in L$.
\nomenclature[divg]{$\div(g)$}{Divisorial valuation of $g$}
\end{itemize}
\end{definition}

\begin{rem} \label{rem:wtdiv}
\begin{list}{(\thedummy)}
{
  \usecounter{dummy}
  \setlength{\leftmargin}{0pt}
  \setlength{\itemsep}{0pt}
  \setlength{\itemindent}{4.5pt}
}

\item
To any $v\in\V$ there corresponds a component, say $D_v$,
of the exceptional divisor
of the modification of $\C^3$ inducing the resolution of $X$. 
Then $\wt_v$ is the divisorial valuation on $\O_{\C^3,0}$ associated
with $D_v$. However, $\wt$ and $\div$ are generally not the same on
$\O_{X,0}$, see \ref{lem:Eb_GZ}.

\item
The divisorial valuations $\div_v$ on $\O_{X,0}$ and $\wt_v$ 
on $\O_{\C^3,0}$, as well as the order functions $\wt_v$ on $\O_{X,0}$
have been considered by many authors, see e.g.
\cite{Ebeling_Gusein-Zade,Lemah,Nem_Poinc,Hamm}.

\item \label{it:rem_wtdiv_n1}
The above definitions are not restricted to the surface case. The weights
$\wt_n$ can be defined on $\O_{\C^{n+1},0}$ for any $n\geq 0$
using the linear functions $\ell_n$ for $n\in\Nd$ from 
\cref{rem:Oka}\cref{it:rem_Oka_n1}.
The divisorial filtrations $\div_n$ are defined similarly on
$\O_{X,0} = \O_{\C^{n+1},0}/(f)$ assuming $n\geq 2$.

\end{list}
\end{rem}

\begin{definition} \label{def:HGam}
Let $Z\in L$ and $v\in\V$.
Start by defining the associated hyperplane and halfspace
\[
\begin{split}
  H_v^=(Z)    &= \set{p\in\R^3}{\ell_v(p) =    m_v(Z)} \\
  H_v^\geq(Z) &= \set{p\in\R^3}{\ell_v(p) \geq m_v(Z)}.
\end{split}
\nomenclature[Hv=Z]{$H_v^=(Z)$}{Hyperplane associated with $v$ and $Z$}
\nomenclature[Hv=Z]{$H_v^\geq(Z)$}{Halfspace associated with $v$ and $Z$}
\]
Since $H_v^=(Z)$ only depends on the
number $m = m_v(Z)$, we also set $H_v^=(m) = H_v^=(Z)$ and
$H_v^\geq(m) = H_v^\geq(Z)$.
We define the
\emph{Newton polyhedron} \index{Newton polyhedron}
of $Z$ as
\[
  \Gamma_+(Z) = \R_{\geq0}^3 \cap \bigcap_{v\in\V} H_v^\geq (Z).
\nomenclature[GZ]{$\Gamma_+(Z)$}{Newton polyhedron associated with $Z$}
\]
The \emph{face} corresponding to a node $n\in\Nd$ is
\[
  F_n(Z) = \Gamma_+(Z) \cap H^=_n(Z).
\nomenclature[G+Z]{$F_n(Z)$}{Face associated with $n$ and $Z$}
\]
The
\emph{polygon} \index{polygon}
corresponding to $n\in\Nd$ is
\[
  \Fnb_n(Z) = H^=_n(Z) \cap \bigcap_{u\in\V_n} H^\geq_u(Z).
\nomenclature[Fnb]{$\Fnb_n(Z)$}{Polygon associated with $n$ and $Z$}
\]
\end{definition}

\begin{rem}
\begin{list}{(\thedummy)}
{
  \usecounter{dummy}
  \setlength{\leftmargin}{0pt}
  \setlength{\itemsep}{0pt}
  \setlength{\itemindent}{4.5pt}
}

\item
Note that for any $Z\in L$, the
Newton polyhedron \index{Newton polyhedron}
$\Gamma_+(Z)$ and
its faces $F_n(Z)$ are, by definition, subsets of the positive
octant $\R^3_{\geq 0}$. The
polygons \index{polygon}
$\Fnb_n(Z)$, however, may contain
points with negative coordinates.

\item
By \cref{rem:Oka}\cref{it:Oka_nodes} we have $\V_n = \V_n^*$ for any
$n\in\Nd$. Therefore, $\Fnb_n(Z)$ is always a finite
polygon \index{polygon} \index{polygon}
(or empty).

\item
The $F_n(Z)$ is not necessarily a polygon, it can consist of a segment,
as single
point or be empty. It is in any case a bounded subset of an affine plane given
by finitely many inequalities.

\end{list}
\end{rem}

We finish this subsection by a well known formula for the
anticanonical cycle \index{anticanonical cycle}
$Z_K$.

\begin{prop}[Merle and Teissier \cite{Merle_Teissier} 2.1.1, Oka \cite{Oka} 9.1]\label{prop:Z_K}
We have
\[
  Z_K - E = \wt(f) - \wt(x_1 x_2 x_3).
\]
\qed
\end{prop}

\begin{cor}
We have $\Gamma_+(Z_K-E) = (\Gamma_+(f)-(1,1,1)) \cap \R^3_{\geq0}$.
\end{cor}

\subsection{The Newton filtration} \label{ss:Newt_filt}

In this subsection, we will assume that $f\in\O_{\C^{n+1},0}$ has Newton
nondegenerate principal part and defines an isolated singularity
$(X,0)\subset (\C^{n+1},0)$. We make the further assumption that
$\Gamma(f)$ has at least one face of dimension $n$.
Kouchnirenko defines a filtration
on $\O_{\C^{n+1},0}$ called the
\emph{Newton filtration} \index{Newton filtration}
\cite{Kouchnirenko}.
In this subsection we provide an equivalent definition.
With our definition, the filtration is indexed with rational numbers, whereas
Kouchnirenko's filtration is index by integers, see \cref{rem:Nfilt_index}.
Furthermore, in our discussion, more emphasis is laid on the filtration
of $\O_{X,0}$ than on that of $\O_{\C^{n+1},0}$.

We emphasize that all the results in this section assume that the Newton
diagram indeed has a face of dimension $n$. Note, however, that in the
case $n=2$, the diagram $\Gamma(f)$ contains a face unless the
singularity is of type $A_k$ for some $k\geq 1$ by
\cite{Newton_nondeg} Remark 2.1.5.

\begin{definition} \label{def:Newt}
Let $f\in\O_{\C^{n+1},0}$ and assume that the
Newton diagram \index{Newton diagram}
$\Gamma(f)$ contains at least one compact face of dimension $n$.
Recall the family $(\ell_n)_{n\in\Nd}$ of linear support
functions of compact faces of $\Gamma(f)$ as in
\cref{block:oka_alg} (see also \cref{rem:Oka}\cref{it:rem_Oka_n1})
and the weights $\wt_n$ in \cref{def:wtdiv} (see also
\cref{rem:wtdiv}\cref{it:rem_wtdiv_n1}).
For $p\in\R_{\geq 0}^{n+1}$, we set
$\ell_f(p) = \min_{n\in\Nd} \ell_n(p)/\wt_n(f)$.
\nomenclature[lf]{$\ell_f$}{Newton weight function}
If $\Gamma(f)$ is a convenient diagram, then
$\ell_f$ is the unique function on $\R^{n+1}_{\geq 0}$ which
restricts to a linear function on each ray and takes the value $1$ on
$\Gamma(f)$. In general, $\ell_f$ is the largest such function which
is concave.
For any $0\neq g\in\O_{\C^{n+1},0}$
 representing $\overline g\in\O_{X,0}$, we set
\[
  \ell_f(          g) = \min_{p\in\supp(g)} \ell_f(p),\quad
  \ell_f(\overline g) = \max_{h\in(f)}      \ell_f(g+h).
\nomenclature[l]{$\ell_f$}{Newton weight function}
\] 
We refer to $\ell_f$ as the
\emph{Newton weight function}.  \index{Newton weight function}
This yields the
\emph{Newton filtrations} \index{Newton filtration}
on $\O_{\C^{n+1},0}$ and $\O_{X,0}$
\[
  \A_{\C^{n+1}}(r) = \set{g\in\O_{\C^{n+1},0}}{\ell_f(g) \geq r}, \quad
  \A_{   X}(r) = \set{\overline g\in\O_{   X,0}}{\ell_f(\overline g) \geq r}
\nomenclature[ACrf]{$\A_{\C^{n+1}}(r)$}{Newton filtration on $\O_{\C^{n+1},0}$}
\nomenclature[AXrf]{$\A_X(r)$}{Newton filtration on $\O_{X,0}$}
\]
for $r\in\Q$ and the associated graded rings
\[
\begin{split}
  A_{\C^{n+1}} = \oplus_r A_{\C^{n+1},r}, &\quad 
      A_{\C^{n+1},r} = \A_{\C^{n+1}}(r) / \cup_{s>r} \A_{\C^{n+1}}(s), \\
  A_X      = \oplus_r A_{X   ,r}, &\quad 
      A_{X   ,r} = \A_{X}(r)    / \cup_{s>r} \A_{X}(s).
\end{split}
\nomenclature[ACr]{$A_{\C^{n+1}}$}{Graded associated with Newton filtration}
\nomenclature[AXr]{$A_X$}{Graded associated with Newton filtration}
\]
The associated
\emph{Poincar\'e series} \index{Poincar\'e series}
are given as
\[
  P^\A_{\C^{n+1}}(t) = \sum_{r\in\Q} \dim_\C A_{\C^{n+1},r} t^r, \quad
  P^\A_X(t) = \sum_{r\in\Q} \dim_\C A_{X,r} t^r.
\nomenclature[PAC]{$P^\A_{\C^{n+1}}(t)$}{Poincar\'e series associated
with Newton filtration}
\]
\nomenclature[PAX]{$P^\A_X(t)$}{Poincar\'e series associated
with Newton filtration}
\end{definition}

\begin{rem} \label{rem:Nfilt_index}
It follows that there is an $m\in\Z$ so that if $A_{\C^{n+1},r} \neq 0$ then
$r\in\frac{1}{m}\N$.
In \cite{Kouchnirenko}, the Newton filtration is normalized in such a way
that for $r$ big, we have $A_{\C^{n+1},r} \neq 0$ if and only if $r\in\N$.
For the rest of this subsection, we will fix such a value
for $m$ and use it in proofs.
\end{rem}

\begin{lemma} \label{lem:PNseries}
We have $P^\A_{\C^{n+1}}(t) = \sum_{p\in\N^{n+1}} t^{\ell_f(x^p)}$ and
$P^\A_X(t) = (1-t) P^\A_{\C^{n+1}}(t)$.
\end{lemma}
\begin{proof}
The first statement follows from the fact that for any $r\in\Q$,
the residue classes of the monomials $x^p$ with $\ell_f(p) = r$ form
a basis of $A_{\C^{n+1},r}$.

For the second statement, note first that for any $g\in\O_{\C^{n+1},0}$ we have
\[
  \ell_f(fg) = \ell_f(g) + 1.
\]
To see this, note first that by the convexity of $\Gamma_+(f)$,
the function $\ell_f$ is concave, that is,
$\ell_f(p_1+p_2) \geq \ell_f(p_1) + \ell_f(p_2)$ for
$p_1, p_2 \in \R_{\geq 0}^{n+1}$. Furthermore, we have $\ell_f(p) \geq 1$ for
$p\in\supp(f)$ by construction. This shows $\ell_f(gf) \geq \ell_f(g) + 1$.
To prove equality, take $n\in\Nd$
so that $\ell_f(g) = \wt_n(p_1)/\wt_n(f)$. Then
$\wt_n(gf) = \wt_n(g) + \wt_n(f)$, giving
$\ell_f(gf) \leq \wt_n(gf)/\wt_n(f) =  \ell_f(g)+1$.

This shows that for any $r\in\Q$ we have a sequence
\[
\xymatrix
{
    0                   \ar[r] 
  & A_{\C^{n+1},r-1 }   \ar[r]^{\cdot f} 
  & A_{\C^{n+1},r   }   \ar[r]
  & A_{X,r          }   \ar[r] 
  & 0
}
\]
whose exactness proves the lemma. The rows of the diagram
\[
\xymatrix{
  & 0                                         \ar[d]
  & 0                                         \ar[d]
  & 0                                         \ar[d]
  &                                                      \\
    0                                         \ar[r]
  & \A_{\C^{n+1}}\left(r-1+\frac{1}{m}\right) \ar[r] \ar[d]
  & \A_{\C^{n+1}}\left(r-1            \right) \ar[r] \ar[d]
  & A_{\C^{n+1},r-1}                          \ar[r] \ar[d]
  & 0                                                    \\
    0                                         \ar[r]
  & \A_{\C^{n+1}}\left(r  +\frac{1}{m}\right) \ar[r] \ar[d]
  & \A_{\C^{n+1}}\left(r              \right) \ar[r] \ar[d]
  & A_{\C^{n+1},r}                            \ar[r] \ar[d]
  & 0                                                    \\
    0                                         \ar[r]
  & \A_X     \left(r  +\frac{1}{m}\right)     \ar[r] \ar[d]
  & \A_X     \left(r              \right)     \ar[r] \ar[d]
  & A_{X,r}                                   \ar[r] \ar[d]
  & 0                                                    \\
  & 0                                    
  & 0                                    
  & 0                                    
  &                                                  
}
\]
are exact and the columns are complexes. Furthermore, the first two columns
are exact. The exactness of the third column now follows from the long exact
sequence associated with a short exact sequence of complexes.
\end{proof}

\begin{thm} \label{thm:spec_pol_part}
Let $f\in\O_{\C^{n+1},0}$ have Newton nondegenerate
principal part, defining an isolated singularity $(X,0)$.
Assume further that $\Gamma(f)$ has at least one face of dimension $n$.
Identify the
spectrum \index{spectrum}
(see \cref{ss:spec}) with its image under the canonical
isomorphism $\Z[t^{\pm1/\infty}] \cong \Z[\Q]$.
The
polynomial part \index{polynomial part}
of the
Poincar\'e series \index{Poincar\'e series}
associated with the
Newton filtration \index{Newton filtration}
then recovers the part $\Sp_{\leq0}(f,0)$ of the spectrum
via the formula
\[
  \Sp_{\leq 0}(f,0) = P^{\A,\mathrm{pol}}_X(t^{-1}).
\] \index{polynomial part}
\end{thm}
\begin{proof}
First assume that $\Gamma(f)$ is
convenient. \index{convenient}
By \cref{prop:Saito_exp}, it is enough to prove
\begin{equation} \label{eq:pf:spec_pol_part}
  P^{\A,\mathrm{pol}}_X(t) = \sum_{p\in\Z_{>0}^{n+1}\cap \Gamma_-(f)}
                             t^{1-\ell_f(x^p)}.
\end{equation}
For any face $\sigma$, let
$P_\sigma(t) = \sum_{p\in\R_{\geq 0}\sigma\cap\Z^{n+1}} t^{\ell_f(p)}$.
Let $\Sigma$ be the set of faces not contained in
a coordinate hyperplane.  Then,
similarly as in \cite{Kouchnirenko}, using \cref{lem:PNseries}, we find
\begin{equation} \label{eq:PA_cones}
  P^\A_{\C^{n+1}}(t) = \sum_\sigma (-1)^{\dim \sigma} P_\sigma(t)
\end{equation}
hence
\begin{equation} \label{eq:PAX_cones}
  P^\A_X(t) = (1-t)\left(\sum_\sigma (-1)^{\dim \sigma} P_\sigma(t)\right).
\end{equation}
If $\sigma\in\Sigma$, then the function $\ell_f$ restricts to a
linear function on any cone $\R_{\geq 0}\sigma$ satisfying the
conditions in \cref{lem:cone_pol_part}. Therefore,
\cref{eq:PAX_cones} gives \cref{eq:pf:spec_pol_part}.

Assume now that $\Gamma(f)$ is not
convenient. \index{convenient}
For any $d\in\N$, set $f_d = f + \sum_{i=1}^{n+1} x_i^d$.
Then, if $d$ is large enough, then
$f$ and $f_d$ are analytically equivalent and
$\Gamma(f) \subset \Gamma(f_d)$.
Then $\Sp_{\leq 0}(f,0) = \Sp(f',0)$, so if $(X',0)$ is the germ
defined by $f'$, setting  it suffices to show
that
\[
  \left(
    P_X^\A(t) - P_{X'}^\A(t)
  \right)^{\mathrm{pol}}
  = 0
\]
We consider $\Gamma(f)$ as a subcomplex of $\Gamma(f_d)$. We define
$\Sigma$ and $\Sigma_d$ as the set of cells contaned in $\Gamma(f)$ and
$\Gamma(f_d)$ respectively, but not in a coordinate hyperplane
Also, let
$\overline\Sigma$ be the set of all cells in $\Gamma(f)$.
Similarly as above, for any face $\sigma$ we define
\[
  P _\sigma(t) = \sum_{p\in\R_{\geq 0}\sigma\cap\Z^{n+1}} t^{\ell_f(p)}, \quad
  P_{\sigma,d}(t) = \sum_{p\in\R_{\geq 0}\sigma\cap\Z^{n+1}} t^{\ell_{f'}(p)}.
\]
Since $\ell_f$ and $\ell_{f'}$ coincide on any face contained in $\Gamma(f)$,
we get
\[
  P_X^\A(t) - P_{X'}^\A(t)
  =
  \sum_{\sigma\in\Sigma'\setminus\overline\Sigma}
  P_\sigma(t)^{\phantom{f}} - P'_\sigma(t).
\]
It is therefore enough to prove that for any
$\sigma\in\Sigma'\setminus\overline\Sigma$ we have 
$P_\sigma^{\mathrm{pol}}(t) = 0 = (P'_\sigma(t))^{\mathrm{pol}}$.
By \cref{lem:cone_pol_part} it suffices to show that
\begin{equation} \label{eq:no_pts}
  \set{p\in\R_{>0}\sigma^\circ\cap\Z_{>0}^{n+1}}
      {\ell_{f'}(p) \leq 1}
  = \emptyset,
\end{equation}
where $\sigma^\circ$ is the relative interior of $\sigma$ (note that since
$\ell_{f'} \geq \ell_f$, the corresponding statement for $\ell_f$ follows).
To prove \cref{eq:no_pts}, assume that the set on the left hand side contains
a point $p$. By construction, there exists a linear support function
$\ell$ of $\Gamma_+(f)$ so that the minimal
set of $\ell$ on $\Gamma_+(f)$ is a noncompact face and
$0<\ell(q) < m$ where $m = \min_{\Gamma_+(f)} \ell$, for any
$q \in \sigma^\circ$. It follows that $\ell(p) < m$.
Furthermore, if $e_1, \ldots, e_{n+1}$ is the standard basis of $\R^{n+1}$, 
then there is an $i$ so that $\ell(e_i) = 0$, since otherwise, the minimal
set of $\ell$ on $\Gamma_+(f)$ would be compact.
This implies that for all
$k\in\N$, we have $p + k e_i \notin \Gamma_+(f)$.
Since $\Gamma_+(f) = \cap_{d=0}^{\infty} \Gamma_+(f_d)$, this means that for
any $k$ there exists a $d$ so that $p+ke_i\in\Gamma_-(f_d)$, thus,
$\lim_{d\to\infty} P_{X_d}^{\A,\mathrm{pol}}(1) = \infty$.
But since the analytic type of $(X_d,0)$ is constant for $d$ large, the
number $P_{X_d}^{\A,\mathrm{pol}}(1) = \Sp_{\geq 0}(f_d,0)(1)$
is also constant for $d$ large. This finishes the proof.
\end{proof}

\subsection{The anatomy of Newton diagrams} \label{ss:anatomy}

In this subsection we will recall some classification results of Braun
and N\'emethi \cite{Newton_nondeg} which will serve as basis for the
case-by-case analysis in \cref{s:SW}. We will also fix some notation.
We assume that $f\in\O_{\C^3,0}$ has Newton nondegenerate principal
part and defines an isolated singularity $(X,0)$. The graph $G$ with
vertex set $\V$ and the subset $\Nd\subset\V$ are
defined as in \cref{block:oka_alg}. The set $\Nd$ then consist of the
nodes of the graph $G$.

\begin{definition} \label{def:legs}
Let $n\in \Nd$. A
\emph{leg} \index{leg}
of $n$ is a sequence of vertices 
$v_1, \ldots, v_s \in \V$ so that for $j=1,\ldots,s-1$ we have
$\V_{v_j} = \{v_{j-1}, v_{j+1}\}$ where we set $v_0 = n$ and
$\delta_{v_s} = 1$. In this case, $v_s$ is called the
\emph{end} \index{end}
of the leg.
The set of all ends of legs of $n$ is denoted by
$\E_n$, \nomenclature[En]{$\E_n$}{Ends of a node $n$}
and we set
$\E = \cup_{n\in\Nd} \E_n$. \nomenclature[Ef]{$\E$}{Ends}
We say that the vertices in $\E$ are the \emph{ends} of the graph $G$.
If $e\in \E$, then there are unique $n_e\in \Nd$ and
$n^*_e\in\Nd^*\setminus\Nd$ so that $e\in\E_n$ and $e$ lies on the
bamboo \index{bamboo}
connecting $n_e$ and $n_e^*$.  For $e = v_s \in\E$ as above, define
$\alpha_e/\beta_e = [b_{v_1}, \ldots, b_{v_s}]$ as
the \emph{fraction} of $e$, where $\alpha_e, \beta_e \in \N$ and
$\gcd(\alpha_e, \beta_e) = 1$. Thus,
$\alpha_e = \alpha_{n^{\vphantom{*}}_e, n^*_e}$
\nomenclature[ae]{$\alpha_e$}{$\alpha_{n^{\vphantom{*}}_e, n^*_e}$}
and
$\beta_e = \beta_{n^{\vphantom{*}}_e, n^*_e}$.
\nomenclature[be]{$\beta_e$}{$\beta_{n^{\vphantom{*}}_e, n^*_e}$}
Define also 
$u_e = u_{n^{\vphantom{*}}_e,n_e^*}$.
\nomenclature[ue]{$u_e$}{Neighbour of node on leg containing end $e$}
A
\emph{leg group} \index{leg group}
is a maximal nonempty set of legs of a fixed $n\in \Nd$
for which the ratio $\alpha_e/\beta_e$
is fixed, where $e$ is the end of the leg.
\end{definition}

\begin{example}
The graph in \cref{fig:ex_graph} has four nodes, one of which has no
legs, the others have two legs each. The decorations of the legs
$\alpha_e/\beta_e$ are $2/1, 2/1$; $4/3, 3/1$ and $2/1, 3/1$.
\end{example}

\begin{definition} \label{def:central} 
\begin{itemize}
\item
A two dimensional triangular face of $\Gamma(f)$ is called a
\emph{central triangle} \index{central triangle}
if it intersects all three coordinate hyperplanes, but none of the coordinate
axis. The corresponding node is called a
\emph{central node}.
\item
A
\emph{trapezoid} \index{trapezoid}
in $\Gamma(f)$ is a face whose vertices (modulo permutation)
are of the form
$(0,p,a)$, $(q,0,a)$, $(r_1,r_2,0)$, $(r_1', r_2',0)$ where
$(r_1',r_2',0) - (r_1, r_2,0) = k(-q,p,0)$ for some $k>0$.
\item
An edge in $\Gamma$ (a one dimensional face) is called a
\emph{central edge} \index{central edge}
if it intersects all three coordinate hyperplanes.
\end{itemize}

A
\emph{central face} \index{central face}
is a central triangle or a trapezoid.
\end{definition}

\begin{definition} \label{def:arm}
The collection of faces of $\Gamma(f)$ (of positive dimension)
whose vertices lie
on the union of two of the coordinate hyperplanes is called an
\emph{arm}. \index{arm}
If the intersection of the two planes is the $x_i$ axis, then we say that
the arm \emph{goes in the direction of the $x_i$ axis}. An arm is
\emph{degenerate} \index{arm!degenerate}
if it does not contain a two dimensional face.
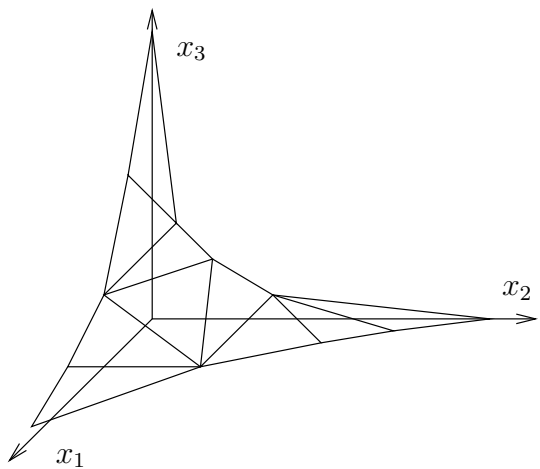
\mynd{arms}{ht}{A Newton diagram with arms consisting of 2, 4, 3 faces in
the direction of the $x_1$, $x_2$, $x_3$ axis.}{fig:arms}
\end{definition}

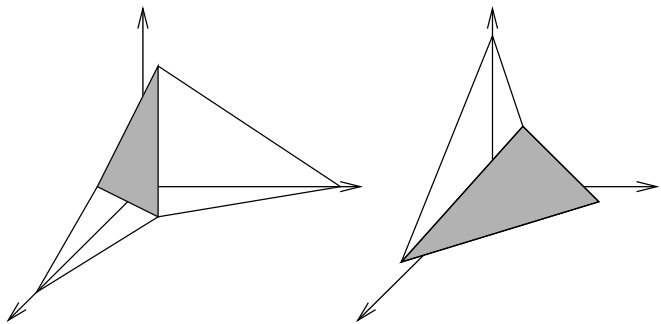
\mynd{class}{ht}{Examples of Newton diagrams with a central triangle.
The first one has two nondegenerate arms, the second has only one.}{class}
\begin{prop}[Braun and N\'emethi \cite{Newton_nondeg} Proposition 2.3.9]
  \label{prop:anatomy}
Let $f\in \O_{\C^3,0}$ be a function germ with
Newton nondegenerate \index{Newton nondegenerate}
principal
part, defining an isolated singularity $(X,0)$ with a
rational homology sphere \index{homology sphere!rational}
link. \index{link}
Then exactly one of the following hold (see
\ref{block:oka_alg} for definition of $t_{n,n'}$):
\begin{enumerate}
\item \label{it:anatomy_central_face}
$\Gamma(f)$ has a
central face \index{central face}
and three
(possibly degenerate) arms. \index{arm}
We have $\Nd = \cup_{\kappa=1}^3 \{ n_0^\kappa, \ldots,  n_{j^\kappa}^\kappa\}$
where $n_0^1 = n_0^2 = n_0^3$ is the central face
and the arm in the direction of the $x_\kappa$ axis is
$F_{n_1^{(\kappa)}} \cup \ldots \cup F_{n_{j^{(\kappa)}}^{(\kappa)}}$
in the nondegenerate
case, or the corresponding edge of $F_{n_0}$ in the degenerate case.
We have
$t_{n_{r\vphantom{'}}^{\kappa\vphantom{'}}, n_{r'\vphantom{'}}^{\kappa'}} = 1$
if $\{r,r'\} = \{0,1\}$
or $\kappa=\kappa'$ and $|r-r'| = 1$ and
$t_{n_{r\vphantom{'}}^{\kappa\vphantom{'}}, n_{r'\vphantom{'}}^{\kappa'}} = 0$
otherwise (recall definiton of $t_{n,n'}$ in \cref{block:oka_alg}).
\item \label{it:anatomy_central_edge}
There exist $c>0$
central edges. \index{central edge}
We have
$\Nd = \cup_{\kappa=1}^2 \{ n_1^\kappa, \ldots,  n_{j^\kappa}^\kappa \}$ with
$n_r^1 = n_{c-r}^2$ if $1\leq r\leq c-1$.
Further, we have
$t_{n_{r\vphantom{'}}^{\kappa\vphantom{'}}, n_{r'\vphantom{'}}^{\kappa'}} = 1$
if $\kappa = \kappa'$ and
$|r-r'| = 1$ or $\kappa\neq \kappa'$ and $|r - (c-r')| = 1$ and
$t_{n_{r\vphantom{'}}^{\kappa\vphantom{'}}, n_{r'\vphantom{'}}^{\kappa'}} = 0$
otherwise.
\end{enumerate}
In case \ref{it:anatomy_central_edge}, if $j_\kappa \geq c$, we set
$n^{\kappa'}_0 = n^\kappa_c$, where $\{\kappa,\kappa'\} = \{1,2\}$.
\end{prop}
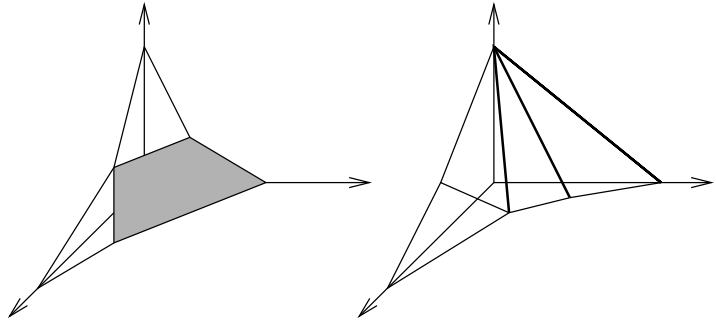
\mynd{classct}{ht}{Examples of Newton diagrams, one with a central
trapezoid, the other with three central edges.}{classct}

\begin{prop}[Braun and N\'emethi] \label{prop:arm_hands}
Let $n_r = n_r^\kappa$, $r=1, \ldots, j = j^\kappa$ be an arm as in
\cref{prop:anatomy},
\ref{it:anatomy_central_face} or \ref{it:anatomy_central_edge}.
Assume that the arm goes in the direction of $x_3$.
\begin{enumerate}
\item
For any $1\leq r < j$, the numbers $\alpha_e, \beta_e$ are independent of
the choice of $e\in \E_{n_r}$. Furthermore, we have either
$\ell_{n_e^*} = \ell_1$ for all $e\in\E_{n_r}$, or 
$\ell_{n_e^*} = \ell_2$ for all $e\in\E_{n_r}$.
That is, $n_r$ has a unique leg group. \index{leg group}

\item \label{it:arm_hands_finger}
There are two distinct integral functions
$\tilde\ell_1, \tilde\ell_2 : \R^3 \to \R$
so that $\set{\ell_{n^*_e}}{e\in\E_{n_j}} = \{ \tilde\ell_1, \tilde\ell_2 \}$.
After possibly permuting the coordinates $x_1, x_2$, we have
$\tilde\ell_1 = \ell_1$ and either $\tilde\ell_2 = \ell_2$, or there is
an $a\in\Z_{\geq 0}$ so that $\tilde\ell_2 = a\ell_2 + \ell_1$.

\item
With $\tilde\ell_1, \tilde\ell_2$ as above, set
$\E_{n_j}^\lambda = \set{e\in\E_{n_j}}{\ell_{n^*_\lambda} = \tilde\ell_e}$
\nomenclature[Efnj]{$\E_{n_j}^\lambda$}{Subset of $\E_{n_j}$}
for $\lambda=1,2$.
We then have
integers $\alpha_\lambda, \beta_\lambda$ for $\lambda=1,2$ so that
$\alpha_e = \alpha_\lambda$
and $\beta_e = \beta_\lambda$ for $e\in\E_{n_j}^\lambda$.
That is, $n_r$ has exactly two leg groups.
Furthermore, if
$\tilde\ell_2 = \ell_2$, then $\gcd(\alpha_1,\alpha_2) = 1$, but
if $\tilde\ell_2 = a\ell_2 + \ell_1$, then $\alpha_1 | \alpha_2$.
\end{enumerate}
\end{prop}
\begin{proof}
This follows from Lemma 2.3.5 and 5.2.5 of \cite{Newton_nondeg}.
\end{proof}

\begin{example}
The diagram in \cref{fig:arms} has three nondegenerate arms. Consider
the arm going in the direction of the $x_\kappa$ axis.
The functions $\tilde \ell_1, \tilde \ell_2$ in
\cref{prop:arm_hands}\ref{it:arm_hands_finger}
are, given in coordinates,
$(0,1,0)$ and $(0,1,a)$ for some $a>0$ in the case $\kappa=1$,
$(1,0,0)$ and $(0,0,1)$ in the case $\kappa = 2$ and
$(1,0,0)$ and $(0,1,0)$ in the case $\kappa = 3$.
\end{example}

\newpage
\section{Two dimensional real affine geometry} \label{s:affine}

In this section we describe some technical results about
polygons in affine spaces.
If $H\subset \R^3$ is a hyperplane given by an affine equation with integral
coefficients so that $H\cap\Z^3\neq \emptyset$, then there exists an affine
isomorphism $H\to \R^2$, restricting to an isomorphism
$H\cap \Z^3 \to \Z^2$. When dealing with such hyperplanes in $\R^3$, we 
implicitly assume such an identification given, which allows us to apply
results obtained in $\R^2$.

\subsection{General theory and classification}

\begin{definition}
An
\emph{integral polygon} \index{integral polygon}
$F$ is the convex hull $\convx(P)$ of
a finite set of integral points spanning $\R^2$ as an affine space.
A
\emph{vertex} \index{integral polygon!vertex of}
of $F$
is an element $p\in P$ so that $\convx(P\setminus\{p\}) \neq F$.
An
\emph{edge} \index{integral polygon!edge of}
of $F$ is a segment contained in the boundary of $F$
whose endpoints are vertices.
\end{definition}

\begin{definition} \label{def:reg_sing}
A
\emph{regular vertex} \index{regular vertex}
$p$ of $F$ is a vertex having the property that
primitive vectors parallel to the two boundary segments having $p$ as
an endpoint form an integral basis of $\R^2$. A vertex which is not
regular is called
\emph{singular}. \index{singular vertex}
\end{definition}

The boundary $\partial F$ and relative interior $F^\circ$ of a polygon $F$ will
have their usual meaning and an \emph{internal point} of $F$ is nothing
but an element of $F^\circ$. By an \emph{integral affine isomorphism}
we mean an $\R$-affine automorphism $\R^2 \to \R^2$ restricting to a
$\Z$-affine automorphism $\Z^2 \to \Z^2$.

\begin{definition}
An integral polygon $F\subset \R^2$ is
\emph{empty} \index{integral polygon!empty}
if
$F^\circ \cap \Z^2 = \emptyset$.
\end{definition}

\begin{example} \label{ex:Newton_poly}
If $F_n \subset \Gamma(f)$ as above where $f$ has
Newton nondegenerate \index{Newton nondegenerate}
principal part and defines an isolated singularity with a
rational homology sphere \index{homology sphere!rational}
link, \index{link}
then $F_n$ is an
empty polygon \index{integral polygon!empty}
in the hyperplane
$H^=_n(\wt (f))$ by \cref{prop:isol}.
\end{example}

The simple proof of the following classification result is left to
the reader. For this, \cref{fig:polygons} may be of help.

\begin{prop} \label{prop:poly_class}
Let $F\subset \R^2$ be an
empty integral polygon. \index{integral polygon!empty}
Then, after, perhaps, applying an integral
affine isomorphism on $\R^2$, one of the following holds.
\begin{itemize}

\item
\emph{Big triangle:} We have $F = \convx \{ (0,0), (2,0), (0,2) \}$.
\index{integral polygon!big triangle}

\item
\emph{$t$-triangle:} We have $F = \convx \{ (0,0), (t,0), (0,1) \}$ for some
$t\geq 0$.
\index{integral polygon!t-triangle@$t$-triangle}

\item
\emph{$t$-trapezoid:} We have
$F = \convx \{ (0,0), (t,0), (0,1), (1,1) \}$ for some $t\geq 0$.
\index{integral polygon!t-trapezoid@$t$-trapezoid}

\item
\emph{$(t,s)$-trapezoid:} We have
$F = \convx \{ (0,0), (t,0), (0,1), (s,1) \}$ for some $t\geq s>1$.
\index{integral polygon!ts-trapezoid@$(t,s)$-trapezoid}
\qed
\end{itemize}
\end{prop}
\mynd{polygons}{ht}{Empty polygons}{fig:polygons}

\begin{definition} \label{def:top_edge}
If F is a $t$-trapezoid as above with $t>1$, then the edge $[(0,1),(1,1)]$
is called the
\emph{top edge}. \index{top edge}
This edge can be identified independently of
coordinates as the unique edge of length one, whose adjacent edges both
have length one.
\end{definition}

\begin{cor} \label{cor:reg_vx}
If $F\subset \R^2$ is an
empty polygon \index{integral polygon!empty}
and $p\in F$ is a
singular vertex, \index{singular vertex}
then $F$ is a $t$-triangle with $t>1$, and assuming $F$ is
of the form given in \cref{prop:poly_class}, we have $p = (0,1)$.
Equivalently, $F$ is a triangle and the opposing edge to $p$ is not
primitive. \qed
\end{cor}

\begin{example}
\begin{list}{(\thedummy)}
{
  \usecounter{dummy}
  \setlength{\leftmargin}{0pt}
  \setlength{\itemsep}{0pt}
  \setlength{\itemindent}{4.5pt}
}

\item
An exercise shows that the only
Newton diagrams \index{Newton diagram}
as in
\cref{ex:Newton_poly} containing
big triangles \index{integral polygon!big triangle}
are
$\Gamma(x_1^{2a} + x_2^{2b} + x_3^{2c})$ where $a,b,c$ are pairwise 
coprime positive integers.

\item
Similarly, a Newton diagram as in \cref{ex:Newton_poly} cannot contain
a $(t,s)$-trapezoid.
In fact, in \cite{Newton_nondeg}, Braun and N\'emethi show that such
a diagram can contain at most one $t$-trapezoid.

\end{list}
\end{example}

\begin{definition}
Let $F \in \R^2$ be an integral polygon and $S\subset F$ an edge. The
unique primitive integral affine function
$\ell_S:\R^2\to\R$
\nomenclature[lS]{$\ell_S$}{Support function of an integral polygon}
satisfying
$\ell_S|_S \equiv 0$ and $\ell_S|_F\geq 0$ is called the \emph{support function}
of $S$ with respect to $F$.

More generally, if $r\in\R_+$, then we have the dilated polygon $rF$ which
is not
necessarily integral, but the term \emph{edge} retains its meaning.
The support function of an edge $S\subset rF$ is the unique primitive integral
affine function $\ell_S:\R^2\to \R$ satisfying
\nomenclature[lrS]{$\ell_{rS}$}{Support function of a dilated polygon}
$\ell_S|_S \equiv m_S \in ]-1,0]$ and $\ell_S|_{rF} \geq m_S$.
\end{definition}

\begin{lemma}
Let $p,q,r\in F$ be vertices of an empty polygon so that the segments
$[p,q]$ and $[q,r]$ are edges of $F$, $[p,q]$ is primitive and $q$ is regular.
Then $\ell_{[q,r]}(p) = 1$.
\end{lemma}
\begin{proof}
This follows more or less from the definition.
\end{proof}

\subsection{Counting lattice points in dilated polygons}

\begin{definition}
Let $F$ be an empty integral polygon with an edge
$S = [p,q]\subset \partial F$. The 
\emph{content} \index{content}
$c_S$ \nomenclature[cS]{$c_S$}{Content of segment}
of $S$ is the content of the vector $q-p$.
\end{definition}

\begin{thm} \label{thm:point_count}
Let $F\subset \R^2$ be an empty integral polygon and
$r\in\R_+$. Furthermore, for any edge $S\subset \partial F$ with
$\ell_{rS}|_{rS} \equiv 0$, choose $\epsilon_S \in \{0,1\}$, for other
edges let $\epsilon_S = 0$. Let $F^- = F \setminus \cup_{\epsilon_S = 1} S$.
Then, there is a number $c_{rF^-} \in \Z$ satisfying
$\sum_{S\subset\partial F} c_{S} (\ell_{rS} - \epsilon_S) \equiv c_{rF^-}$
and
\[
  \max \{ 0, c_{rF^-} + 1 \} =
\begin{cases}
  |rF^- \cap \Z^2|                        & \textrm{if}\,\, r   < 1, \\
  |rF^- \cap \Z^2| - |(r-1)F^- \cap \Z^2| & \textrm{if}\,\, r\geq 1.
\end{cases}
\]
\end{thm}
\begin{rem}
If we consider $\R^2$ as an abstract affine plane only (with
the affine lattice $\Z^2\subset\R^2$), then
the number $c_{rF^-}$ above depends on the polygon $rF^-$ and cannot be
determined from $F^-$ and $r$ alone unless one fixes an origin.
\end{rem}

\begin{definition}
We call the number $c_{rF^-}$ in the theorem above the
\emph{content} \index{content}
of the dilated polygon $rF^-$ with the \emph{boundary conditions}
$\epsilon_S$, $S\subset \partial F$.
\nomenclature[crF-]{$c_{rF^-}$}{Content of dilated polygon
with boundary conditions}
\end{definition}

\begin{proof}[Proof of \cref{thm:point_count}]
We start by showing that the sum 
$\sum_{S\subset F} c_{S} (\ell_{rS} - \epsilon_S)$ is a constant function.
Since the epsilons are already constant, it is enough to show that
$\sum_{S\subset F} c_{S} \ell_{rS}$ is constant, i.e., assume $\epsilon_S = 0$
for all edges $S$.  Furthermore, for any
$S$, the difference $\ell_S - \ell_{rS}$ is a constant (since the segments
$S$ and $rS$ are parallel), so we may assume
that $r = 1$. In the case when $F$ is a $1$-triangle, we have sides
$S_1, S_2, S_3$ and a simple exercise shows that
$\ell_{S_1}+\ell_{S_2}+\ell_{S_3} = 1$, hence, $\sum_S c_S \ell_S \equiv 1$.
If $F$ is any integral polygon, take an integral triangulation of $F$, that is,
write $F = \cup_k F_k$ where $F_k$ are $1$-triangles, and
$\dim(F_k \cap F_h) \leq 1$ for $k\neq h$. We then get
$\sum_{S\subset \partial F} c_S \ell_S
= \sum_k \sum_{S\subset \partial F_k} c_S \ell_S$ which is a constant
by the above result. Here, we have equality because in the second sum,
if $S = F_k\cap F_h$, then $\ell_S$ is counted twice, with opposite sign
and if $S \subset \partial F$ is a primitive boundary segment, then
$\ell_S$ is counted once.

We define $c_{rF-}$ as the value of this constant function.
We will prove the theorem in the cases of a $t$-triangle or a $t$-trapezoid.
One proves the theorem in the cases of a big triangle or a $(t,s)$-trapezoid
using similar methods.

We start with the case when $F$ is a $t$-triangle, and $\epsilon_S = 0$
for all edges $S$. Write
$\partial F = S_1\cup S_2 \cup S_3$, where the $S_k$ are edges and
$S_1$ has length $t$, thus $S_2$ and $S_3$ have length $1$.
If $r\geq1$, then $(r-1) F^- + p \subset r F^-$ and we have
\[
  |rF^-\cap\Z^2| - |(r-1)F^-\cap\Z^2| = 
  |rF^-\setminus ((r-1)F^-+p) \cap\Z^2|.
\]
Furthermore, we have
\begin{equation} \label{eq:point_count_pf}
  rF^-\setminus ((r-1)F^-+p) \cap\Z^2
  = \set{p\in rF^-\cap\Z^2}{\ell_{rS_1}(p) = 0}.
\end{equation}
Note that in the case when $r<1$, the set $rF^-\cap\Z^2$ is also given by
the right hand side above. Therefore, to prove the lemma, we must prove
\begin{equation} \label{eq:pc_pf}
  \max \{0, c_{rF^-}+1\}
  = \left|\set{p\in rF^-\cap\Z^2}{\ell_{rS_1}(p) = 0}\right|.
\end{equation}
Since the endpoints of the segment $S_1$ are
both regular vertices, the support functions $\ell_{rS_2}, \ell_{rS_3}$
restrict
to primitive functions $\ell_{rS_2}|_{L_1},\ell_{rS_3}|_{L_1}:L_1 \to \R$,
where
$L_1 = \set{p\in\R^2}{\ell_{rS_1}(p) = 0}$. Therefore, if the right hand side
of \cref{eq:pc_pf} is nonempty, then it is given as
$\{ p_0, \ldots, p_c\}$ where $\ell_{rS_2}(p_k) = k$ and
$\ell_{rS_3}(p_k) = c-k$. The result therefore
follows by evaluating the sum
$\sum_{S\subset\partial F} c_S\ell_{rS}$ at
the point $p_0$.
If the set is empty, then
there is a unique point $p_0\in L_1$ with $\ell_{rS_2}(p_0) = 0$,
and we must have $\ell_{rS_2}(p_0)<0$, hence the result.

Now, assuming that $\epsilon_{S_1} = 0$ and $\epsilon_{S_2} = 1$ or
$\epsilon_{S_3} = 1$, then the right hand side of
\cref{eq:point_count_pf} is given as
$\{p_{\epsilon_{S_2}}, \ldots, p_{c-\epsilon_{S_3}}\}$ and the result
is verified in the same way. If $\epsilon_{S_1} = 1$, then, instead of
\cref{eq:point_count_pf}, we have
\begin{equation} \label{eq:point_count_pf_1}
  rF^-\setminus ((r-1)F^-+p) \cap\Z^2 = \set{p\in rF^-\cap\Z^2}{\ell_1(p) = 1}.
\end{equation}
If this set is not empty, then it is given as
$\{p_{\epsilon_{S_2}}, \ldots, p_{c-\epsilon_{S_3}}\}$ where
$\ell_{rS_2}(p_k) = k$ and $\ell_{rS_3}(p_k) = c-k$, hence
$(t(\ell_{rS_1}-1) + (\ell_{rS_2}-\epsilon_{S_2})
+ (\ell_{rS_3}-\epsilon_{S_3}))(p_{\epsilon_{S_2}})
= \ell_{rS_3}(p_0) - \epsilon_{S_3} = c - \epsilon_{S_2} - \epsilon_{S_3}
=|\{p_{\epsilon_{S_2}}, \ldots, p_{c-\epsilon_{S_3}}\}|-1$.
The result follows in a similar way as above if \cref{eq:point_count_pf_1}
is empty.

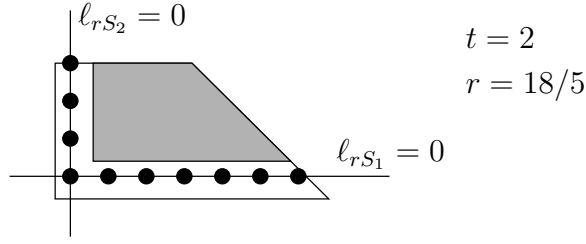
\mynd{points}{ht}{Counting points in $rF^-\setminus ((r-1)F^- + p)$
when $F$ is a trapezoid.}{fig:points}
The lemma is proved using a similar method if $F$ is a $t$-trapezoid.
Assuming this, write $\partial F = S_1\cup S_2\cup S_3\cup S_4$, where
the edge $S_1$ has length $t$, and $S_k$ and $S_{k+1}$ intersect in a vertex.
If $r\geq 1$, then
\[
  |rF^-\cap\Z^2| - |(r-1)F^-\cap\Z^2| = 
  |rF^-\setminus ((r-1)F^-+p) \cap\Z^2|.
\]
where $p$ is the intersection point of $S_3$ and $S_4$. We get
\begin{equation} \label{eq:pc_pf_trap}
  rF^-\setminus ((r-1)F^-+p) \cap\Z^2
  = \set{p\in rF^-\cap\Z^2}{\ell_{rS_1}(p) = \epsilon_{S_1}\,\textrm{or}\,
                            \ell_{rS_2}(p) = \epsilon_{S_2}}.
\end{equation}
We see then that the right hand side above is given as
$\{p_{\epsilon_{S_2}},\ldots, p_{c - \epsilon_{S_4}}\} \cup
\{p'_{\epsilon_{S_1}},\ldots, p'_{c'-\epsilon_{S_3}}\}$, where
$p_{\epsilon_{S_2}} = p'_{\epsilon_{S_1}}$ and
\[
\begin{array}{rclcrcl}
  \ell_{rS_1}(p _k) &=& \epsilon_{S_1},      &\quad&
  \ell_{rS_2}(p'_k) &=& \epsilon_{S_2},              \\
  \ell_{rS_2}(p _k) &=& k,                   &\quad&
  \ell_{rS_1}(p'_k) &=& k,                           \\
  \ell_{rS_4}(p _k) &=& c  - k,              &\quad&
  \ell_{rS_3}(p'_k) &=& c' - k.                      \\
\end{array}
\]
In particular, if we set $q = p_{\epsilon_{S_2}} = p'_{\epsilon_{S_1}}$, we
get
\[
\begin{array}{rclcrcl}
  \ell_{rS_1}(q) &=&      \epsilon_{S_1},        &\quad&
  \ell_{rS_2}(q) &=&      \epsilon_{S_2},                \\
  \ell_{rS_3}(q) &=& c' - \epsilon_{S_1},        &\quad&
  \ell_{rS_4}(q) &=& c  - \epsilon_{S_2}.                \\
\end{array}
\]
This gives
\[
\begin{split}
  \sum_{S\subset\partial F} c_S(\ell_S - \epsilon_S)(q)
    &= t(\epsilon_{S_1} - \epsilon_{S_1})
      + (\epsilon_{S_2} - \epsilon_{S_2})
      + (c' - \epsilon_{S_1} - \epsilon_{S_3})
      + (c  - \epsilon_{S_2} - \epsilon_{S_4})\\
    &= c - \epsilon_{S_2} - \epsilon_{S_4} +
	   c' - \epsilon_{S_1} - \epsilon_{S_3}+1.
\end{split}
\]
The right hand side above is the cardinality of the right hand side of
\cref{eq:pc_pf_trap}, if this set is nonempty, otherwise it is 
nonpositive. This finishes the proof.
\end{proof}

\newpage
\section{Construction of sequences} \label{s:seq}

In this section we will construct computation sequences for certain
cycles on the
resolution graph \index{resolution graph}
of
Newton nodegenerate \index{Newton nondegenerate}
surface singularities
described in \cref{ss:Okas_alg} and compare the intersection numbers on the
right hand side of \cref{eq:comp_seq} with a lattice point count
``under the diagram''. In \cref{s:pg} we will use these results to identify
the
geometric genus \index{geometric genus}
topologically and in \cref{s:SW} we make the same
identification of the
normalized Seiberg--Witten invariant
\index{Seiberg--Witten invariant!normalized}
of the canonical
$\spinc$ structure.

In \cref{ss:Laufer_seq} we give a technical result which essentially allows us
to work in a reduced lattice. These ideas are already present in
\cite{Nemethi_OzsSzInv,Lasz_Nem_Red,Nem_Rom_Sd}.
In \cref{ss:alg} we give an algorithm, which explicitly constructs the
computation sequences \index{computation sequence}
which we will consider.
In \cref{ss:pt_count} we compute some intersection numbers coming from these
computation sequences.

\subsection{Laufer sequences} \label{ss:Laufer_seq}

In this section, $L$ is the lattice associated with a
resolution graph \index{resolution graph} $G$
of a normal surfaces singularity with a rational homology sphere link as
described in \cref{ss:res}. As before, $\V$ is the set of vertices
in $G$ and let $\Nd$ be the set of nodes (this coincides with the
construction \cref{block:oka_alg}, see \cref{prop:min}\ref{it:min_corr}).
In applications of the results presented, $G$ will be the graph constructed
by
Oka's algorithm \index{Oka's algorithm}
in \cref{ss:Okas_alg}.
We will describe the Laufer operator $x$ on $L$ and the
associated
generalized Laufer sequences. \index{generalized Laufer sequence}
N\'emethi considers a similar operator
in \cite{Nemethi_OzsSzInv} in a specific case, and in 
\cite{Laszlo_th} L\'aszl\'o provides a general theory. Many of the proofs
in this subsection can be found in these sources. See also
\cite{Nem_Bal}.

For $Z\in L$, we define an element $x(Z)$. Under certain conditions, 
one may expect $Z\leq x(Z)$, in which case $x(Z)$ can be calculated
by the computation sequence constructed in \cref{prop:x_cseq}.
This sequence has the property that if it contains $Z_i$ and
$Z_{i+1} = Z_i+E_{v(i)}$, then $(Z_i,E_{v(i)}) > 0$.
These steps are therefore trivial, in the sense of
\cref{rem:subseq}.

\begin{prop} \label{prop:x_exist}
Let $Z\in L$. There exists a unique cycle
$x(Z)$ \nomenclature[xZ]{$x(Z)$}{Laufer operator}
satisfying the
following properties:
\begin{enumerate}
\item \label{it:x_exists_eq}
$m_n(x(Z)) = m_n(Z)$ for all $n\in\Nd$.
\item \label{it:x_exists_ineq}
$(x(Z),E_v) \leq 0$ for all $v\in\V\setminus\Nd$.
\item \label{it:x_exists_min}
$x(Z)$ is minimal with respect to the above conditions.
\end{enumerate}
\end{prop}

\begin{proof}
Let $\overline G = G \setminus \Nd$ be the subgraph of $G$ generated
by the vertex set $\V\setminus\Nd$. Finding an element $x(Z)$ satisfying
the above conditions is clearly equivalent to finding a minimal element
$Z_{\overline G}$ in the lattice $L_{\overline G}$ associated with $\overline G$
satisfying
\begin{enumerate}[(i')]
\stepcounter{enumi}
\item \label{it:x_exists}
For all $v\in \V \setminus \Nd$ we have
$(Z_{\overline G}, E_v)_{\overline G} \leq - \sum_{n\in\Nd \cap \V_v} m_n$.
\end{enumerate}
The existence of a minimal element satisfying (\ref{it:x_exists}'),
as well as its uniqueness, now follows in a similar way as that of
the
minimal cycle, \index{minimal cycle}
see \cref{def:Zmin}
\end{proof}

\begin{rem}
The above proposition and its proof hold if we replace $\Nd$ with any subset
of $\V$.
\end{rem}

\begin{prop} \label{prop:x_cseq}
If $Z\leq x(Z)$ then $x(Z)$ can be calculated
using a
computation sequence \index{computation sequence}
as follows. Start by setting $Z_0 = Z$.
Then, assuming that $Z_i$ has been defined, if we have
$(Z_i,E_v) \leq 0$ for all $v\in\V\setminus \Nd$, then
$Z_i = x(Z)$. Otherwise, there is a $v(i)$ so that
$(Z_i,E_{v(i)}) > 0$ and we define $Z_{i+1} = Z_i + E_{v(i)}$.
\end{prop}

The assumption $Z\leq x(Z)$ seems difficult to verify without knowing
$x(Z)$. In our application of this statement in \cref{def:comp_seq_constr},
however, it will follow from the properties of $x$ listed in the next
proposition.

\begin{proof}[Proof of \cref{prop:x_prop}]
It is enough to prove the following: If $Z \leq x(Z)$ and
$v\in\V\setminus\Nd$ so that $(Z,E_v) > 0$, then $Z+E_v \leq x(Z)$.
Indeed, assuming the contrary, we have $m_v(Z) = m_v(x(Z))$, hence
$(Z,E_v) = (x(Z),E_v) - (x(Z)-Z, E_v) \leq 0$, a contradiction.
\end{proof}

\begin{definition}
The operator $x$ is called the \emph{Laufer operator}. The computation sequence
in \cref{prop:x_cseq} is called the \emph{generalized Laufer sequence}.
\end{definition}

\begin{prop} \label{prop:x_prop}
The operator $x$ satisfies the following properties:
\begin{enumerate}

\item \label{it:x_prop_mon}
If $Z_1, Z_2 \in L$ and $m_n(Z_1) \leq m_n(Z_2)$ for all $n\in\Nd$ then
$x(Z_1)\leq x(Z_2)$.

\item \label{it:x_prop_idemp}
$x(x(Z)) = x(Z)$ for all $Z\in L$.

\item \label{it:x_prop_integ}
Let $Z\in L$ and $Z'\in L_\Q$ and assume that $m_n(Z) = m_n(Z')$ for all
$n\in\Nd$ and $(Z',E_v) = 0$ for all $v\in\V\setminus\Nd$. Then
$x(Z) \geq Z'$, with equality if $Z'\in L$.

\end{enumerate}
\end{prop}

\begin{proof}
For \ref{it:x_prop_mon}, define $Z'\in L$ by $m_n(Z') = m_n(Z_1)$ for
$n\in\Nd$ and $m_v(Z') = m_v(x(Z_2))$ for $v\in\V\setminus\Nd$. Then
$Z'$ satisfies the first two conditions in \cref{prop:x_exist} for
$Z = Z_1$. By definition, we get $x(Z_1) \leq Z' \leq x(Z_2)$.

\ref{it:x_prop_idemp} follows immediately from definition.

For \ref{it:x_prop_integ}, let $\overline G = G \setminus \Nd$.
Assume that $Z_1\in L$ satisfies \ref{it:x_exists_eq} and \ref{it:x_exists_ineq}
of \cref{prop:x_exist}.
Write $Z_1 = Z' + Z'_1$ where
$\supp(Z'_1) \cap \Nd = \emptyset$. Then, we have $(Z'_1,E_v) \leq 0$
for all $v\in\V\setminus\Nd$. Applying \cref{lem:Es_pos} to each connected
component of $\overline G$ we find $Z'_1\geq 0$. If $Z'\in L$, then
$Z'_1\in L$ and by minimality, $Z'_1 = 0$.
\end{proof}

\begin{lemma} \label{lem:x_interpol}
Let $Z\in L$ and take $n\in \Nd$ and $n'\in\Nd^*_n$. Let $u\in\V_n$
be the neighbour of $n$ in the connected component of $G \setminus n$
containing $n'$. If $Z = x(Z)$, then
\begin{equation} \label{eq:x_interpol}
  m_u(Z) = 
    \left\lceil
     \frac{\beta_{n,n'} m_n(Z) + m_{n'}(Z)}{\alpha_{n,n'}}
    \right\rceil,
\end{equation}
where we set $m_{n'}(Z) = 0$ if $n'\in\Nd^*\setminus\Nd$.
\end{lemma}

\begin{proof}
Let $v_1, \ldots, v_s$ be the vertices of the
bamboo \index{bamboo}
between $n$ and $n'$
as in \cref{fig:Bamboo}. We will assume that $s\geq 2$, since, in the cases
$s=0$ or $s=1$, the lemma is a simple consequence of the definition.
Set also $v_0 = n$ and $v_{s+1} = n'$. The condition
$Z = x(Z)$ then implies that the sequence $(m_r)_{r=0}^{s+1}$, given by
$m_r = m_{v_r}(Z)$ is the minimal family satisfying
$m_0 = m_n(Z)$, $m_{s+1} = m_{n'}(Z)$ and
$m_{r-1} - b_{v_r} m_r + m_{r+1} \leq 0$ for $1\leq r \leq s$.
Let $m_0' = m_0$ and
\[
   m'_{s+1}
    = \inf \set{m\in\Z}
	           {m\geq m_{s+1},\,\beta_{n,n'} m_0 + m \equiv 0\,
			                                         (\mod\alpha_{n,n'})}.
\]
Since $\beta_{n,n'} m_0 + m_{s+1} \equiv 0\, (\mod\alpha_{n,n'})$, the
equations
\begin{equation} \label{eq:string}
\begin{gathered}
\begin{array}{lclclcrcl}
           &-& b_1 m'_1 &+& m'_2     &=& -m'_0     &     &            \\
  m'_{r-1} &-& b_1 m'_r &+& m'_{r+1} &=&     0     &\quad&  1 < r < s \\
  m'_{s-1} &-& b_1 m'_s & &          &=& -m'_{s+1} &     &            \\
\end{array}
\end{gathered}
\end{equation}
have integral solutions $m_1',\ldots,m_s'$, see e.g.
\cite{Riemenschneider_DefQ} or \cite{BHPVdV}, III.5.
Furthermore, we have
\[
  m'_1
   = 
     \frac{\beta_{n,n'} m'_0(Z) + m'_{s+1}(Z)}{\alpha_{n,n'}}
   =
    \left\lceil
     \frac{\beta_{n,n'} m_n(Z) + m_{n'}(Z)}{\alpha_{n,n'}}
    \right\rceil.
\]
By \cref{prop:x_prop}\ref{it:x_prop_integ}, $m'_1,\ldots,m'_s$ is the minimal
sequence satisfying $m'_{r-1} - b_r m'_r + m'_{r+1} \leq 0$ for all
$1\leq r \leq s$, and by \cref{prop:x_prop}\ref{it:x_prop_mon} we have
$m_1 \leq m'_1$. Thus, we have proved the $\leq$ part of \cref{eq:x_interpol}.

For the opposite inequality, set $m''_0 = m_0$, $m''_{s+1} = m_{s+1}$
and take $m''_1,\ldots,m''_s$ as the rational solution of
\cref{eq:string}, with $m'_0$ and $m'_{s+1}$ on the right side
replaced with $m''_0$ and $m''_{s+1}$. Then we have
$m_{r-1} -b_rm_r + m_{r+1}$ for $0<r<s+1$, and so
$m_r \geq m''_r$ for all $r$ by \cref{prop:x_prop}\ref{it:x_prop_mon}.
But we also have
\[
  m''_1
   = 
     \frac{\beta_{n,n'} m_0(Z) + m_{s+1}(Z)}{\alpha_{n,n'}},
\]
hence, the $\geq$ part of \cref{eq:x_interpol}.
\end{proof}

\begin{lemma} \label{lem:x_0Z}
Let $G$ be a graph constructed by
Oka's algorithm \index{Oka's algorithm}
as in \cref{ss:Okas_alg}
from the
Newton diagram \index{Newton diagram}
$\Gamma(f)$. We have $x(0) = 0$. Furthermore, 
\begin{enumerate}
\item \label{it:x_0Z_a}
If $G$ is the graph of a
minimal good resolution \index{minimal good resolution}
(i.e. $\Gamma(f)$ is a
minimal diagram) \index{minimal Newton diagram}
then $x(Z_K) = Z_K$.

\item \label{it:x_0Z_N}
If $\Gamma(f)$ is 
convenient, \index{convenient}
then $x(\wt(f)) = \wt(f)$.

\item \label{it:x_0Z_s}
Without the assumption of minimality or convenience,
we have $x(Z_K-E) = Z_K-E + Z_{\mathrm{legs}}$ where
$Z_{\mathrm{legs}}$ is the support of all
legs \index{leg}
in $G$ (see \cref{def:legs}).
\end{enumerate}
\end{lemma}

\begin{proof}
The equality $x(0) = 0$ follows from \cref{prop:x_prop}\ref{it:x_prop_integ}.
Similarly, \ref{it:x_0Z_N} follows from the same lemma, once we show
that if $v\in\V\setminus\Nd$, then $(\wt(f),E_v) = 0$. For such a $v$, there
are $n\in\Nd$ and $n'\in\Nd^*_n$ so that $v$ is on a bamboo connecting
$n$ and $n'$ as in \cref{fig:Bamboo}.
Set $v_0 = n$ and $v_{s+1} = n'$.
We then get $\wt_{v_r}(f) = \ell_{v_r}(p)$ for $0\leq r \leq s$, as well
as $\ell_{v_{s+1}}(p) = 0$ (this follows from convenience).
We therefore get
$(\wt(f), E_{v_r}) =
\ell_{v_{r-1}}(p) - b_{v_r} \ell_{v_r}(p) + \ell_{v_{r+1}}(p) = 0$
for $1\leq r < s$, and
$(\wt(f),E_{v_s}) = 
\ell_{v_{s-1}}(p) - b_{v_s} \ell_{v_s}(p) =
\ell_{v_{s-1}}(p) - b_{v_s} \ell_{v_s}(p) + \ell_{v_{r+1}}(p) = 0$.

Next, we prove \ref{it:x_0Z_s}.
We start with showing $Z_K-E \leq x(Z_K-E)$. By negative
definiteness, there exists a rational solution $(m_v)_{v\in\V\setminus\Nd}$
to the linear equations
$-b_v m_v + \sum_{u\in\V_v\setminus\Nd} m_v
= -\sum_{n\in\Nd\cap\V_v} m_n(Z_K-E)$ for $v\in\V\setminus\Nd$.
Take $Z\in L_\Q$ with
$m_n(Z) = m_n(Z_K-E)$ for $n\in\Nd$ and $m_v(Z) = m_v$ for
$v\in\V\setminus\Nd$ and set $Z_1 = Z_K-E - Z$. Then $Z_1$ is supported
on $\V\setminus\Nd$ and we have $(Z_1,E_v) = (Z_K-E,E_v) \geq 0$ for
$v\in\V\setminus\Nd$. By \cref{lem:Es_pos}, we have $Z_1 \leq 0$,
thus $Z_K-E \leq Z \leq x(Z_K-E)$ by
\cref{prop:x_prop}\ref{it:x_prop_integ}. Now, if $e\in\E$, we have
$(Z_K-E,E_e) = 1$, and so we can start a computation sequence
as in \cref{prop:x_cseq} with $e$. Using the notation $v_1, \ldots, v_s$ as in
\cref{def:legs}, we show that if we already have a
computation sequence \index{computation sequence}
$v_s, v_{s-1}, \ldots, v_{r-1}$ for some $r>1$,
we may take $v_r$ as the next element. But this follows from 
the fact that $(Z_K-E - E_{v_s} - \ldots - E_{v_{r-1}}, E_r) = 1$. 
Thus, we get a computation sequence starting with $Z_K-E$, ending
with $Z_K-E+Z_{\mathrm{legs}}$, at which point we have
$(Z_K-E+Z_{\mathrm{legs}},E_v) \leq 0$ for all $v\in\V\setminus\Nd$.
Indeed, if $v\in\V\setminus\Nd$ is not on a leg, then
$(Z_K-E+Z_{\mathrm{legs}},E_v) = (Z_K-E,E_v) = \delta_v-2 = 0$.
If $v = v_1$ with the notation above, then we get
$(Z_K-E+Z_{\mathrm{legs}},E_v) = (Z_K,E_v) - 1 = -b_v+1\leq 0$
and if $v=v_r$ with $r>1$, we get 
$(Z_K-E+Z_{\mathrm{legs}},E_v) = (Z_K,E_v)= -b_v+2\leq 0$.
This proves $x(Z_K-E) = Z_K-E+Z_{\mathrm{legs}}$.

Finally, we prove \cref{it:x_0Z_a}.
To calculate $x(Z_K)$, we can construct a computation sequence as in
\cref{prop:x_cseq} starting at $Z_K-E+Z_{\mathrm{legs}}+\sum_{n\in\Nd}E_n$,
since $Z_K-E+Z_{\mathrm{legs}} = x(Z_K-E) \leq x(Z_K)$ by
\cref{prop:x_prop}\ref{it:x_prop_mon} and the above
computations, and therefore,
$Z_K-E+Z_{\mathrm{legs}}+\sum_{n\in\Nd}E_n \leq x(Z_K)$.
This sequence is similar to the above.
Take any $n,n'\in\Nd$ with $n'\in\Nd_n$ and a bamboo $v_1, \ldots, v_s$
connecting $n,n'$ as in \cref{fig:Bamboo}. We can then take
$v_1, \ldots, v_s$ as the start of the computation sequence. This is becasue
if
$Z = Z_K-E+Z_{\mathrm{legs}}+\sum_{n\in\Nd}E_n + E_{v_1}+\ldots+E_{v_{r-1}}$,
then $(Z,E_{v_r}) 
  = m_{v_{r-1}}(Z_K) - b_{v_1} (m_{v_r}(Z_K) - 1) + m_{v_{r+1}}(Z_K)-1
  = (Z_K,E_{v_r}) + b_{v_r} - 1 = 1$ by the adjunction equalities.
Now, the concatenation of all the sequences along such bamboos gives a sequence
which ends at $Z_K$. Furthermore, we have
$(Z_K,E_v) \leq 0$ for all $v\in\V\setminus\Nd$ by the minimality assumption,
so this is where the sequence
stops.
\end{proof}

\subsection{Algorithms} \label{ss:alg}

In this subsection we give three different constructions for a computation
sequence, each having some good properties.

\begin{definition} \label{def:rt}
The
\emph{ratio test} \index{ratio test}
is a choice of a node $n \in \Nd$ given a cycle
$Z\in L$.
More precisely, we consider the following three minimising conditions:
\begin{enumerate}[I.]
\item \label{rt:article}
Given $Z\in L$, choose $n$ to minimise the fraction
\[
  \frac{m_n(Z)}{m_n(Z_K-E)}.
\]

\item \label{rt:Newton}
Given $Z\in L$, choose $n$ to minimise the fraction
\[
  \frac{m_n(Z)}{\wt_n(f)}.
\]

\item \label{rt:spec}
Given $Z\in L$, choose $n$ to minimise the fraction
\[
  \frac{m_n(Z)+\wt_n(x_1 x_2 x_3)}{\wt_n(f)}.
\]
\end{enumerate}
If given a choice between more than one nodes minimizing the given fraction,
we choose one maximising the intersection number $(Z,E_n)$.
We also define $Z^{\ref{rt:article}} = Z_K$,
$Z^{\ref{rt:Newton}} = \wt(f)$ and $Z^{\ref{rt:spec}} = x(Z_K-E)$.
Note that since we assume that $(X,0)$ is a hypersurface singularity,
it is Gorenstein. In particular, $Z_K \in L$.
\end{definition}

\begin{definition} \label{def:comp_seq_constr}
Computation sequence $*$ = \ref{rt:article},\ref{rt:Newton},\ref{rt:spec}
is defined recursively as follows. Start by setting $\bar Z_0 = 0$. Given
$\bar Z_i$, if $\bar Z_i = Z^*$, then stop the algorithm. Otherwise,
choose $\bar v(i) \in \Nd$ according to ratio test $*$ and set
$\bar Z_{i+1} = x(\bar Z_i + E_{\bar v(i)})$.
We obtain a computation sequence $(Z_i)$ for $Z^*$
by connecting
$\bar Z_i + E_{\bar v(i)}$ and $\bar Z_{i+1}$ using the
generalized Laufer sequence \index{generalized Laufer sequence}
from \cref{prop:x_cseq}.
This is possible since we have
$\bar Z_i = x(\bar Z_i) \leq x(\bar Z_i+E_{\bar v(i)}) = \bar Z_{i+1}$
by \cref{prop:x_prop}, and so
$\bar Z_i + E_{\bar v(i)} \leq x(\bar Z_i+E_{\bar v(i)})$
(the inequality gives
$m_v(\bar Z_i+E_{\bar v(i)}) \leq m_v(x(\bar Z_i+E_{\bar v(i)}))$
for $v\neq \bar v(i)$, whereas
$m_{\bar v(i)}(\bar Z_i+E_{\bar v(i)}) = m_v(x(\bar Z_i+E_{\bar v(i)}))$
by definition).
Note also that by \cref{lem:x_0Z} and
\cref{conv:alg}, we have $x(Z^*) = Z^*$ in each case.

In case \ref{rt:article}, we will only consider the
finite sequence going from $0$ to $Z_K$.
In case and \ref{rt:spec}, similarly,
we will only consider the
finite sequence going from $0$ to $x(Z_K-E)$.
In the case \ref{rt:Newton}
we continue the sequence to infinity, as in \cref{def:comp_seq},
yielding an infinite sequence $(\bar Z_i)_{i=0}^\infty$.
\end{definition}

\begin{block}
In order to see that we do indeed get a computation sequence for
$Z^*$, it is enough to show that if $\bar Z_i < Z^*$, then
$\bar Z_{i+1} \leq Z^*$. In case \ref{rt:article}, the assumption
$\bar Z_i < Z^*$ gives $(m_n(\bar Z_i)-1) / m_n(Z_K-E) \leq 1$ for
all $n\in\Nd$. If equality holds for all $n$, then $\bar Z_i = Z_K$
and we are at the end of the algorithm. Otherwise, it follows from
the ratio test that $\bar v(i)$ has been chosen so that
$(m_n(\bar Z_i)-1) / m_n(Z_K-E) < 1$, which gives
$(m_n(\bar Z_{i+1})-1) / m_n(Z_K-E) \leq 1$ for all $n$, thus
$\bar Z_i \leq Z_K$ by \cref{prop:x_prop}\ref{it:x_prop_mon}
and \cref{lem:x_0Z}\ref{it:x_0Z_a}. A similar proof holds in
the other cases.
\end{block}

\begin{rem}
Ratio tests \ref{rt:article}, \ref{rt:Newton} were chosen in such a way
that the sets $P_i$ in \cref{def:a_i} are always contained in the
cone generated by $F_n(Z_K-E)$, $F_n$, respectively
(see \cref{def:C_r} and \cref{lem:pts_faces}). Ratio test \cref{rt:spec}
results in a similar, but shifted, statement.
\end{rem}

\begin{conv} \label{conv:alg}
In case \ref{rt:article}, we will assume that the diagram $\Gamma(f)$ is
minimal, \index{minimal Newton diagram}
whereas in cases \ref{rt:Newton} and \ref{rt:spec}, we will
assume that the diagram is
convenient. \index{convenient}
This is motivated by the following
facts.

By \cref{prop:min}, the minimal resolution
graph is obtained by
Oka's algorithm, \index{Oka's algorithm}
assuming that $\Gamma(f)$ is a
minimal diagram. Therefore, although we use our knowledge of the diagram
$\Gamma(f)$ in the proofs of our statements, the statements themselves 
can be made entirely in terms of the link $M$. In particular, the
geometric genus can be computed using only the
link. \index{link}

In cases \ref{rt:Newton} and \ref{rt:spec}, we already assume the knowledge
of $\wt(x_1x_2x_3)$ in order to construct the computation sequence.
Given a diagram $\Gamma(f)$ of an arbitrary function $f \in \O_{\C^3,0}$
with Newton nondegenerate principal part, defining an isolated singularity
with
rational homology sphere \index{homology sphere!rational}
link, let $f' = f + \sum_{c=1}^3 x_c^d$, where
$d\in\N$ is large. Then $f$ and $f'$ define analytically equivalent germs.
Furthermore, let $G'$ be the graph obtained from running
Oka's algorithm on the
diagram $\Gamma(f')$. For any $e\in \E$, set
$\gamma_e = -(Z_K-E+\wt(f),E_e)+1$ if $(Z_K-E+\wt(f),E_e) \neq 0$, but
$\gamma_e = 0$ otherwise.
For a suitable choice for $d$, the graph $G'$ is then obtained from the graph
$G$ by blowing up each end $\gamma_e$ times.
This means that $d$ is chosen so that the combinatorial volume of the
boundary of the resulting convenient diagram is as large as possible.
Therefore, assuming that $\Gamma(f)$ is convenient imposes no restriction
in generality if we already assume the knowledge of $\wt(x_1x_2x_3)$.
\end{conv}

\begin{rem} \label{rem:EO52}
\begin{list}{(\thedummy)} 
{
  \usecounter{dummy}
  \setlength{\leftmargin}{0pt}
  \setlength{\itemsep}{0pt}
  \setlength{\itemindent}{4.5pt}
}
\item
The number $k$ will be fixed throughout as the number of steps in the sequence
$(\bar Z_i)_i^k$. However, it depends on which case we are following. In order
not to complicate the notation, this is not indicated. In case
\ref{rt:article} we have
$k = \sum_{n\in\Nd} m_n(Z_K)$,
in case \ref{rt:spec} we have $k = \sum_{n\in\Nd}m_n(Z_K-E)$ and
in case \ref{rt:Newton}, we have $k = \sum_{n\in\Nd} \wt_n(f)$.

\item \label{it:EO52ii}
Note that $(\bar Z_i)$
\nomenclature{$(\bar Z_i)$}{Subsequence of computation sequence}
forms a
subsequence \index{computation sequence!subsequence}
of $(Z_i)$ as in \cref{rem:subseq}.
That is, by \cref{prop:x_cseq}, we have $(Z_i,E_{v(i)}) > 0$
unless there is an $i'$ so that $Z_i = \bar Z_{i'}$ and
$v(i) = \bar v(i)$.
From the viewpoint of \cref{thm:comp_seq}, the only interesting
part of the
computation sequences \index{computation sequence}
constructed in \cref{def:comp_seq_constr}
is formed by the terms $\bar Z_i$.
\end{list}
\end{rem}

\subsection{Intersection numbers and lattice point count} \label{ss:pt_count}

In this subsection we assume that we have constructed a
sequence \index{computation sequence}
$(\bar Z_i)_{i=0}^k$ as in the previous subsection.
Note that by \cref{rem:EO52}\ref{it:EO52ii}, these are the nontrivial
steps of the computation sequence $(Z_i)$.
The main result is
\cref{thm:pts} which connects numerical data obtained from the sequence
$(\bar Z_i)_{i=0}^k$ with a lattice point count associated with the
Newton diagram. \index{Newton diagram}
The difficult part of proving this result is the technical
\cref{lem:pts_faces} which says that the set $\bar P_i$ consists of the
integral points in a dilated polygon in an affine hyperplane. The
number of these points is then obtained using \cref{thm:point_count}
and this format is compared with the intersection number
$(\bar Z_i,E_{\bar v(i)})$.

We remark that in this section, and in what follows, in
cases \ref{rt:Newton} and \ref{rt:spec}, we assume that the Newton diagram
$\Gamma(f)$ is convenient. In case \ref{rt:article}, however, we
assume that $\Gamma(f)$ is minimal.

\begin{definition} \label{def:a_i}
In cases \ref{rt:article}, \ref{rt:Newton}, \ref{rt:spec}, for any $i$,
define 
\[
  a_i = \max\{ 0, (-Z_i,E_{v(i)}) + 1 \}.
\]
Furthermore, set
$P_i = (\Gamma_+(Z_i) \setminus \Gamma_+(Z_{i+1}))\cap \Z^3$.
Set also $\bar a_i = a_{i'}$ and $\bar P_i = P_{i'}$ 
\nomenclature[Pi]{$\bar P_i$}{Set of integral point at step $i$}
if $\bar Z_i = Z_{i'}$. Thus, in cases
\ref{rt:article}, \ref{rt:spec} we have a sequence $(\bar a)_{i=0}^{k-1}$,
whereas in case \ref{rt:Newton} we consider the infinite sequence
$(\bar a)_{i=0}^\infty$.
\end{definition}

\begin{rem}
Since the sequence $(Z_i)$ is increasing, it follows from definition that
the sequence $(\Gamma_+(Z_i))$ is decreasing.
\end{rem}

\begin{thm} \label{thm:pts}
Assume the notation introduced above and in \cref{ss:cyc_polyt} as well as
the sequence $(\bar Z_i)_{i=0}^k$ defined in \cref{def:comp_seq_constr}.
In case \ref{rt:Newton}, consider as well the continuation of the sequence
as in \cref{def:comp_seq}. Then the following hold:
\begin{enumerate}
\item \label{it:pts_as}
In cases \ref{rt:article}, \ref{rt:spec} we have
$\Z_{\geq 0}^3 \setminus \Gamma_+(Z_K-E) = \amalg_{i=0}^{k-1} \bar P_i$ and
$|\bar P_i| = \bar a_i$ for all $i=0, \ldots, k-1$.

\item \label{it:pts_N}
In case \ref{rt:Newton} we have
$\Z_{\geq 0}^3 = \amalg_{i=0}^\infty \bar P_i$.
In particular, 
$\Z_{\geq 0}^3 \setminus \Gamma_+(\wt(f)) = \amalg_{i=0}^k \bar P_i$.
Furthermore, we have $|\bar P_i| = \bar a_i$ if $i<k$ and
$|\bar P_i| - |\bar P_{i-k}| = \bar a_i$ if $i\geq k$.

\end{enumerate}
\end{thm}

In order to simplify the proof of \cref{thm:pts}, we start with some lemmas.
The proof of the theorem is given in the end of the section.
For the next definition, recall \cref{def:HGam}

\begin{definition} \label{def:C_r}
To each node $n\in\Nd$ in the graph we associate a cone $C_n$
and at each step in the algorithm we record the minimal fraction from
the ratio test.
\begin{itemize}
\item
In case \ref{rt:article} we set 
\[
  C_n = \R_{\geq 0} F_n(Z_K-E), \quad
  \bar r_i = \frac{m_{v(i)}(\bar Z_i)}{m_{v(i)}(Z_K-E)}.
\]
Furthermore, for any $n\in \Nd$, set  $\epsilon_{i,n} = 1$ if
$m_n(\bar Z_i) = \bar r_i m_n(Z_K-E) + 1$, but
$\epsilon_{i,n} = 0$ otherwise.
\item
In case \ref{rt:Newton} we set 
\[
  C_n = \R_{\geq 0} F_n,\quad
  \bar r_i = \frac{m_{v(i)}(\bar Z_i)}{\wt_{v(i)}(f)}.
\]
Furthermore, for any $n\in \Nd$, set $\epsilon_{i,n} = 1$ if
$m_n(\bar Z_i) = \bar r_i \wt_n(f) + 1$, but
$\epsilon_{i,n} = 0$ otherwise.
\item
In case \ref{rt:spec} we set 
\[
  C_n = (\R_{\geq 0} F_n - (1,1,1))\cap\R_{>-1}^3 ,\quad
  \bar r_i = \frac{m_{v(i)}(\bar Z_i) + \wt_{v(i)}(x_1x_2x_3)}{\wt_{v(i)}(f)}.
\]
Furthermore, for any $n\in \Nd$, we set $\epsilon_{i,n} = 1$ if
$m_n(\bar Z_i) + \wt_n(x_1x_2x_3) = \bar r_i \wt_n(f) + 1$,
but $\epsilon_{i,n} = 0$ otherwise.
\nomenclature[Cn]{$C_n$}{Cone associated with $n$}
\nomenclature[rn]{$\bar r_i$}{Ratio at step $i$}
\nomenclature[ev]{$\epsilon_{v,i}$}{Boundary values at step $i$}
\end{itemize}

Fix a step $i$ of the
computation sequence \index{computation sequence}
in cases \ref{rt:article}, 
\ref{rt:Newton}, \ref{rt:spec}. For $n\in\Nd_{\bar v(i)}$,
take $u = u_{\bar v(i),n}\in\V_{\bar v(i)}$ and define
$\epsilon_{i,u} = 1$
if $\epsilon_{i,n} = 1$ and
$\beta_{\bar v(i),n} m_{\bar v(i)}(\bar Z_i) + m_n(\bar Z_i)-1
\equiv 0\, (\mod \alpha_{\bar v(i),n})$, otherwise, set
$\epsilon_{i,u} = 0$.
For $n\in\Nd^*_{\bar v(i)}\setminus\Nd$, we use the following definition.
\begin{itemize}
\item
In case \ref{rt:article}, set $\epsilon_{i,u} = 1$ if $\bar r_i = 1$, but
$\epsilon_{i,u} = 0$ otherwise.
\item
In case \ref{rt:Newton}, set $\epsilon_{i,u} = 0$ for all $i$.
\item
In case \ref{rt:spec}, set $\epsilon_{i,u} = 1$ if
$m_{\bar v(i)}(\bar Z_i) + \wt_{\bar v(i)}(x_1x_2x_3) - 1
\equiv 0\, (\mod \alpha_{n,n'})$, but
$\epsilon_{i,u} = 0$ otherwise.
\end{itemize}

Although in case \ref{rt:spec}, the sets $C_n$ are not technically cones,
we still refer to them as such.
\end{definition}

\begin{rem}
It can happen that for an $n\in\Nd_{\bar v(i)}$ and $u = u_{\bar v(i),n}$
we have $n = u$. In this case, $\alpha_{\bar v(i),n} = 1$, so the condition
$\beta_{\bar v(i),n} m_{\bar v(i)}(\bar Z_i) + m_n(\bar Z_i)-1
\equiv 0\, (\mod \alpha_{\bar v(i),n})$ is vacuous. Therefore,
$\epsilon_{i,u} = \epsilon_{i,n}$ is well defined.
\end{rem}

\begin{lemma} \label{lem:ceil}
For any $i\geq 0$ and $n\in\Nd$ we have
\begin{equation} \label{eq:ceil_n}
  m_n(\bar Z_i) =
  \begin{cases}
    \left\lceil \bar r_i m_n(Z_K-E) + \epsilon_{i,n} \right\rceil
	  & \textrm{in case }\,\ref{rt:article}, \\
    \left\lceil \bar r_i \wt_n(f) + \epsilon_{i,n} \right\rceil
	  & \textrm{in case }\,\ref{rt:Newton}, \\
    \left\lceil \bar r_i \wt_n(f) - \wt_n(x_1x_2x_3) + \epsilon_{i,n}
	  \right\rceil
	  & \textrm{in case }\,\ref{rt:spec}. \\
  \end{cases}
\end{equation}
Similarly, if $\bar v(i) = n$ and $u\in\V_n$, then
\begin{equation} \label{eq:ceil_u}
  m_u(\bar Z_i) =
  \begin{cases}
    \left\lceil \bar r_i m_u(Z_K-E) + \epsilon_{i,u} \right\rceil
	  & \textrm{in case }\,\ref{rt:article}, \\
    \left\lceil \bar r_i \wt_u(f) + \epsilon_{i,u} \right\rceil
	  & \textrm{in case }\,\ref{rt:Newton}, \\
    \left\lceil \bar r_i \wt_u(f) - \wt_u(x_1x_2x_3) + \epsilon_{i,u}
	  \right\rceil
	  & \textrm{in case }\,\ref{rt:spec}. \\
  \end{cases}
\end{equation}
\end{lemma}
\begin{proof}
We prove \cref{eq:ceil_n} in case \ref{rt:article}, the other cases are similar.
For a fixed $i$ and $n\in\Nd$, set
$i' = \max\set{a\in\N}{a\leq i,\,\bar v(a) = n}$.
If $n = \bar v(i)$, then the statement is clear, so we will assume that
$i \neq i'$.
Then $m_n(\bar Z_i) = m_n(\bar Z_{i'})+1$.
The
ratio test \index{ratio test}
guarantees that the sequence $(\bar r_i)$ is increasing.
In particular, we have $\bar r_{i'} \leq \bar r_i$, hence
\[
  \frac{m_n(\bar Z_{i'})}{m_n(Z_K-E)}
  = \bar r_{i'} \leq \bar r_i
\]
and so $m_n(\bar Z_i) - 1 = m_n(\bar Z_{i'}) \leq \bar r_i m_n(Z_K-E)$.
The ratio test furthermore gives $\bar r_i m_n(Z_K-E) \leq m_n(\bar Z_i)$.
Therefore, we have
\[
  \bar r_i m_n(Z_K-E) \leq m_n(\bar Z_i) \leq 
  \bar r_i m_n(Z_K-E)+1.
\]
If we have equality in the second inequality above, then $\epsilon_{i,n} = 1$
and the result holds. Otherwise, we have $\epsilon_{i,n} = 0$ and
$m_n(\bar Z_i) = \lceil \bar r_i m_n(Z_K-E) \rceil$, which also proves the
result.

Next, we prove \cref{eq:ceil_u} in case \ref{rt:article}, the other cases
follow similarly. Assume first that $n = \bar v(i)$ for some $i$, and that
$u = u_{n,n'}$ for some $n'\in\Nd_n$. If $\epsilon_{i,u} = 1$, then
we get, by \cref{lem:x_interpol} and the definition of $\epsilon_{i,u}$
and the above result,
\[
\begin{split}
  m_u(\bar Z_i)
  &=
  \left\lceil
    \frac{\beta_{n,n'} m_n(\bar Z_i) + m_{n'}(\bar Z_i)}{\alpha_{n,n'}}
  \right\rceil \\
  &=
    \frac{\beta_{n,n'} m_n(\bar Z_i) + m_{n'}(\bar Z_i)-1}{\alpha_{n,n'}}
	+ 1 \\
  &=
    \bar r_i
    \frac{\beta_{n,n'} m_n(Z_K-E) + m_{n'}(Z_K-E)}{\alpha_{n,n'}}
	+ 1 \\
  &=
    \bar r_i m_u(Z_K-E) + 1.
\end{split}
\]
The result follows by a similar string of equalities in the case
$\epsilon_{i,n'} = 1 \neq \epsilon_{i,u}$ as in the case
$\epsilon_{i,n'} = 0 = \epsilon_{i,u}$.
If, on the other hand, $n'\in\Nd^*_n\setminus\Nd$, then
\[
\begin{split}
  m_u(\bar Z_i)
  &=
  \left\lceil
    \frac{\beta_{n,n'} m_n(\bar Z_i)}{\alpha_{n,n'}}
  \right\rceil \\
  &=
  \left\lceil
    \frac{\beta_{n,n'} \bar r_i m_n(Z_K-E)}{\alpha_{n,n'}}
  \right\rceil \\
  &=
  \left\lceil
    \bar r_i
    \frac{\alpha_{n,n'} m_u(Z_K-E)+1}{\alpha_{n,n'}}
  \right\rceil \\
  &=
  \begin{cases}
    \lceil \bar r_i m_u(Z_K-E)\rceil     & \bar r_i < 1, \\
    \lceil \bar r_i m_u(Z_K-E)\rceil + 1 & \bar r_i = 1.
  \end{cases}
\end{split}
\]
Here, the first equalities follow as before. The case $\bar r_i = 1$
is clear. The inequality $\bar r_i < 1$
is equivalent to $m_n(\bar Z_i) < m_n(Z_K-E)$. Assuming this, we must prove
\[
  \left\lceil
    \frac{m_n(\bar Z_i)\alpha_{n,n'} m_u(Z_K-E)+m_n(\bar Z_i)}
	     {m_n(Z_K-E) \alpha_{n,n'}}
  \right\rceil
  =
  \left\lceil
    \frac{m_n(\bar Z_i)\alpha_{n,n'} m_u(Z_K-E)}
	     {m_n(Z_K-E) \alpha_{n,n'}}
  \right\rceil.
\]
In order to prove the above equation, we will show that there is
no integer $k\in\Z$ satisfying
\[
  m_n(\bar Z_i)\alpha_{n,n'} m_u(Z_K-E)  + m_n(\bar Z_i)
  \geq
  m_n(Z_K-E) \alpha_{n,n'} k
  >
  m_n(\bar Z_i) \alpha_{n,n'} m_u(Z_K-E).
\]
Using \cref{lem:mod_ZKE}, this is equivalent to
\[
  m_n(\bar Z_i) \beta_{n,n'} m_n(Z_K-E)
  \geq
  m_n(Z_K-E) \alpha_{n,n'} k
  >
  m_n(\bar Z_i) (\beta_{n,n'} m_n(Z_K-E) - 1)
\]
i.e.
\[
  \beta_{n,n'} m_n(Z_K-E)
  \geq
  \alpha_{n,n'} k
  >
  \beta_{n,n'} m_n(Z_K-E) - \bar r_i.
\]
But this is impossible by the assumption $0 \leq \bar r_i < 1$.
\end{proof}

\begin{lemma} \label{lem:null_nopts}
Let $Z\in L$ and assume that $(Z,E_v) > 0$ for some $v\in \V$. Then
$\Fnb_v(Z) \cap \R_{\geq 0}^3 = \emptyset = F_v(Z)$.
If, furthermore, $v\in\Nd$, then 
$\Fnb_v(Z) = \emptyset$.
\end{lemma}
\begin{proof}
Assuming that there is a point $p\in\Fnb_v(Z)\cap\R^3_{\geq 0}$
we arrive at the following
contradiction 
\[
  0 <     -b_v m_v   (Z) + \sum_{u\in\V_v} m_u   (Z) 
    \leq  -b_v \ell_v(p) + \sum_{u\in\V_v} \ell_u(p) 
    =     -\sum_{u\in\V^*_v\setminus\V} \ell_u(p) 
	\leq  0
\]
where the equality is \cref{eq:nbr_sum}. The last inequality follows since
$\ell_v(p) \geq 0$ for all $v\in\V^*$ and $p\in\R_{\geq 0}^3$.
Furthermore, we have $F_v(Z) \subset \Fnb_v(Z)$.
The second statement follows in the same way, since, by construction, we
have $\V_v = \V_v^*$ if $v\in\Nd$.
\end{proof}

\begin{lemma} \label{lem:cone_eqs}
The cones $C_n$, for $n\in\Nd$, are given as follows:
\begin{enumerate}
\item \label{it:cone_eqs_i}
In case \ref{rt:article}
\[
  C_n = \set{p\in\R^3}
            {\fa{ n'\in\Nd_n^*}
             {
                \frac{\ell_{n'}(p)}{m_{n'}(Z_K-E)}  \geq
                \frac{\ell_n   (p)}{m_n   (Z_K-E)}
             }
            }
\]
where we replace $m_{n'}(Z_K-E)$ with $-1$ if $n'\in\Nd^*\setminus\Nd$.
\item \label{it:cone_eqs_ii}
In case \ref{rt:Newton}
\[
  C_n = \set{p\in\R^3}
            {\fa{ n'\in\Nd_n}
             {
                \frac{\ell_{n'}(p)}{\wt_{n'}(f)}   \geq
                \frac{\ell_n   (p)}{\wt_n   (f)}
             },\quad
             \fa{ n'\in\Nd^*_n\setminus\Nd}
             {
                \ell_{n'}(p) \geq 0
             }
            }.
\]
\item \label{it:cone_eqs_iii}
In case \ref{rt:spec}
\[
  C_n = \set{p\in\R^3_{> -1}}
            {
			\begin{array}{l}
			\displaystyle
             \fa{ n'\in\Nd_n}
             {
                \frac{\ell_{n'}(p)+\wt_{n'}(x_1x_2x_3)}{\wt_{n'}(f)}  \geq
                \frac{\ell_n   (p)+\wt_n   (x_1x_2x_3)}{\wt_n   (f)}
             },\\
             \fa{ n'\in\Nd^*_n\setminus\Nd}
             {
                \ell_{n'}(p) > -1
             }
            \end{array}
            }.
\]
\end{enumerate}
\end{lemma}
\begin{proof}
The face $F_n - (1,1,1)$ is given
by the equation $\ell_n = m_n(Z_K-E)$ and the inequalities
$\ell_{n'} \geq m_{n'}(Z_K-E)$ for $n'\in\Nd_n$. \ref{it:cone_eqs_i} therefore
follows, since $C_n$ is the cone over $F_n-(1,1,1)$.

For \ref{it:cone_eqs_ii}, we have, similarly as above, that
$C_n$ is given by inequalities
$\ell_{n'}/\wt_{n'}(f) \geq \ell_n/\wt_n(f)$ for $n'\in\Nd_n$
and $\ell_{n'} \geq 0$ if $n'\in\Nd^*_n$.
If $n'\in\Nd^*_n\setminus\Nd$, then $\ell_{n'}$ is one of the coordinate
functions. Since $F_n\subset\R^3_{\geq 0}$, the above inequalities
are equivalent with
$\ell_{n'}/\wt_{n'}(f) \geq \ell_n/\wt_n(f)$ for $n'\in\Nd_n$
and $\ell_c \geq 0$ for $c=1,2,3$.

\ref{it:cone_eqs_iii} follows in a similar way as \ref{it:cone_eqs_ii}.
\end{proof}

\begin{lemma} \label{lem:nb_ineq}
Let $n\in\Nd$. We have $\Fnb_n(\wt(f)) = F_n$ and
$\Fnb_n(Z_K-E) = F_n - (1,1,1)$. Furthermore,
$\Fnb_n(Z_K-E)$ consists of those points $p \in H^=_n(Z_K-E)$ satisfying
$\ell_{n'}(p) \geq m_{n'}(Z_K-E)$ for $n'\in\Nd_n$ and
$\ell_{n'}(p) \geq -1$ for $n'\in\Nd^*_n\setminus\Nd$.
\end{lemma}
\begin{proof}
Start by observing that for $p \in H^=_n(\wt(f))$ and $n'\in\Nd^*_n$ we have
$\ell_{n'}(p) \geq \wt_{n'}(f)$ if and only if $\ell_u(p) \geq \wt_u(f)$,
where $u = u_{n,n'}$. Indeed, the halfplane defined by either inequality
has boundary the affine hull of the segment $F_n\cap F_{n'}$ and contains
$F_n$. By definition, the face $F_n$ is defined by the equation
$\ell_n(p) = \wt_n(f)$ and inequalities $\ell_{n'}(p) \geq \wt_{n'}(p)$
for $n'\in\Nd^*_n$. The equality $F_n = \Fnb_n(\wt(f))$ follows.
This result, combined with \cref{prop:Z_K}, provides
$\Fnb_n(Z_K-E) = F_n^{\vphantom{\mathrm{nb}}}-(1,1,1)$.

For the last statement, we observe as above that for $p\in H^=_n(Z_K-E)$,
$n'\in\Nd_n$ and $u = u_{n,n'}\in\V_n$,
the inequality $\ell_{n'}(p) \geq m_{n'}(Z_K-E)$ is equivalent
with $\ell_u(p) \geq m_u(Z_K-E)$.
Furthermore if $n'\in\Nd_n^*\setminus\Nd$, and $u = u_{n,n'}$, using
\begin{equation} \label{eq:nb_ineq_pf}
\begin{split}
  \alpha_{n,n'} \ell_u     &= \beta_{n,n'} \ell_n     + \ell_{n'} \\
  \alpha_{n,n'} m_u(Z_K-E) &= \beta_{n,n'} m_n(Z_K-E) - 1
\end{split}
\end{equation}
we find that $\ell_u(p) \geq m_u(Z_K-E)$ if and only if
$\ell_{n'}(p) \geq -1$, since we are assuming that $\ell_n(p) = m_n(Z_K-E)$.
Here, the first equality in \cref{eq:nb_ineq_pf} follows from
\cref{block:abc_comp} and the second one is \cref{lem:mod_ZKE}.
\end{proof}

\begin{definition} \label{def:cone_face}
For any $i$, let
$\Fcn_i = C_{\bar v(i)} \cap H^=_{\bar v(i)}(\bar Z_i)$.
\nomenclature[Fcni]{$\Fcn_i$}{Cone section}
For any $u\in\V_{\bar v(i)}$, let $S_{i,u}$ be the minimal set of
$\ell_u$ on $\Fcn_i$, and set
$\Fcnm_i = \Fcn_i\setminus\cup_{\epsilon_{i,u}=1} S_{u,i}$.
\nomenclature[Fcnmi]{$\Fcnm_i$}{Cone section with boundary conditions}
\end{definition}

\begin{lemma} \label{lem:pts_faces}
In cases \ref{rt:article}, \ref{rt:spec}, for $i=0, \ldots, k-1$ and
in case \ref{rt:Newton}, for $i\geq 0$, we have
\[
  \bar P_i
    = F_{\bar v(i)}(\bar Z_i)     \cap \Z^3
    = \Fnb_{\bar v(i)}(\bar Z_i)  \cap \Z^3
	= \Fcnm_i                     \cap \Z^3.
\]
\end{lemma}
\begin{proof}
We start by proving the inclusions
\begin{equation} \label{eq:lem_incl}
  \bar P_i
    \subset F_{\bar v(i)}(\bar Z_i)     \cap \Z^3
    \subset \Fnb_{\bar v(i)}(\bar Z_i)  \cap \Z^3
	\subset \Fcnm_i                     \cap \Z^3.
\end{equation}
For the first inclusion in \cref{eq:lem_incl}, note that
\[
  \bar P_i = \set{ p\in\Z^3 \cap \Gamma(\bar Z_i) }
                  { m_{\bar v(i)}(\bar Z_i) \leq \ell_{\bar v(i)}(p)
                                       <     m_{\bar v(i)}(\bar Z_i) + 1}.
\]
Since the function $\ell_{\bar v(i)}$ takes integral values on integral
points, we may replace the two
inequalities with $\ell_{\bar v(i)}(p) = m_{\bar v(i)}(Z_i)$, yielding,
in fact, $\bar P_i = F_{\bar v(i)}(\bar Z_i)\cap\Z^3$.

The second inclusion in \cref{eq:lem_incl} follows from definition. 

For the third inclusion, we prove case \ref{rt:article},
cases \ref{rt:Newton} and \ref{rt:spec} follow in a similar way.
Take
$p\in\Fnb_{\bar v(i)}(\bar Z_i) \cap \Z^3$. Clearly, we have
$\ell_{\bar v(i)}(p) = m_{\bar v(i)}(p)$, thus
$p\in H_{\bar v(i)}^=(\bar Z_i)$.
We start with proving $p\in C_{\bar v(i)}$, i.e. that $p$ satisfies
the inequalities in \cref{lem:cone_eqs}\ref{it:cone_eqs_i}.
Take $n\in\Nd_{\bar v(i)}$ and set $u = u_{\bar v(i),n}$. Then
\[
\begin{split}
  \ell_n(p)
    &=    \alpha_{\bar v(i),n} \ell_u(p) -
	      \beta _{\bar v(i),n} \ell_{\bar v(i)}(p) \\
    &\geq \alpha_{\bar v(i),n} m_u(\bar Z_i) -
	      \beta _{\bar v(i),n} m_{\bar v(i)}(\bar Z_i) \\
    &\geq \bar r_i \left(\alpha_{\bar v(i),n} m_u(Z_K-E) -
	      \beta _{\bar v(i),n} m_{\bar v(i)}(Z_K-E)\right) \\
    &=    \bar r_i m_n(Z_K-E).
\end{split}
\]
By \cref{lem:cone_eqs}, this gives $p\in C_{\bar v(i)}$.
If $u = u_{\bar v(i),n}$ and $\epsilon_{i,u} = 1$, then the second inequality
above would be strict by \cref{lem:ceil}. By \cref{lem:cone_eqs}, this
implies $p\notin C_n$. By definition, we get $p\in \Fcnm_i$.

By \cref{lem:disj}, the sets $\Fcnm_{\bar v(i)}\cap\Z^3$ are pairwise disjoint.
Therefore, to prove equality in \cref{eq:lem_incl}, it is now enough to prove 
$\cup_{i=0}^{k-1} \bar P_i \supset \cup_{i=0}^{k-1} \Fcnm_{\bar v(i)}\cap\Z^3$
in cases \ref{rt:article} and \ref{rt:spec}, and
$\cup_{i=0}^\infty \bar P_i \supset \cup_{i=0}^\infty \Fcnm_{\bar v(i)}\cap\Z^3$
in case \ref{rt:Newton}. But this is clear, since, by construction,
we have 
$\cup_{i=0}^{j-1} \bar P_i = \Z_{\geq 0}^3\setminus \Gamma(\bar Z_j)$
for any $j$, hence
$\cup_{i=0}^{k-1} \bar P_i = \Z_{\geq 0}^3\setminus \Gamma(Z_K-E)$
in cases \ref{rt:article} and \ref{rt:spec}, and
$\cup_{i=0}^\infty \bar P_i = \Z_{\geq 0}^3$ in case \ref{rt:Newton}.
\end{proof}

\begin{lemma} \label{lem:last_ones}
In cases \ref{rt:article} and \ref{rt:spec}, if $\bar r_i = 1$
then $(\bar Z_i,E_{\bar v(i)}) > 0$.
Similarly, in cases \ref{rt:article} and \ref{rt:Newton}, if $\bar r_i = 0$,
then $(\bar Z_i,E_{\bar v(i)}) > 0$ unless $i = 0$.
\end{lemma}
\begin{proof}
We start by proving the first statement.

For each $n\in\Nd$, there is a unique $i$ so that $\bar v(i) = n$ and
$\bar r_i = 1$. Since the sequence $\bar r_0,\ldots, \bar r_{k-1}$ is,
by construction, increasing, we see that
$\bar Z_{k-|\Nd|} = x(Z_K-E)$, that the sequence
$\bar v(k-|\Nd|), \bar v(k-|\Nd|+1), \ldots, \bar v(k-1)$ contains
each element in $\Nd$ exactly once and that $\bar r_i < 1$ for $i<k-|\Nd|$.
Recall that by \cref{lem:x_0Z}, we have $x(Z_K-E) = Z_K-E + Z_{\mathrm{legs}}$.

If $u = u_e \in \V_{\bar v(j)}$
for some $k-|\Nd|\leq j \leq k-1$, then we have
$m_u(\bar Z_{j'}) = m_u(Z_K)$ for $k-|\Nd|\leq j' \leq k-1$.
This clearly holds
for $j=k-|\Nd|$ by \cref{lem:x_0Z}, as well as for $j=k-1$. By monotonicity,
\cref{prop:x_prop}\ref{it:x_prop_mon}, the statement holds for all
$k-|\Nd|\leq j' \leq k-1$. By definition of $\epsilon_{i,n}$ we also see
$\epsilon_{j,\bar v(j')} = 1$ if and only if $j'<j$. Thus,
if $k-|\Nd|\leq j,j'\leq k-1$ and
$\bar v(j') \in \Nd_{\bar v(j)}$,
then, by \cref{lem:x_interpol,lem:mod_ZKE}, 
\[
  m_u(\bar Z_j)
  =
  \left\lceil
    \frac{\beta_{n,n'} m_n(Z_K-E) +
	      m_{n'}(Z_K-E)+\epsilon_{j,\bar v(j')}}{\alpha_{n,n'}}
  \right\rceil
  = m_u(Z_K-E) + \epsilon_{j,\bar v(j')}
\]
where $n = \bar v(j)$ and $n' = \bar v(j')$ and
$u = u_{n,n'}$.
Here we use \cref{lem:mod_ZKE}, which implies that
$\beta_{n,n'} m_n(Z_K-E) + m_{n'}(Z_K-E) \equiv 0\, (\mod \alpha_{n,n'})$.
Therefore, if $k-|\Nd| \leq j \leq j'$, we get
\begin{equation} \label{eq:leaves}
\begin{split}
  (\bar Z_j, E_{\bar v(j')})
   &=  (Z_K-E,E_{\bar v(j')})
	    + |\E_{v(j')}| + |\set{v(j'')}{j''< j}\cap\Nd_{v(j')}|\\
   &=  2 - |\set{v(j'')}{j''\geq j}\cap\Nd_{v(j')}|
\end{split}
\end{equation}
because $(Z_K-E,E_n) = 2-\delta_n$ and $\delta_n = |\E_n| + |\Nd_n|$
for all $n\in\Nd$.

For $j=k-|\Nd|,\ldots, k-1$, let
$H_j$ be the graph with vertex set $\bar v(j), \ldots, \bar v(k-1)$
and an edge between $n,n'$ if and only if $n'\in\Nd_n$. We will prove by
induction that the graphs $H_j$ are all trees, i.e. connected, and that
if $j<k-1$, then $v(j)$ is a leaf in $H_j$, that is, it has exactly
one neighbour in $H_j$.

We know already that $H_{k-|\Nd|}$ is a tree. Furthermore, removing a leaf from
a tree yields another tree. Therefore, it is enough to prove that
if, for some $j$, the graph $H_j$ is a tree, then $\bar v(j)$ is a leaf.
The ratio test says that we must indeed choose $\bar v(j)$ from
the graph $H_j$, maximising the intersection number
$(\bar Z_j, E_{\bar v(j)})$.  For simplicity, identify the
graph $H_j$ with its set of vertices.
Then \cref{eq:leaves} says that for any $n\in H_j$, we have
$(\bar Z_j,E_n) = 2-|H_j \cap \Nd_n|$. This number is clearly maximized
when $n$ is a leaf of $H_j$. Furthermore, we have
$(\bar Z_j,E_{\bar v(i)}) = 1$ for $j<k-1$ and
$(\bar Z_{k-1},E_{\bar v(k-1)}) = 2$, proving the first statement of the lemma.

We sketch the proof of the second statement. The sequence
$v(0),\ldots, v(|\Nd|-1)$ contains each element of $\Nd$ exactly once.
Therefore, we have $m_{\bar v(i)}(\bar Z_i) = 0$ for $i<|\Nd|$.
Similarly as in the case above, one shows that if $0<i<|\Nd|$, then
$\bar v(i)$ is chosen in such a way that there is an $i'<i$ so that
$\bar v(i')\in\Nd$, and hence, $m_u(\bar Z_i) > 0$ for
$u = u_{\bar v(i), \bar v(i')}$, which yields
$(\bar Z_i,E_{\bar v(i)}) > 0$.
\end{proof}

\begin{lemma} \label{lem:disj}
In cases \ref{rt:article}, \ref{rt:spec}, let $0\leq i'<i \leq k-1$, in case
\ref{rt:Newton}, let $0\leq i'<i$. Then
$\Fcnm_i \cap \Fcnm_{i'} \cap \Z^3
= \emptyset$.
\end{lemma}
\begin{proof}
We will assuma that we have a point $p \in \Fcnm_i \cap \Fcnm_{i'} \cap \Z^3$,
to arrive at a contradiction.
Set $n = \bar v(i)$ and $n' = \bar v(i')$.

We start with cases \ref{rt:article}, \ref{rt:spec}. In these cases
we will show that $n'\in\Nd_n$ and that $\epsilon_{i,u} = 1$ where
$u = u_{n,n'}$, hence, $p\notin \Fcnm_i$, a contradiction.
Since $p\in C_n\cap C_{n'}$, we have, by \cref{lem:cone_eqs}
\[
    \bar r_i
  = \frac{m    _n   (\bar Z_i   )}{m_n   (Z_K-E)}
  = \frac{\ell _n   (p          )}{m_n   (Z_K-E)}
  = \frac{\ell _{n'}(p          )}{m_{n'}(Z_K-E)}
  = \frac{m    _{n'}(\bar Z_{i'})}{m_{n'}(Z_K-E)}
  = \bar r_{i'}
\]
in case \ref{rt:article}. In case \ref{rt:spec}, we have, similarly,
\[
\begin{split}
     \bar r_i
  &= \frac{   m_n   (\bar Z_i   ) + \wt_n   (x_1x_2x_3)}{\wt_n   (f)} 
   = \frac{\ell_n   (p          ) + \wt_n   (x_1x_2x_3)}{\wt_n   (f)}  \\
  &= \frac{\ell_{n'}(p          ) + \wt_{n'}(x_1x_2x_3)}{\wt_{n'}(f)} 
   = \frac{   m_{n'}(\bar Z_{i'}) + \wt_{n'}(x_1x_2x_3)}{\wt_{n'}(f)} 
   = \bar r_i.
\end{split}
\]
The ratio test guarantees that the sequence $\bar r_0, \bar r_1, \ldots$
is increasing. In particular, there is no $i''$ with $i'<i''<i$ and
$v(i'') = n'$. Therefore, we find
$m_{n'}(\bar Z_i) = m_{n'}(\bar Z_{i'}) + 1$. By definition, we find
$\epsilon_{i,n'} = 1$. 

If $\bar r_i \neq 0$, then
define $\tilde p = \bar r_i^{-1} p$ in case \ref{rt:article} and
$\tilde p = \bar r_i^{-1} (p-(1,1,1)) + (1,1,1)$ in case \ref{rt:spec}.
In each case, we have $\tilde p\in (\Gamma(f) - (1,1,1)) \cap \R_{\geq 0}$,
as well as $\tilde p \in C_n\cap C_{n'}$.
In particular, $\tilde p$ is not in the boundary of the shifted diagram
$\partial\Gamma(f) - (1,1,1)$. Therefore, the intersection
$F_n\cap F_{n'}$ must be one dimensional. Thus, $n'\in\Nd_n$.
Furthermore, we have
$m_n(\bar Z_i) + m_{n'}(\bar Z_i) - 1 = 
m_n(\bar Z_i) + m_{n'}(\bar Z_{i'}) = 
\ell_n(p) + \ell_{n'}(p) \equiv 0\, (\mod\alpha_{n,n'})$, hence
$\epsilon_{i,u} = 1$, where $u = u_{n,n'}$.
But since $\tilde p \in F_n\cap F_{n'} -(1,1,1)$, the point 
$p$ is in the minimal set of $\ell_{n'}$ on $\Fcn_i$, thus
$p\notin \Fcnm_i$.
This concludes the proof in cases \ref{rt:article}, \ref{rt:spec}
if $\bar r_i \neq 0$.

By construction, we cannot have $\bar r_i = 0$ in case \ref{rt:spec},
and as we saw in the proof of \cref{lem:last_ones}, since
$i>0$, the node $n$ has a neighbour $n''$ for which $m_{n''}(\bar Z_i) = 1$,
hence $\epsilon_{u,i} = 1$ for $u = u_{n,n''}$, finishing the proof as above.

Next, we prove the lemma in case \ref{rt:Newton}.
For brevity, we cite some of the methods used above. For instance, we
find $\bar r_i = \bar r_{i'}$ in a similar way.
The case $\bar r_i = 0$ can also be treated in the same way as in
case \ref{rt:article}, so we will assume $\bar r_i > 0$. Set
$\tilde p = \bar r_i^{-1} p$. Then $\tilde p \in \Gamma(f)$.
Unless $\tilde p$ is an integral point, we see in the same way above
that $n'\in\Nd_n$ and that $\epsilon_{i,u} = 1$ for $u = u_{n,n'}$,
finishing the proof. Therefore, assume, that $\tilde p$ is integral.
Then $\tilde p$ lies on one
of the coordinate hyperplanes and it lies on the boundary
$\partial \Gamma(f)$. It follows that we have $n_1, \ldots, n_j\in\Nd$
so that $F_{n_s}$ for $1\leq s \leq j$ are precisely the faces of $\Gamma(f)$
containing $\tilde p$ and that $n_{s'} \in\Nd_{n_s}$ if and only
if $|s-s'| = 1$. There are also have numbers $i_1, \ldots, i_j\in\N$ so that
for each $s$, we have $\bar v(i_s) = n_s$ and
$m_{\bar v(i_s)}(\bar Z_{i_s}) = \ell_{\bar v(i_s)}(p)$.
Let $\sigma$ be a permutation on $1,\ldots,j$ which orders the numbers
$i_1, \ldots, i_j$, that is, $i_{\sigma(1)} < \ldots < i_{\sigma(j)}$.
Just like in the previous case, we see, by \cref{lem:cone_eqs},
that $\bar r_{i_s}$ is constant for $1\leq s\leq j$,
and so for $1\leq s,s' \leq j$ we have
$m_{n_s}(\bar Z_{v(i_{s'})}) = \ell_{n_s}(p) + \epsilon_{i_{s'},n_s}$ and
$\epsilon_{i_{s'},n_s} = 1$ if and only if $s<s'$.
Furthermore, if $u = u_{n_s,n_{s\pm 1}}$ for some $s$,
and $\epsilon_{i_s,n_{s\pm1}} = 1$, then we get $\epsilon_{i_s,u} = 1$
in the same way as before.

By the assumption $p\in C_n\cap C_{n'}$, there are $s,s'$ so that
$n = n_s = \bar v(i_s)$ and $n'=n_{s'} = \bar v(i_s)$.
In particular, $i_s > i_{s'} \geq  i_{\sigma(1)}$. Thus, the lemma is proved,
once we show that for any $s$ with $i_{\sigma(s)} > i_{\sigma(1)}$, we
have, either $i_\sigma(s+1)<i_{\sigma(s)}$ or $i_\sigma(s-1)<i_{\sigma(s)}$,
because, if e.g. $i_\sigma(s+1)<i_{\sigma(s)}$ then $\epsilon_{i,u} = 1$
where $u = u_{n_s, n_{s+1}}$, and so $p\notin \Fcnm_i$.
We will prove this using the following 
satement. If $\sigma(s'')\geq \sigma(s)$, then
\begin{equation} \label{eq:3rd_N}
\begin{gathered}
  (\bar Z_{i_s}, E_{n_{s''}})
  \begin{cases}
    \leq 0 & \mathrm{if}\quad \epsilon_{i_s,n_{s''+1}} = 0\,\mathrm{and}\,
                              \epsilon_{i_s,n_{s''-1}} = 0                \\
	>    0 & \mathrm{if}\quad1<s''<j \,\mathrm{and}\,
	                    \epsilon_{i_s,n_{s''-1}} = 1\,\mathrm{or}\,
	                    \epsilon_{i_s,n_{s''+1}} = 1.
  \end{cases}
\end{gathered}
\end{equation}
Here, we exlude the condition $\epsilon_{i_s,n_{s''+1}} = 0$ if $s''=j$
as well as the condition $\epsilon_{i_s,n_{s''-1}} = 0$ if $s''=1$, since
they have no meaning.

We finish proving the lemma assuming \cref{eq:3rd_N}. Assume that
$1\leq s \leq j$ and $\sigma(s) > \sigma(1)$. The node
$\bar v(i_s)$ is chosen according to the ratio test. If there is
an $s''$ so that $1<s''<j$ and $\epsilon_{i_s,n_{s''}} = 0$, then
this $s''$ can be chosen so that either $\epsilon_{i_s,n_{s''+1}} = 1$ or
$\epsilon_{i_s,n_{s''-1}} = 1$, and therefore
$(\bar Z_{i_s},E_{n_{s''}}) > 0$ by the second part of \cref{eq:3rd_N}.
By the maximality condition in the ratio test, we find
$(\bar Z_{i_s},E_{\bar v(i_s)}) > 0$, and therefore, by the first part
of \cref{eq:3rd_N}, either $\epsilon_{i_s,n_{s''+1}} = 1$ or 
$\epsilon_{i_s,n_{s''-1}} = 1$, that is, either
$i_\sigma(s+1)<i_{\sigma(s)}$ or $i_\sigma(s-1)<i_{\sigma(s)}$.

If, however, there is no such $s''$, then
$\epsilon_{i_s,n_{s''}} = 1$ for $1<s'' <j$, and so $\bar v(i_s) = n_1$
or $\bar v(i_s) = n_j$. In the fist case, we have
$\epsilon_{i_s,n_2} = 1$ and in the second case, we have
$\epsilon_{i_s,n_{j-1}} = 1$, which, in either case, finishes the proof.

We remark that if $k = \sum_{n\in\Nd} \wt_n(f)$, then
$\bar v(i+k) = \bar v(i)$ and $\epsilon_{i,v} = \epsilon_{i+k,v}$ for
any $v$ where $\epsilon_{i,v}$ is defined. It therefore suffices to prove
the above statement for $i<k$, which is equivalent to $\bar r_i < 0$.

We will now prove \cref{eq:3rd_N}. For the first part, take $1\leq s'' \leq j$
with $\sigma(s'') \geq \sigma(s)$, hence $\epsilon_{i_s,n_{s''}} = 0$.
Assume further the given condition, namely that, if $s'' > 1$, then
$\epsilon_{i_s,n_{s''-1}} = 0$ and if $s''<j$, then
$\epsilon_{i_s,n_{s''+1}} = 0$. Since
$m_{n_{s''}}(\bar Z_{i_s}) = \ell_{n_{s''}}(p)$, it is enough,
by \cref{eq:nbr_sum}, to show that $m_u(\bar Z_i) \leq \ell_u(p)$ 
for all $u\in\V^*_{n_{s''}}$.
Note that since $n_{s''} \in \Nd$ we have
$\V^*_{n_{s''}} = \V_{n_{s''}}$, see \cref{rem:Oka}\cref{it:Oka_nodes}.
If $u = u_{n_{s''},n_{s''\pm1}}$, then
\begin{equation} \label{eq:3rd_pm}
\begin{split}
  m_u(\bar Z_{i_s})
   &=
     \left\lceil
      \frac{\beta_{n_{s''},n_{s''\pm1}} m_{n_{s''}}(\bar Z_i)
	                                   +m_{n_{s''\pm1}}(\bar Z_i)}
           {\alpha_{n_{s''},n_{s''\pm1}}}
     \right\rceil\\
   &= \frac{\beta_{n_{s''},n_{s''\pm1}} \ell_{n_{s''}}(p)
	                                   +\ell_{n_{s''\pm1}}(p)}
           {\alpha_{n_{s''},n_{s''\pm1}}}\\
   &= \ell_u(p)
\end{split}
\end{equation}
by \cref{lem:x_interpol}.
If $u\in\V_{n_{s''}}$ is any other neighbour, then there is an $n''\in\Nd^*$
so that $u = u_{n_{s''},n''}$. If $n''\in\Nd^*\setminus\Nd$, then
$n'' = n^*_e$ for some $e\in\E_{n_{s''}}$, and
\begin{equation} \label{eq:3rd_e}
  m_u(\bar Z_i) =
  \left\lceil
    \frac{\beta_e m_{n_{s''}}(\bar Z_{i_s})}{\alpha_e}
  \right\rceil
  =
  \left\lceil
    \frac{\beta_e \ell_{n_{s''}}(p)}{\alpha_e}
  \right\rceil
  \leq
    \frac{\beta_e \ell_{n_{s''}}(p) + \ell_{n''}(p)}{\alpha_e}
  = \ell_u(p).
\end{equation}
If $n'' \in\Nd$ and $n''$ is not one of the nodes $n_1, \ldots, n_j$, then
$p\notin C_{n''}$, in particular, $p\in C_{n_{s''} }\setminus C_{n''}$,
and so by \cref{lem:cone_eqs}
\[
  \frac{\ell_{n''}(p)}{\wt_{n''}(f)} > 
  \frac{\ell_{n_{s''}}(p)}{\wt_{n_{s''}}(f)}
  = \bar r_{i_s}
\]
which, by \cref{lem:ceil}, gives
$\ell_{n''}(p) \geq \bar r_{i_s} \wt_{n''}(f) + \epsilon_{i_s,n''}
= m_{n''}(\bar Z_{i_s})$
because if $\epsilon_{i_s,n''} \neq 0$, then $\bar r_{i_s} \wt_{n''}(f)\in\Z$.
This yields
\begin{equation} \label{eq:3rd_ne}
\begin{split}
  m_u(\bar Z_{i_s})
   &=  
     \left\lceil
         \frac{\beta_{n_{s''},n''} m_{n_{s''}}(\bar Z_i)
	                              +m_{n''}(\bar Z_i)}
              {\alpha_{n_{s''},n''}}
     \right\rceil \\
   &\leq \frac{\beta_{n_{s''},n''} \ell_{n_{s''}}(p)
	                              +\ell_{n''}(p)}
              {\alpha_{n_{s''},n''}}\\
   &=    \ell_u(p).
\end{split}
\end{equation}
This finishes the first part of \cref{eq:3rd_N}.

We prove next the second part of \cref{eq:3rd_N}.
So, assume that $1<s''<j$ and that
$\epsilon_{i_{s\vphantom{'}},n_{s''+1}}
+ \epsilon_{i_{s\vphantom{'}},n_{s''-1}} > 0$.
As in \cref{eq:3rd_pm} we find
$m_u(\bar Z_{i_s}) = \ell_u(p) + \epsilon_{i_{s\vphantom{'}},n_{s''\pm1}}$,
if $u = u_{n_{s''}, n_{s''\pm1}}$. The result therefore follows
from \cref{eq:nbr_sum}, once we prove $m_u(\bar Z_{i_s}) = \ell_{u}(p)$
for $u\in\V_{n_{s''}}\setminus\{n_{s''\pm1}\}$.

We start with the case $u = u_e$ with $e\in\E_{n_{s''}}$. In this case,
we will show that we have, in fact, equality in \cref{eq:3rd_e} (where
$n'' = n_e^*$). This follows once we prove that
$\ell_{n_e^*}(p) < \alpha_e$.
Since the face $F_{n_{s''}}$ has at most four edges, the edge
$F_{n_{s''}} \cap F_{n_e^*}$ is adjacent to at least one of the edges
$F_{n_{s''}} \cap F_{n_{s''\pm1}}$, let us assume that it is adjacent
to $F_{n_{s''}} \cap F_{n_{s''+1}}$, and define $\tilde p_1$ as the point
of intersection of the two edges.
Define also $p_1 = \bar r_i \tilde p$.
Then $\tilde p_1$
is a vertex of the face $F_{n_{s''}}$. From \cref{cor:reg_vx}, we see that
this is in fact a regular vertex, and from \cref{prop:content} we have
$\ell_{n_e^*}(\tilde p - \tilde p_1) = \alpha_{n_{s''},n_e^*}$.
Furthermore, since, in case \ref{rt:Newton}, we assume that the diagram
is convenient, the function $\ell_{n_e^*}$ is one of the coordinates,
and $p_1$ is on the corresponding coordinate hyperplane, thus
$\ell_{n_e^*}(p_1) = 0$. We get
$\ell_{n_e^*}(p) =
\ell_{n_e^*}(p - p_1) = \bar r_i \alpha_e$
and $\bar r_i < 1$ since we are assuming case \ref{rt:Newton}.

For the case when $n''\in\Nd_{n_{s''}}$, equality in \cref{eq:3rd_ne} is proved
in a similar way. This finishes the proof of the second part of
\cref{eq:3rd_N}, and so, the lemma is proved.
\end{proof}

\begin{proof}[Proof of \cref{thm:pts}]
Take any $i$, with $0\leq i \leq k-1$ in cases \ref{rt:article}, \ref{rt:spec},
and $i\geq 0$ in case \ref{rt:Newton}.
Each edge $S$ of the
polygon \index{polygon}
$\Fcn_i$ is the minimal set of some $\ell_u$ with
$u\in\V_{\bar v(i)}$.
In this case, define $\epsilon_S = \epsilon_{i,u}$.
We have
$-b_{\bar v(i)}\ell_{\bar v(i)} + \sum_{u\in\V_{\bar v(i)}} \ell_u \equiv 0$.
Thus, for any $p\in H^=_{\bar v(i)}(Z_i)$, we have
$-b_{\bar v(i)} m_{\bar v(i)}(\bar Z_i)
= \sum_{u\in\V_{\bar v(i)}} \ell_u(p)$. This gives
\begin{equation} \label{eq:int}
  (-\bar Z_i, E_{\bar v(i)}) = 
    - b_{\bar v(i)}  m_{\bar v(i)}(\bar Z_i)
	  - \sum_{u\in\V_{\bar v(i)}} m_u(\bar Z_i) 
    = \sum_{u\in\V_{\bar v(i)}} \ell_u(p) - m_u(\bar Z_i).
\end{equation}
If $u\in\V_{\bar v(i)}$, then there is an $n\in\Nd_{\bar v(i)}^*$ so that
$u = u_{\bar v(i),n}$. Let $S\subset \Fcn_i$ be the minimal set
of $\ell_u$. We then have
\[
  \lceil\ell_u|_S\rceil
  =
  \left\{
  \begin{array}{ll}
    \lceil\bar r_i m_u(Z_K-E) \rceil & \textrm{in case}\,\ref{rt:article}\\
    \lceil\bar r_i \wt_u(f)   \rceil & \textrm{in case}\,\ref{rt:Newton}\\
	\displaystyle
	\left\lceil
	  \frac{\bar r_i m_u(Z_K-E) + \wt_u(x_1x_2x_3)}{\wt_u(f)}
	\right\rceil
                                     & \textrm{in case}\,\ref{rt:spec}\\
  \end{array} 
  \right\}
  =  m_u(\bar Z_i) - \epsilon_{i,u}
\]
by \cref{lem:ceil}.
Furthermore, $\ell_u|_{H^=_{\bar v(i)}(\bar Z_i)}$
is a primitive affine function, whose minimal set on $\bFcn_i$
is $\bar r_i S$.
Using notation from \cref{s:affine}, it follows, that
$\ell_{S} = \ell_u - m_u(\bar Z_i)+\epsilon_S$, and so, by
\cref{eq:int}, we have
$(-\bar Z_i,E_{\bar v(i)}) = c_{\Fcnm_i}$. \index{content}
The theorem therefore follows from \cref{thm:point_count}.
\end{proof}


\newpage
\section{The geometric genus and the spectrum}\label{s:pg}

In this section we assume that $(X,0) \subset (\C^3,0)$ is an isolated
singularity with rational homology sphere link, given by a function
$f\in\O_{\C^3,0}$ with Newton nondegenerate principal part.
Notation from previous sections is retained.

\subsection{A direct identification of $p_g$ and $\Sp_{\leq 0}(f,0)$}
\label{ss:pgspec}

In this subsection we give a simple formula for both the
geometric genus \index{geometric genus}
$p_g$ and part of the
spectrum, \index{spectrum}
$\Sp_{\leq 0}(f,0)$, in terms
of
computation sequences \index{computation sequence}
\ref{rt:article} and \ref{rt:spec}.
\Cref{eq:thm:pg_as} has already been proved in \cite{Nem_Bal} using the
same method.

\begin{thm} \label{thm:pg_as}
Let the
computation sequence \index{computation sequence}
$(\bar Z_i)_{i=0}^k$ be defined as in
\cref{def:comp_seq_constr}, cases \ref{rt:article}, \ref{rt:spec}. Recall
the numbers $\bar r_i \in [0,1]$ from
\cref{def:C_r}. Then, the
geometric genus \index{geometric genus}
of $(X,0)$ is given by the formula
\begin{equation} \label{eq:thm:pg_as}
  p_g = \sum_{i=0}^{k-1} \max \{ 0, (-\bar Z_i, E_{\bar v(i)}) + 1 \}.
\end{equation}
Furthermore, in case \ref{rt:spec} we have \index{spectrum}
\begin{equation} \label{eq:thm:pg_sp}
  \Sp_{\leq 0}(f,0) = \sum_{i=0}^{k-1}
      \max \{ 0, (-\bar Z_i, E_{\bar v(i)}) + 1 \} [\bar r_i] \in \Z[\Q].
\end{equation}
\end{thm}

\begin{lemma}[Ebeling and Gusein-Zade \cite{Ebeling_Gusein-Zade}, proof
of Proposition 1]
\label{lem:Eb_GZ}
Let $g\in \O_{\C^3,0}$ and $n\in \Nd$. Writing $g = \sum_{p\in\N^3} b_p x^p$,
set $g_n = \sum_{\ell_n(p) = \wt_n(g)} b_p x^p$. Then $\wt_n(g) < \div_n(g)$
if and only if $g_n$ is divisible by $f_n$ in the localized ring
$\O_{\C^3,0}[x_1^{-1}, x_2^{-1}, x_3^{-1}]$.
\qed
\end{lemma}

\begin{proof}[Proof of \cref{thm:pg_as}]
We start by proving \cref{eq:thm:pg_as}.
By \cref{prop:pg_hl}, we have $p_g = h_{Z_K}$. Therefore, \cref{eq:thm:pg_as}
follows from \cref{thm:comp_seq}, once we prove
\begin{equation} \label{eq:pf:pg_as}
  \dim_\C \frac{H^0(\X,\O_{\X}(-\bar Z_i))}{H^0(\X,\O_{\X}(-\bar Z_{i+1}))}
    \geq \max \{ 0, (\bar Z_i, E_{\bar v(i)}) + 1 \}
\end{equation}
for all $i = 0, \ldots, k-1$. We start by noticing that for any $p\in \bar P_i$
we have $x^p \in H^0(\X,\O_{\X}(-\bar Z_i))$ (we identify a function on
$(\C^3,0)$ with its restriction to $(X,0)$, as well as its
pullback via $\pi$ to $\X$).  By \cref{thm:pts}\ref{it:pts_as},
the right hand side of \cref{eq:pf:pg_as} is the cardinality of $\bar P_i$, and
so the inequality is proved once we show that the family
$(x^p)_{p\in \bar P_i}$ is linearly independent modulo 
$H^0(\X,\O_{\X}(-\bar Z_{i+1}))$.

Take an arbitrary $\C$-linear combination
$g = \sum_{p\in \bar P_i} b_p x^p$ and assume that
$g \in H^0(\X,\O_{\X}(-\bar Z_{i+1}))$. Since $g = g_{\bar v(i)}$,
\cref{lem:Eb_GZ} says that there is an
$h\in \O_{\C^3,0}[x_1^{-1}, x_2^{-1}, x_3^{-1}]$ so that $g = h f_n$.
In the case when $\bar r_i = 1$ we have
$\bar P_i = F_n(Z_K-E) \cap \Z_{\geq 0}^3 = \emptyset$. Otherwise,
we have $\bar r_i < 1$ and
the support of $g$ is contained in a translate of $r_i F_{\bar v(i)}$.
But the convex hull  of the support of
$h f_n$ must contain a translate of $F_{\bar v(i)}$, unless $h = 0$.
Therefore, we must have $g = 0$, proving the independence of
$(x^p)_{p\in \bar P_i}$.

For \cref{eq:thm:pg_sp}, we note
that if $p\in \bar P_i$, then $p + (1,1,1) \in \R_{\geq 0} F_{\bar v(i)}$
and so $\bar r_i = \ell_f(p)$ (see \cref{def:Newt} for $\ell_f$).
The family of sets $\bar P_i + (1,1,1)$ provides a partition of
the set $\Z_{>0}^3 \setminus \Gamma_+(f)$. Saito's
\cref{prop:Saito_exp} therefore gives
\[
  \Sp_{\leq 0}(f,0)
    = \sum_{i=0}^{k-1} \sum_{p\in \bar P_i} [\ell_f(p)]
    = \sum_{i=0}^{k-1} \max\{0,(-\bar Z_i,E_{\bar v(i)}+1\} [\bar r_i].
\]
\end{proof}

\begin{rem}
It follows from the proof of \cref{thm:pg_as} that the monomials
$x^p$ for $p\in\bar P_i$ form a basis for the vector space
\[
  \frac{H^0(\X,\O_{\X}(-\bar Z_i))}{H^0(\X,\O_{\X}(-\bar Z_{i+1}))}.
\]
\end{rem}

\subsection{The Poincar\'e series of the Newton filtration and the spectrum}
\label{ss:PNfilt}

In this subsection, we give a formula for the
Poincar\'e series \index{Poincar\'e series}
$P^\A_X(t)$
(see \cref{def:Newt})
in terms of
computation sequence \index{computation sequence}
\ref{rt:Newton}. In particular, we
recover $\Sp_{\leq 0}(f,0)$ again.

\begin{thm} \label{thm:Newt}
Let $(\bar Z_i)_{i=0}^\infty$ be the computation sequence defined in
\cref{def:comp_seq_constr}, case \ref{rt:Newton} and define
\begin{equation} \label{eq:thm:Newt}
  P^{\ref{rt:Newton}}_X(t)
  = \sum_{i=0}^\infty \max\{0,(-\bar Z_i,E_{\bar v(i)})+1\} t^{\bar r_i}.
\end{equation}
Then $P^\A_X(t) = P^{\ref{rt:Newton}}_X(t)$. In particular, we have
$P^{\ref{rt:Newton}}_X(t)$ is a rational Puiseux series and
$\Sp(f,0)_{\leq 0} = P^{\ref{rt:Newton},\mathrm{pol}}_X(t^{-1})$.
\end{thm}
\begin{proof}
If $i\geq k$, then we have $\bar r_{i-k} = \bar r_i - 1$.
Therefore, \cref{thm:pts}\ref{it:pts_N} can be rephrased as saying
$\sum_{i=0}^\infty \bar a_i t^{\bar r_i}
= (1-t)\sum_{i=0}^\infty |\bar P_i| t^{\bar r_i}$.
Furthermore, the family $(\bar P_i)$ is a partition of $\Z_{\geq 0}^3$,
and for any $i\geq 0$ and $p\in \bar P_i$, we have $\ell_f(p) = \bar r_i$
which gives $\sum_{i=0}^\infty |\bar P_i|t^{\bar r_i}$ by \cref{lem:PNseries}
and so
\[
   P^{\ref{rt:Newton}}_X(t)
    = (1-t)\sum_{i=0}^\infty |\bar P_i| t^{\bar r_i}
	= (1-t)\sum_{p\in\Z^3_{\geq 0}} t^{\ell_f(p)}
	= P_X^\A(t).
\]
The other statements now follow from \cref{thm:spec_pol_part}.
\end{proof}

\newpage
\section{The Seiberg--Witten invariant} \label{s:SW}

In this section we compare the numerical data obtained in \cref{s:seq}
with coefficients of the
counting function \index{counting function}
$Q_0(t)$ from
\cref{ss:ZQ}. Using \cref{prop:Q_SW}
we recover the
normalized Seiberg--Witten invariant
\index{Seiberg--Witten invariant!normalized}
associated with the
canonical $\spinc$ structure
\index{spinc structure@$\spinc$ structure!canonical}
on the link from computation sequence \ref{rt:article}
from \cref{def:comp_seq_constr}.
The strategy we will follow is similar to that of the
geometric genus. \index{geometric genus}
We do not know whether N\'emethi's main identity $Z_0 = P$ holds
(see \cite{Nem_Poinc}). We will, however, see that computation
sequence \ref{rt:article}
defined in \cref{s:seq} does in fact compute the normalized
Seiberg--Witten invariant $\sw_M(\scan) - (Z_K^2+|\V|)/8$,
using the counting function $Q_0(t)$ in the same way the
geometric genus was obtained using the
Hilbert series \index{Hilbert series}
$H(t)$.

In this section we assume that $M$, a
rational homology sphere, \index{homology sphere!rational}
is the
link \index{link}
of an isolated singularity
$(X,0) \subset (\C^3,0)$ given by a function $f\in\O_{\C^3}$ with
Newton nondegenerate \index{Newton nondegenerate}
principal part. We will assume that $G$ is the
minimal graph \index{minimal graph}
representing the link. Equivalently, it
is the graph obtained by Oka's algorithm in \cref{ss:Okas_alg}, under
the assumption that the diagram $\Gamma(f)$ is
minimal. \index{minimal diagram}
Furthermore, we have the series $Z, Q$ defined in \cref{ss:ZQ}.
We assume that $(\bar Z_i)$ is computation sequence \ref{rt:article} from
\cref{def:comp_seq_constr}.

\begin{thm} \label{thm:qZ_ident}
For $i=0,\ldots,k-1$ we have
\begin{equation} \label{eq:qZ_ident}
  q_{\bar Z_{i+1}} - q_{\bar Z_i}
    = \max\{ 0, (-\bar Z_i, E_{\bar v(i)}) + 1 \}.
\end{equation} \index{counting function}
In particular, we have
\begin{equation} \label{eq:qZ_ident_k}
   \sw_{M}(\scan) - \frac{Z_K^2+|\V|}{8}
     = \sum_{i=0}^{k-1} \max\{ 0, (-\bar Z_i, E_{\bar v(i)}) + 1 \}.
\end{equation} \index{Seiberg--Witten invariant!normalized}
\end{thm}

\begin{cor}[SWIC for Newton nondegenerate hypersurfaces]  \label{cor:SWIC}
The
Seiberg--Witten invariant conjecture
\index{Seiberg--Witten invariant conjecture}
holds for
Newton nondegenerate \index{Newton nondegenerate}
hypersurface \index{hypersurface}
singularities (see \cref{ss:SWIC}).
\qed
\end{cor}

\begin{proof}[Proof of \cref{thm:qZ_ident}]
If the graph $G$ contains a single node, that is, $|\Nd| = 1$, then
\cref{eq:qZ_ident} follows from \cref{lem:1nd}. If $\bar v(i)$ is
a central node (and not the only node),
then \cref{eq:qZ_ident} follows from \cref{lem:central_node}.
If $G$ contains exactly one or two nondegenerate arms and
$\bar v(i)$ is not central, then
\cref{eq:qZ_ident} follows from
\cref{lem:onetwo_arms}.
If $G$ contains three nondegenerate arms and $\bar v(i)$ is not central,
then \cref{eq:qZ_ident} follows from \cref{lem:three_arms}.
By \cref{prop:anatomy}, there are no other cases to consider.

Summing the left hand side of \cref{eq:qZ_ident} gives a telescopic series
yielding $q_{Z_K}-q_0$. We have $q_0 = 0$ because $Z_0(t)$ is supported
on the Lipman cone $\Stop\subset\Z_{\geq 0}\langle\V\rangle$, and
$q_{Z_K} = \sw_{M}(\scan) - (Z_K^2+|\V|)/8$ by \cref{prop:Q_SW}.
\end{proof}

\subsection{Coefficients of the reduced zeta function} \label{ss:red_coeff_z}

In this subsection we will describe a reduction process which will simplify
the proof, as well as computing the coefficients of the
reduced zeta function. \index{zeta function!reduced}
The reduction is a special case of a general reduction theory
established by L\'aszl\'o \cite{Laszlo_th}.

\begin{definition} \label{def:red_groups}
Define $L^\Nd = \Z\gen{E_n}{n\in\Nd}\subset L$ and let $\pi^\Nd:L\to L^\Nd$
be the canonical projection. Set
$V'_Z =  \Z\gen{E^*_v}{v\in\Nd\cup\E}$
\nomenclature[VZ']{$V'_Z$}{$\Z\gen{E^*_v}{v\in\Nd\cup\E}$}
and
$V_Z = V'_Z\cap L$
\nomenclature[VZ]{$V_Z$}{$V'_Z\cap L$}
and $V_Z^\Nd = \pi^\Nd(V_Z)$. For $l\in L$ we also write
$\pi^\Nd(l) = l|_\Nd$.
\nomenclature[lNd]{$l\vert_\Nd$}{Reduction to nodes}
\end{definition}

\begin{lemma} \label{lem:mod}
We have
\begin{equation} \label{eq:lem:mod}
  V_Z^\Nd = \set{l\in L^\Nd}{\fa{ n\in\Nd,\,n'\in\Nd_n}
    {\frac{\beta_{n,n'} m_n(l) + m_{n'}(l)}{\alpha_{n,n'}}\in\Z}}.
\end{equation}
Furthermore, assuming $l'\in V_Z$ with $l'|_\Nd = l$ and $n\in\Nd$,
then, for $n'\in\Nd_n$, we have
\begin{equation} \label{eq:mod_Nd}
  m_u(l') = \frac{\beta_{n,n'} m_n(l) + m_{n'}(l)}{\alpha_{n,n'}},
\end{equation}
where $u=u_{n,n'}$, and for $n'\in\Nd^*_n\setminus\Nd$
\begin{equation} \label{eq:mod_Nds}
  m_u(l') = \frac{\beta_{n,n'} m_n(l) - (l',E_e)}{\alpha_{n,n'}},
\end{equation}
where $e\in\E_n$ so that $n'=n_e^*$ and, again, $u=u_{n,n'}$.
\end{lemma}
\begin{proof}
We start by noting that \cref{eq:mod_Nd,eq:mod_Nds} follow 
from \cref{block:abc_comp}. In fact, this proves the inclusion
$\subset$ in \cref{eq:lem:mod}. By further application of
\cref{block:abc_comp},
given an $l$ in the right hand side of \cref{eq:lem:mod},
$n\in\Nd$ and $n'\in\Nd_n$, we can construct a
sequence $m_{v_1}(l'), \ldots, m_{v_s}(l')$ between
$m_n(l') := m_n(l)$ and $m_{n'}(l') := m_{n'}(l)$, where $v_1, \ldots, v_s$
are as in \cref{fig:Bamboo}. In fact, we have
$m_{v_1}(l') = (\beta_{n,n'}m_n(l) + m_{n'}(l))/\alpha_{n,n'}$ and
the other multiplicities are determined by
$m_{v_{s-1}}(l') - b_{v_s}m_{v_s}(l') + m_{v_{s+1}}(l') = 0$.
For $n'\in\Nd_n^*\setminus\Nd$, we can choose $m_u(l')$ randomly for
$u=u_{n,n'}$ and construct a similar sequence. This yields an element
$l'\in L$ satisfying $(l',E_v) = 0$ for any $v\in\V$ with $\delta_v = 2$,
that is, $l'\in V_Z$, proving the inclusion $\supset$ in \cref{eq:lem:mod}, 
hence equality.
\end{proof}

\begin{definition} \label{def:D_legs}
For any $e\in \E$, set $D_e = \alpha_e E_e^* - E_{n_e}^*$.
\end{definition}

\begin{lemma} \label{lem:end_nodes}
Let $n\in\Nd$ and $e\in\E_n$. Then $D_e$ is an effective integral cycle
which is supported on the leg containing $e$.
In fact, the family $(D_e)_{e\in\E}$ is a $\Z$-basis for
$\ker(V_Z \to V_Z^\Nd)$.
\end{lemma}
\begin{proof}
First, if $v\in\V$ is a vertex outside the leg containing $e$, then
$I_{n,v}^{-1} = \alpha_e I_{e,v}^{-1}$ (recall the notation for the
intersection matrix \index{intersection matrix}
and its inverse, \cref{def:intersection_form}).
This follows from \cite{NeumWa_SpQ}, Theorem 12.2, see also
\cite{EisNeu}, Lemma 20.2.
Thus, $m_v(D_e) = 0$ for $v\in\V$ not on the leg,
i.e. $D_e$ is supported on the leg.
Let $u_e$ be the neighbour of $n$ on this leg. We find
$m_{u_e}(D_e) = (E_n, D_e) = 1$. Furthermore, if the leg consists of
vertices $v_1, \ldots, v_s$ as in \cref{fig:Bamboo}, then the equations
$m_{v_{r-1}}(D_e) - b_{v_r} m_{v_r}(D_e) + m_{v_{r+1}}(D_e) = 0$
recursively show that $m_{v_r} \in \Z$ for all $r$. Thus, we have
$D_e \in L$. Since $(D_e,E_v) \leq 0$ for all $v$ on the leg,
we find, by \cref{lem:Es_pos}, that $m_v(D_e) > 0$ for any such $v$, that is,
$D_e$ is effective and its support is the leg.

For the last statement, set $K = \ker (V_Z \to V_Z^\Nd)$. Note first that by
\ref{lem:mod} we have $\rk K = |\E|$. It is then enough to
find a dual basis, that is, $\lambda_e \in \Hom(K,\Z)$ satisfying
$\lambda_e(D_{e'}) = \delta_{e,e'}$. By what we have just shown,
this is satisfied by $\lambda_e(l) = m_{u_e}(l)$.
\end{proof}

\begin{definition} \label{def:chop_Stop}
Recall the definition of the
Lipman cone \index{Lipman cone}
$\Stop$ in \cref{def:Stop}. 
Set
$\S_Z = \Stop \cap V_Z$
\nomenclature[SZ]{$\S_Z$}{$\Stop \cap V_Z$}
and for $l\in V_Z^\Nd$, define
$\S_Z(l) = \S_Z\cap(\pi^\Nd)^{-1}(l)$.
\nomenclature[SZl]{$\S_Z(l)$}{$\S_Z\cap(\pi^\Nd)^{-1}(l)$}
Define also
$\S_Z^\Nd = \pi^\Nd(\S_Z)$.
\nomenclature[SZNd]{$\S_Z^\Nd$}{$\pi^\Nd(\S_Z)$}
\end{definition}

\begin{lemma} \label{lem:Sl_struct}
Let $l\in V_Z^\Nd$ and choose $l'\in V_Z$ so that $l'|_\Nd = l$.
The element
\begin{equation} \label{eq:def_psi}
  \psi(l) = l' -
            \sum_{e\in\E}
			\left\lfloor
			  \frac{(-l',E_e)}{\alpha_e}
			\right\rfloor
    D_e
\end{equation}
\nomenclature[psl]{$\psi(l)$}{Section to $\pi^\Nd$}
is independent of the choice of $l'$. Furthermore, the set $\S_Z(l)$ consists
of the elements $\psi(l) + \sum_{e\in\E} k_e D_e$
where $k_e \in\N$ satisfy $\sum_{e\in\E_n} k_e \leq (-\psi(l),E_n)$
for all $n\in\Nd$.
\end{lemma}
\begin{proof}
Let $\psi'$ be the element on the right hand side of \cref{eq:def_psi}.
For any $l''\in V_Z$, also satisfying $l''|_\Nd = l$
define $\psi''$ similarly, using $l''$.
By \cref{lem:end_nodes}, there exist $k_e\in\Z$ for $e\in\E$, so that
$l'' = \psi' + \sum_{e\in\E} k_e D_e$.
By definition, we have $0\leq (-\psi',E_e) < \alpha_e$, and so
$k_e = \left\lfloor \frac{(-l'',E_e)}{\alpha_e} \right\rfloor$,
which gives $\psi'' = \psi'$.

For the second statement, we note first that by \cref{lem:end_nodes}, any
element $l'\in V_Z$, restricting to $l$, is of the form
$\psi(l) + \sum_{e\in\E} k_e D_e$ for some $k_e\in\Z$. Also, we have
$l'\in\S_Z(l)$ if and only if $(l',E_v) \leq 0$ for all $v\in\E\cup\Nd$.
For $e\in\E$ we have $-\alpha_e < (\psi(l),E_e) \leq 0$
and $(l',E_e) = (\psi(l),E_e) - k_e\alpha_e$, showing $(l',E_e) \leq 0$
if and only if $k_e\geq 0$. Using \cref{lem:end_nodes} and the results
found in its proof, we find $(l',E_n) = (\psi(l),E_n) + \sum_{e\in\E} k_e$.
Thus, $(l',E_n) \leq 0$ if and only if $\sum_{e\in\E} k_e \leq (-\psi(l),E_n)$.
\end{proof}

\begin{rem}
Let $l\in V_Z^\Nd$.  By \cref{lem:Sl_struct}, we have  $\S_Z(l) \neq \emptyset$ 
if and only if $\psi(l) \in \S_Z$, which is equivalent to
$(\psi(l),E_n) \leq 0$ for all $n\in\Nd$.
\end{rem}

\begin{lemma} \label{lem:psi_inters}
Let $l\in V_Z$ and $e\in\E$. If $u=u_e$, then
\[
  m_u(\psi(l)) =
	\left\lceil \frac{\beta_e m_n(l)}{\alpha_e} \right\rceil.
\]
\end{lemma}
\begin{proof}
This follows from \cref{eq:mod_Nds} and the fact that
$0 \leq (-\psi(l),E_e) < \alpha_e$.
\end{proof}

\begin{lemma} \label{lem:psi_Z_ineq}
Let  $l'\in \S_Z$ and take $Z \in \Stop$ satisfying $Z = x(Z)$ (see 
\cref{ss:Laufer_seq}). Then $l'\geq Z$ if and only if $l'|_\Nd\geq Z|_\Nd$.
\end{lemma}
\begin{proof}
The ``only if'' part of the statement is trivial.
For the ``if'' part, take $l'\in\S_Z(l)$. We find
$l' \geq x(l')$ by the definition of $x$.
The result therefore follows from the monotonicity of $x$,
\cref{prop:x_prop}\ref{it:x_prop_mon}.
\end{proof}

\begin{definition} \label{def:red_zed}
Define the
\emph{reduced zeta function} \index{zeta function!reduced}
$Z^\Nd_0(t)$ \nomenclature[ZNd0]{$Z^\Nd_0(t)$}{Reduced zeta function}
in $|\Nd|$ variables
by setting $t_v = 1$ in $Z_0(t)$ if $v\notin \Nd$. Thus, we have
$Z_0^\Nd(t) = \sum_{l\in L^\Nd} z_l^\Nd t^l$ where
$z_l^\Nd = \sum \set{z_{l'}}{l'\in \S_Z(l)}$.
\nomenclature[zlNd]{$z_l^\Nd$}{Coefficient of reduced zeta function}
This series is supported on $S_Z^\Nd$.
\end{definition}

\begin{block}
Take $l'\in V_Z$ and write $l' = \sum_{v\in\Nd\cup\E} a_v E_v^*$. Using
\cref{eq:zeta} and the linear independence of the family $(E_v^*)_{v\in\V}$,
we see that $z_{l'} = \prod_{v\in\Nd\cup\E} z_{l',v}$,
\nomenclature[zl'v]{$z_{l',v}$}{Factor of coefficient of zeta function}
where we set
\[
  z_{l',v} =
  \begin{cases}
    1                                 & \mathrm{if}\, v\in\E ,\,0\leq a_v,    \\
    \displaystyle
    (-1)^{a_v}\binom{\delta_v-2}{a_v} & \mathrm{if}\,
	                                    v\in\Nd,\,0\leq a_v \leq\delta_v - 2, \\
    0                                 & \mathrm{otherwise}.
  \end{cases}
\]
For any $l\in L^\Nd$, we therefore have, using \cref{lem:Sl_struct},
\begin{equation} \label{eq:factors}
\begin{split}
  z_l^\Nd
   &= \sum_{\substack{(k_e) \in \N^\E\\
	        \fa{n'\in\Nd}{\sum_{e\in\E_{n'}} k_e\leq(-\psi(l),E_{n'})}}}
			    \prod_{n\in\Nd}
				 z_{\psi(l)+\sum_e k_e D_e,n} \\
   &= \prod_{n\in\Nd}
      \sum_{\substack{(k_e) \in \N^{\E_n}\\
	        \sum_{e\in\E_n} k_e\leq(-\psi(l),E_n)}}
			  (-1)^{(-\psi(l),E_n)-\sum_{e\in\E_n} k_e}
			  \binom{\delta_n-2}{(-\psi(l),E_n)-\sum_{e\in\E_n} k_e}.
\end{split}
\end{equation}
Define
$z_{l,n}^\Nd$
\nomenclature[zlnNd]{$z_{l,n}^\Nd$}{Factor of coefficient of
reduced zeta function}
as the $n^{\mathrm{th}}$ factor in the product
on the right hand side
above, so that $z_l^\Nd = \prod_{n\in\Nd} z_{l,n}^\Nd$.
\end{block}

\begin{lemma} \label{lem:factors}
Let $l\in V_Z^\Nd$ and $n\in\Nd$.
\begin{enumerate}
\item \label{it:factors_1}
If $\delta_n - |\E_n| = 1$, then
\[
  z_{l,n}^\Nd =
  \begin{cases}
  1 & \mathrm{if}\;(-\psi(l),E_n) \geq 0, \\
  0 & \mathrm{else}.
  \end{cases}
\]
\item \label{it:factors_2}
If $\delta_n - |\E_n| = 2$, then
\[
  z_{l,n}^\Nd =
  \begin{cases}
  1 & \mathrm{if}\;(-\psi(l),E_n) =    0, \\
  0 & \mathrm{else}.
  \end{cases}
\]
\item \label{it:factors_3}
If $\delta_n - |\E_n| = 3$, then
\[
  z_{l,n}^\Nd =
  \begin{cases}
  1 & \mathrm{if}\;(-\psi(l),E_n) =    0, \\
 -1 & \mathrm{if}\;(-\psi(l),E_n) =    1, \\
  0 & \mathrm{else}.
  \end{cases}
\]

\item \label{it:factors_0}
If $\delta_n - |\E_n| = 0$, then $z_{l,n}^\Nd = \max\{ 0, (-\psi(l),E_n)+1 \}$.
\end{enumerate}
\end{lemma}

\begin{proof}
From \cref{eq:factors}, we find (setting $k = \sum_{e\in\E_n} k_e$)
\[
  z_{l,n}^\Nd  = \sum_{k=0}^{(-\psi(l),E_n)}
			         (-1)^{(-\psi(l),E_n)-k}
                     \binom{|\E_n| + k - 1}{k}
                     \binom{\delta_n-2}{(-\psi(l),E_n)-k}
               = c_{(-\psi(l),E_n)}
\]
where we set $C(t) = \sum_{k=0}^\infty c_k t^k = A(t) \cdot B(t)$, where
\[
\begin{split}
  A(t) &= \sum_{k=0}^\infty \binom{|\E_n| + k - 1}{k} t^k
	    = (1-t)^{-|\E_n|}, \\
  B(t) &= \sum_{k=0}^\infty (-1)^k \binom{\delta_n-2}{k} t^k
        = (1-t)^{\delta_n-2},
\end{split}
\]
hence $C(t) = (1-t)^{\delta_n-2-|\E_n|}$. In each case, this proves the lemma.
\end{proof}

\begin{lemma} \label{lem:q_sum}
We have
\[
  q_{\bar Z_{i+1}} - q_{\bar Z_i}
    = \sum \set{z_l^\Nd}{l\in V_Z^\Nd,\, l\geq \bar Z_i|_\Nd,\,
	                      m_{\bar v(i)}(l) = m_{\bar v(i)}(\bar Z_i) }.
\]
\end{lemma}
\begin{proof}
By definition, $q_{\bar Z_{i+1}}$ is the sum of $z_{l'}$
for $l'\in V_Z$ with $l'\not\geq \bar Z_{i+1}$.
Subtracting $q_{\bar Z_i}$, we cancel
out those summands for which $l'\not\geq \bar Z_i$. Note that these all
appear in the formula for $q_{\bar Z_{i+1}}$ since $\bar Z_{i+1} > \bar Z_i$.
Thus, by \cref{lem:psi_Z_ineq} and the definition of $z_l^\Nd$, we have
$q_{\bar Z_{i+1}} - q_{\bar Z_i} =
\sum \set{z_l^\Nd}
{l\in V_Z^\Nd,\, l\geq \bar Z_i|_\Nd,\, l\not\geq \bar Z_{i+1}|_\Nd}$.
Since $\bar Z_{i+1}|_\Nd = \bar Z_i|_\Nd + E_{\bar v(i)}$, the
condition $l\not\geq \bar Z_{i+1}|_\Nd$ is equivalent to
$m_{\bar v(i)}(l) = m_{\bar v(i)}(\bar Z_i)$, assuming
$l \geq \bar Z_i|_\Nd$.
\end{proof}

\begin{definition} \label{def:Si}
For each step $i$ in the computation sequence, set
\[
  S_i = \set{l\in V_Z^\Nd}{l\geq \bar Z_i|_\Nd,\,
	                       m_{\bar v(i)}(l) = m_{\bar v(i)}(\bar Z_i),\,
						   z_l^\Nd \neq 0}.
\]
\nomenclature[Si]{$S_i$}{Set of $l\in L^\Nd$ contributing to the sum
$q_{Z_{i+1}}-q_{Z_i}$}
\end{definition}

\begin{cor}
For each $i$, \cref{eq:qZ_ident} is equivalent to
\[
  \sum_{l\in S_i} z_l^\Nd = |\bar P_i|.
\]
\qed
\end{cor}

\subsection{The one node case} \label{ss:SW_1nd}

In the case when the diagram $\Gamma(f)$ contains only a single face,
the graph $G$ is starshaped, i.e. contains a single node $n_0$. Our function
$f$ then has the form
$f = f_{n_0} + f^+$, where
$\wt_{n_0}(f^+) > \wt_{n_0}(f)$. The deformation
$f_t(x) = f_{n_0} + tf^+$ has constant topological type, so for 
computations involving the zeta function, or any other topological
invariant, we may assume that $f = f_{n_0}$, i.e. that $f$ is
weighted homogeneous. \index{weighted homogeneous}
In other words, the variety $X = \{ f = 0\}\subset\C^3$ has 
a
good $\C^*$ action. \index{good $\C^*$ action}
The singularities of such varieties have been studied
in \cite{Pinkham,Neu_AbCovQH,Nem_Nico_SWII} (to name a few).

\begin{lemma} \label{lem:1nd}
Assume that $\Nd = \{n_0\}$. For any $i$ we have $\bar v(i) = n_0$ and
there is at most one element $l_i \in S_i$. In that case, we have
\begin{equation} \label{eq:1nd}
  z_{l_i}^\Nd = \max \{0,(-\bar Z_i, E_{\bar v(i)})+1\}.
\end{equation}
In particular, \cref{eq:qZ_ident} holds.
\end{lemma}
\begin{proof}
It is clear that $\bar v(i) = n_0$ for all $i$ and that $m_{n_0}(l_i) = i$
determines a unique element $l_i \in L^\Nd = V_\Z^\Nd \cong \Z$ (for
$L^\Nd = V_Z^\Nd$, see \cref{lem:mod}).
By \cref{lem:factors}\ref{it:factors_0}, we have
$z_{l_i}^\Nd = \max\{0,(-\psi(l_i),E_{\bar v(i)})+1\}$. By
\cref{lem:x_interpol,lem:psi_inters} we have
$m_u(\bar Z_i) = m_u(\psi(l_i))$ for $u\in\V_{\bar v(i)}$ and, furthermore,
$m_{\bar v(i)}(\bar Z_i) = i = m_{\bar v(i)}(\psi(l_i))$.
Therefore, $(-\psi(l_i),E_{\bar v(i)}) = (-\bar Z_i,E_{\bar v(i)})$,
proving \cref{eq:1nd}.
\end{proof}

\subsection{Multiplicities along arms}

In this subsection we use \cref{lem:factors} to determine multiplicities
along arms given ``local data'', i.e. multiplicities on two nodes.
Recall the definition of
arms \index{arm}
in \cref{ss:anatomy}.

\begin{block} \label{block:arm_not}
Assume that the diagram $\Gamma(f)$ has a nondegenerate
arm \index{arm}
consisting of
faces $F_{n_1}, \ldots, F_{n_j}$ so that for $s=2,\ldots, j-1$ we have
$\Nd_{n_s} = \{ n_{s-1}, n_{s+1} \}$ and $\Nd_{n_j} = \{ n_{j-1} \}$.
In this case, we either have $\Nd_{n_1} = \{ n_2 \}$, or there is a
node $n_0 \in \Nd$ so that $\Nd_{n_1} = \{n_0, n_2\}$ or $\{n_0\}$, depending
on whether $j>1$ or $j=1$.
If there is such a node $n_0$, then we set
$\nu = 0$, \nomenclature[n]{$\nu$}{Either $0$ or $1$}
otherwise, set
$\nu = 1$. Note that if $\nu = 1$, then $\Nd = \{n_1,\ldots,n_j\}$.

We fix the following notation as well.
Let
$\alpha_s = \alpha_{n_s, n_{s+1}}$
\nomenclature[as]{$\alpha_s$}{$\alpha_{n_s, n_{s+1}}$}
and
$\beta_s = \beta_{n_s, n_{s+1}}$,
\nomenclature[bs]{$\beta_s$}{$\beta_{n_s, n_{s+1}}$}
for $\nu \leq s < j$. Also,
let
$\overline\beta_s = \beta_{n_{s+1}, n_s}$,
\nomenclature[bs]{$\overline\beta_s$}{$\beta_s^{-1}\,(\mod \alpha_s)$}
so that
$\beta_s \overline\beta_s \equiv 1\, (\mod \alpha_s)$.
This way, the two equations 
$\beta_s m_s + m_{s+1} \equiv 0$ and
$m_s + \overline\beta_s m_{s+1} \equiv 0\, (\mod \alpha_s)$ are equivalent.

We always assume that $\nu < j$. If $\nu = j$, then we necessarily have
$\nu = j = 1$ and $\Nd = \{ n_1 \}$. This case is covered in
\cref{ss:SW_1nd}.
Note that \cref{lem:hand} does not make sense unless we make this assumption.
\end{block}

\begin{lemma} \label{lem:m_exists}
Assume the notation given in \cref{block:arm_not}.
Let $\nu\leq s < j$ and assume that we have numbers $m_s, m_{s+1}$ satisfying 
$\beta_s m_s + m_{s+1} \equiv 0 \, (\mod \alpha_s)$
Then there exist unique numbers $m_\nu, \ldots, m_j$ (with $m_s$ and
$m_{s+1}$ unchanged), so that for any $r$ we have
\begin{equation} \label{eq:m_exists_mod}
  \beta_r m_r + m_{r+1} \equiv 0 \, (\mod \alpha_r)
\end{equation}
and
\begin{equation} \label{eq:m_exists}
  \frac{m_{r-1} + \bar \beta_{r-1} m_r}{\alpha_{r-1}} +
    E_{n_r}^2 m_r + 
    \frac{\beta_r m_r + m_{r+1}}{\alpha_{r+1}} +
    \sum_{e\in\E_{n_r}}
    \left\lceil\frac{\beta_e m_r}{\alpha_e}\right\rceil = 0.
\end{equation}
\end{lemma}

\begin{rem} \label{rem:m_exists}
Assume that $\nu<r<j$.
If $l\in L^\Nd$ and $m_{n_s}(l) = m_s$ for $s=r-1, r, r+1$, then
\ref{eq:m_exists} is equivalent to $z_{l,n_r}^\Nd \neq 0$, which again
is equivalent to $z_{l,n_r}^\Nd = 1$.
This follow from \cref{lem:factors}\ref{it:factors_2}, and the
fact that the left hand side of \cref{eq:m_exists} equals
$(\psi(l),E_{n_r})$ by \cref{lem:mod,lem:psi_inters}.
\end{rem}

\begin{proof}[Proof of \cref{lem:m_exists}]
Assume that $\nu\leq r<j$ and that we have integers $m_r$ and $m_{r+1}$
satisfying $\beta_r m_r + m_{r+1} \equiv 0 \, (\mod \alpha_r)$. Then
\cref{eq:m_exists} defines an integer $m_{n_{r-1}}$ which satisfies
\cref{eq:m_exists}. It is clear from this definition that
$m_{r-1} + \overline\beta_{r-1} m_r \equiv 0\, (\mod\alpha_{r-1})$, or
equivalently, $\beta_{r-1}m_{r-1} + m_r \equiv 0\, (\mod\alpha_{r-1})$.
This way, we obtain $m_\nu, \ldots, m_{s-1}$ recursively.
A similar process produces the numbers $m_{s+2}, \ldots m_j$.
\end{proof}

\begin{definition} \label{def:arm_seq}
We will refer to a sequence of numbers $m_\nu,\ldots, m_j \in \Z$ satisfying
\cref{eq:m_exists_mod,eq:m_exists} as an
\emph{arm sequence}.  \index{arm sequence}
When there are more than one arms in the diagram, it will be clear from
context which arm is being referred to.
\end{definition}

\begin{rem}
We have $\delta_{n_r} - |\E_{n_r}| = |\Nd_{n_r}| = 2$ for $\nu<r<j$ (for
$1<r<j$ if there is no $n_0$). Therefore, it follows form \cref{lem:factors}
and \cref{lem:psi_inters} that if $l\in L^\Nd$ and $z_l^\Nd \neq 0$
then the sequence given by $m_r = m_{n_r}(l)$ must be an
arm sequence. \index{arm sequence}
\end{rem}

\begin{lemma} \label{lem:armseq_points}
Let $m_\nu,\ldots, m_j$ be an arm sequence. There exist unique points
$p_{\nu+1}, \ldots, p_{j-1} \in \Z^3$
so that for each $\nu<s<j$
and $r=s-1, s, s+1$ we have $\ell_{n_r}(p_s) = m_r$.
\end{lemma}

\begin{proof}
Let $s$ be given, $\nu<s<j$. For simplicity, set $\ell_r = \ell_{n_r}$ for
all $r$. We note first that the functionals $\ell_{s-1}, \ell_s, \ell_{s+1}$
are linearly independent. This follows from the fact that the
functions $\ell_r - \wt_{n_r}(f)$ for $r = s-1, s+1$ restricted to 
the hyperplane $\ell_s = \wt_{n_s}(f)$ support adjacent edges of the polygon
$F_{n_s}$. Thus, $\ell_{s-1}, \ell_{s+1}$ induce an isomorphism
$H^=(\wt(f))\to\R^2$, so the three functions form a dual basis
of $\R^3$. The existence of $p_s\in\R^3$ follows, but we must show that
$p_s$ has integral coordinates.

Define $u_+, u_-,u_0\in\V_{n_s}$ by $u_\pm = u_{n_s,n_{s\pm1}}$, and let
$u_0$ be some other neighbour.  Since the
functional $\ell_{n_s}$ is primitive, the hyperplane $H = H^=_{n_s}(m_s)$
contains a two dimensional affine lattice $H\cap\Z^3$. The restrictions
$\ell_{u_\pm}|_H, \ell_{u_0}|_H$ are all primitive, and by \cref{cor:reg_vx}
the functions $\ell_{u_+}|_H, \ell_{u_0}|_H$ give affine coordinates over $\Z$
of this lattice. It is therefore enough to show that these functionals
take integral values on $p_s$. First, we find
\[
  \ell_{u_+}(p_s) =
    \left(
    \frac{\beta_s \ell_{n_s} + \ell_{n_{s+1}}}{\alpha_s}
    \right)(p_s)
	= 
    \frac{\beta_s m_s + m_{s+1}}{\alpha_s} \in \Z
\]
by \cref{eq:m_exists_mod}, and a similar formula for $\ell_{u_-}(p_s)$.
Subtracting \cref{eq:nbr_sum} from
\cref{eq:m_exists}, evaluated at $p_s$, and dividing by $|\E_n|$ one finds
\begin{equation} \label{eq:leg_ceil}
  \ell_{u_0}(p_s) =
  \left\lceil \frac{\beta_e m_s}{\alpha_e} \right\rceil \in \Z,
\end{equation}
where $e\in\E_{n_s}$.
Note that here we use the fact that $\ell_{u_0}$ does not depend on the choice
of $u_0\in\V_{n_s}$, as long as $u_0\neq u_\pm$, and similarly,
$\alpha_e, \beta_e$
do not depend on $e\in\E_n$. This is because $F_{n_s}$ is a triangle, and
all legs of $n_s$ are associated with one edge of this triangle, see also
\cref{prop:arm_hands}.
\end{proof}

\begin{definition} \label{def:assoc}
Let $m_\nu,\ldots, m_j$ be an arm sequence. We call the points
$p_{\nu+1}, \ldots, p_{j-1}$ the
\emph{associated vertices}. \index{associated vertices}
\nomenclature[ps]{$p_s$}{Associated vertices}
The
\emph{associated lines} \index{associated lines}
are defined as
$L_s = \set{p\in\R^3}{\ell_{n_s}(p) = m_s,\,\ell_{n_{s+1}}(p) = m_{s+1}}$
\nomenclature[Ls]{$L_s$}{Associated lines}
for $\nu\leq s < j$.
Thus, we have $p_s, p_{s+1} \in L_s$, whenever these are defined.
\end{definition}

\begin{lemma} \label{lem:hand}
Let $m_\nu, \ldots, m_j$ be an
arm sequence, \index{arm sequence}
and assume that the
arm goes in the direction of the $x_3$-axis.
Assume furthermore that $l\in V_Z^\Nd$ with
$m_{n_s}(l) = m_s$ for all $s$.
The following are equivalent.
\begin{enumerate}

\item \label{lem:hand_exists}
There is an $\nu\leq s<j$ so that $L_s$
contains an integral point with nonnegative $x_1$ and $x_2$ coordinates.

\item \label{lem:hand_forall}
The line $L_s$
contains an integral point with nonnegative $x_1$ and $x_2$ coordinates
for all $0\leq s<j$.

\item \label{lem:hand_ineq}
We have
\begin{equation} \label{eq:hand_ineq}
  \frac{m_{j-1} + \bar\beta_{j-1} m_{j-1}}{\alpha_{j-1}}
    - b_{n_j} m_j + 
    \sum_{e\in\E_{n_j}}
    \left\lceil\frac{\beta_e m_j}{\alpha_e}\right\rceil
  \leq 0.
\end{equation}

\item \label{lem:hand_nzero}
We have $z_{l,n_j}^\Nd \neq 0$.

\item \label{lem:hand_one}
We have $z_{l,n_j}^\Nd = 1$.

\end{enumerate}
\end{lemma}

\begin{proof}
Using \cref{lem:psi_inters,lem:factors}, we see that \ref{lem:hand_ineq},
\ref{lem:hand_nzero} and \ref{lem:hand_one} are equivalent.
For brevity, let us say (in this proof) that $p\in\R^3$ is \emph{good} if
it has integral coordinates, with the $x_1$ and $x_2$ coordinates nonnegative.
Let $p_{\nu+1}, \ldots, p_{j-1}$ be the points associated with the arm
sequence. We start by proving the following

\emph{Claim.} Assume that $L_s$ contains a good point for some
$\nu\leq s < j$. Then $p_s$ is a good point if $s>\nu$, and
$p_{s+1}$ is good if $s+1<j$.

We prove the claim for $p_s$, the proof for $p_{s+1}$ is the same.
By \cref{prop:arm_hands}, $F_{n_s}$ is a triangle with exactly one edge
on the boundary $\partial \Gamma(f)$, and we can assume that this
edge lies on the $x_2x_3$ plane.
For $k=1,2,3$, let $\ell_k$ be the standard coordinate functions in
$\R^3$, that is, $\ell_1(p) = \langle p,(1,0,0)\rangle$, etc.
Using notation as in \cref{eq:leg_ceil}, we find
\[
  \left(\frac{\beta_e \ell_{n_s} + \ell_1}{\alpha_e}\right)(p_s)
    = \ell_{u_0}(p_s)
	= \left\lceil \frac{\beta_e m_j}{\alpha_e} \right\rceil,
\]
where $e\in\E_n$, which shows that $0\leq \ell_1(p_s) < \alpha_e$. From
\cref{prop:content}
we see that the restricted function $\ell_1|_{L_s}$ has content
$\alpha_e$. This shows that $\ell_1|_{L_s\cap\Z^3}$ takes its minimal
nonnegative value at $p_s$.  Since $L_s$ is parallel to
the edge $F_{n_s} \cap F_{n_{s+1}}$, we find that $\ell_1$ and $\ell_2$
define opposite orientations on $L_s$. Thus, if $\ell_2(p_s) < 0$, 
we have $\ell_2(p) < 0$ for all integral points $p\in L_s$ for which
$\ell_1(p) \geq 0$. By the assumption, this is not the case, so
$\ell_2(p_s) \geq 0$, proving the claim.

The implication \ref{lem:hand_exists}$\Rightarrow$\ref{lem:hand_forall}
now follows from repeated usage of the claim. Namely, if
\ref{lem:hand_exists} holds for some $s<j-1$, then $p_{s+1} \in L_s$ is
good. But this means that $p_{s+1}\in L_{s+1}$ is good, and we can apply the
claim to $L_{s+1}$. This proves that $L_r$ contains good points for any
$r\geq s$. A similar induction proves the same statement for $r<s$.
Thus, \ref{lem:hand_forall} holds.

Next, we prove \ref{lem:hand_forall}$\Rightarrow$\ref{lem:hand_ineq}.
Let $p\in L_{j-1}$ be good.
Subtracting \cref{eq:nbr_sum}, with $n=n_j$ and evaluated at $p$,
from \cref{eq:hand_ineq}, we see that it is enough to prove
$\lceil \beta_e m_j / \alpha_e \rceil \leq \ell_u(p)$
for any $e \in \E_{n_j}$, with $u=u_e$.
Let $n' = n_e^*$. By \cref{prop:arm_hands}, 
$\ell_{n'}$ is a nonnegative linear combination of $\ell_1$ and $\ell_2$.
Therefore, we have $\ell_{n'}(p) \geq 0$. The formula
$\alpha_e \ell_u = \beta_e\ell_n + \ell_{n'}$ therefore gives
$\ell_u(p) \geq \beta_e m_j / \alpha_e$, hence 
$\ell_u(p) \geq \lceil \beta_e m_j / \alpha_e\rceil$, since
$\ell_u(p) \in \Z$.

Finally, we prove \ref{lem:hand_ineq}$\Rightarrow$\ref{lem:hand_exists}.
Assuming \ref{lem:hand_ineq}, we will prove \ref{lem:hand_exists} with
$s = j-1$.
Let $u_- = u_{n_j,n_{j-1}}$ and let $\tilde\ell_1,\tilde\ell_2$
and $\E_{n_j}^1,\E_{n_j}^2$ be as in \cref{prop:arm_hands}.
Take $e_1\in\E_{n_j}^1$. Then $\tilde\ell_1 = \ell_1$ restricted
to the line $L_s$ has content $\alpha_{e_1}$ by \cref{prop:content}.
Thus, there is a unique
integral point $p\in L_s$ so that $0\leq \ell_1(p) < \alpha_{e_1}$
and it suffices to show $\ell_2(p) \geq 0$.
Subtract \cref{eq:nbr_sum} for $n = n_j$, evaluated at $p$
from \cref{eq:hand_ineq} to find
\begin{equation} \label{eq:lem_hand_ceil}
  \sum_{e\in\E_{n_j}}
    \left\lceil \frac{\beta_e m_j}{\alpha_e} \right\rceil
	-
	\ell_{u_e}(p)
  \leq
    0.
\end{equation}
We have
\[
  \ell_{u_{e_1}}(p) = \frac{\beta_e m_j + \ell_1(p)}{\alpha_e}
  = \left\lceil \frac{\beta_e m_j}{\alpha_{e_1}} \right\rceil
\]
by the definition of $p$, and the fact that $\ell_{u_{e_1}}(p) \in \Z$.
Therefore, the summands in \cref{eq:lem_hand_ceil} corresponding to
$e\in\E_{n_j}^1$ vanish, and we are left with summands corresponding
to $e\in \E_{n_j}^2$, yielding
\[
  \frac{\beta_e m_j + \tilde\ell_2(p)}{\alpha_e}
    =    \ell_u(p)
    \geq \left\lceil \frac{\beta_e m_j}{\alpha_e} \right\rceil
\]
for $e\in\E_{n_j}^2$, hence $\tilde\ell_2(p) \geq 0$.
If $\tilde\ell_2 = \ell_2$, then we are done.
Otherwise, we have $\tilde\ell_2 = \alpha_{e_1} \ell_2 + \ell_1$
so we find $\ell_2(p) \geq 0$, since $\ell_1(p) < \alpha_{e_1}$.
\end{proof}

\begin{lemma} \label{lem:arm_mono}
Let $m_\nu, \ldots, m_j$ be a nonzero arm sequence
and assume that $\nu<r<j$. Assume furthermore that the equivalent properties
in \ref{lem:hand} hold. Then, for $\nu<s<j$ we have
\begin{equation} \label{eq:arm_mono+}
  \frac{m_{s-1}}{m_{s-1}(Z_K-E)} \leq \frac{m_s}{m_s(Z_K-E)}
  \Rightarrow
  \frac{m_s}{m_s(Z_K-E)} < \frac{m_{s+1}}{m_{s+1}(Z_K-E)}
\end{equation}
and
\begin{equation} \label{eq:arm_mono-}
  \frac{m_{s+1}}{m_{s+1}(Z_K-E)} \leq \frac{m_s}{m_s(Z_K-E)}
  \Rightarrow
  \frac{m_s}{m_s(Z_K-E)} < \frac{m_{s-1}}{m_{s-1}(Z_K-E)}.
\end{equation}
\end{lemma}

\begin{proof}
We will prove \cref{eq:arm_mono+}, \cref{eq:arm_mono-} follows similarly.
Assume that the arm goes in the direction of the $x_3$ coordinate
and let $p_{\nu+1}, \ldots, p_{j-1}$ be the associated vertices.
Let $B_r = \cup_{t=r}^j C_{n_r}(Z_K-E)$.
The functional
\[
  \ell_r = 
  \frac{\ell_{n_{r-1}}}{m_{r-1}(Z_K-E)} 
  - \frac{\ell_{n_r}}{m_r(Z_K-E)}
\]
seperates the diagram $\Gamma(Z_K-E)$ into two parts, namely
$F_{n_r}(Z_K-E),\ldots, F_{n_j}(Z_K-E)$, where it is nonnegative, and the other
faces where it is nonpositive. Therefore,  $p_s \in B_r$ if
and only if $\ell_r(p_s) \geq 0$, with equality if and ony if
$p_s \in \partial B_r$. Thus, the left hand side of \cref{eq:arm_mono+}
gives $p_s \notin B_s^\circ$, which gives $p_s \notin B_{s+1}$, which, 
again, translates to the right hand side of \cref{eq:arm_mono+}.
\end{proof}

\subsection{Multiplicities around $v(i)$}

In this subsection we assume a fixed step $i$ of the
computation sequence \index{computation sequence}
from \cref{def:comp_seq_constr}. We also assume $|\Nd| > 1$.

\begin{lemma} \label{lem:ell_min}
Let $u\in\V_{\bar v(i)}$ and assume $\bar P_i \neq \emptyset$.
Then
\[
  m_u(\bar Z_i)= \min\set{\ell_u(p)}{p\in \bar P_i}.
\]
\end{lemma}
\begin{proof}
By \cref{lem:pts_faces} we have
\[
  \bar P_i = \set{p\in H^=_{\bar v(i)}(\bar Z_i) \cap \Z^3}
                 {\fa{ u\in\V_{\bar v(i)}}{ \ell_u(p) \geq m_u(\bar Z_i)}}.
\]
It is therefore enough to show that for any $u\in\V_{\bar v(i)}$, there is a
$p\in \bar P_i$ so that $\ell_u(p) = m_u(\bar Z_i)$.
By \cref{cor:reg_vx}, there is a $u'\in\V_{\bar v(i)}$
so that $\ell_u, \ell_{u'}$ form an affine basis when restricted to
$H^=_{\bar v(i)}(\bar Z_i)$. Therefore, there is a
$p \in H^=_{\bar v(i)}(\bar Z_i)$
so that $\ell_{u}(p) = m_u(\bar Z_i)$ and 
$\ell_{u'}(p) = m_{u'}(\bar Z_i)$. If
$F_{v(i)}$ is a triangle, then
there is a $u''\in\V_{\bar v(i)}$ so that $u,u',u''$ represent all bamboos
and leg groups of $\bar v(i)$. Furthermore, we must have
$\ell_{u''}(p) \geq m_{u''}(\bar Z_i)$, since otherwise, by the above
description, we would have $\bar P_i = \emptyset$, a contradiction.

Assume now that $F_{\bar v(i)}$ is a trapezoid. If $u$ lies on a bamboo not
corresponding to the top edge of $F_{\bar v(i)}(f)$ (see \cref{def:top_edge}),
then we may choose $u'$ with the same property. Now define $p$ in the 
same way as above (note that all vertices of a trapezoid are regular).
It is then easy to see
(from e.g. \cref{prop:poly_class}) that for any $u''\in\V_{\bar v(i)}$ with
$\ell_{u''} \neq \ell_u, \ell_{u'}$, the function $\ell_{u''}$ restricted
to the cone
\[
  \set{p'\in H^=_{\bar v(i)}(\bar Z_i)}
                 {\ell_ u  (p') \geq m_ u  (\bar Z_i),\,
				  \ell_{u'}(p') \geq m_{u'}(\bar Z_i)   }
\]
takes
its maximal value at the vertex $p$. From the assumption
$\bar P_i \neq \emptyset$,
we now find $\ell_{u''}(p) \geq m_{u''}(\bar Z_i)$ for all
$u''\in\V_{\bar v(i)}$ and therefore $p\in \bar P_i$.

The last case we must consider is when $F_{\bar v(i)}$ is a trapezoid and
$u$ lies on a bamboo corresponding to a top face. We have
$\bar P_i \subset \bar r_i F_{\bar v(i)}(Z_K-E)$. Since the length of the
interval $\ell_u(F_{\bar v(i)}(Z_K-E))$ is one, we find that if $\ell_u$ takes
an integral value on $\bar r_i F_{\bar v(i)}(Z_K-E)$, then it must be
$\lceil \bar r_i m_u(Z_K-E) \rceil$. In other words, if $p\in P_i$, then
$\ell_u(p) = \lceil \bar r_i m_u(Z_K-E) \rceil$
(in the case $\bar r_i = 1$, this gives
$\ell_u(p) = m_u(Z_K-E)$ or
$\ell_u(p) = m_u(Z_K-E)+1$, but in the latter case, the point $p$ has a
negative coordinate).
This finishes the proof of the lemma.
\end{proof}

\begin{cor} \label{cor:trap_top}
Assume that $F_{n_0}$ is a
trapezoid, \index{trapezoid}
and that $n_1\in\Nd$ so that
$F_{n_0}\cap F_{n_1}$ is the
top edge \index{top edge}
of the trapezoid and that
$\bar v(i) = n_0$. Then
$\ell_u(p) = m_u(\bar Z_i)$ for all $p\in \bar P_i$, where
$u = u_{n_0,n_1}$.
\end{cor}
\begin{proof}
This follows from the above proof.
\end{proof}

\begin{lemma} \label{lem:arm_init}
Assume the notation in \cref{block:arm_not} and that $\bar v(i) = n_r$
for some $\nu < r < j$. For any $l\in S_i$, there is
a unique $p\in \bar P_i$ so that $m_n(l) = \ell_n(p)$ for all
$n\in\Nd_{\bar v(i)}$.
\end{lemma}
\begin{proof}
If $e\in\E_{\bar v(i)}$, then
\cref{lem:psi_inters,lem:x_interpol}
\[
  m_{u_e}(\psi(l)) =
  \left\lceil
    \frac{\beta_e m_{\bar v(i)}(l)}{\alpha_e}
  \right\rceil
  = m_{u_e}(\bar Z_i).
\]
Let $u_+,u_-\in\V_{\bar v(i)}$, $u_\pm = u_{\bar v(i),n_{r\pm1}}$ and take
$e\in\E_{\bar v(i)}$.
By \cref{cor:reg_vx}, the functionals
$\ell_{\bar v(i)}, \ell_{u_e}, \ell_{u_+}$ form a dual basis of $\Z^3$.
Therefore, there is a $p\in\Z^3$ satisfying
\begin{equation} \label{eq:arm_init_pf}
\begin{gathered}
\begin{split}
  \ell_{\bar v(i)}(p) &= m_{\bar v(i)}(     l ) = m_{\bar v(i)}(\bar Z_i), \\
  \ell_{     u_e }(p) &= m_{     u_e }(\psi(l)) = m_{     u_e }(\bar Z_i), \\
  \ell_{     u_+ }(p) &= m_{u_+}(\psi(l))
                       = \frac{\beta_r m_{\bar v(i)}(l) + m_{n_{r+1}}(l)}
                              {\alpha_r}
                       \geq m_{u_+}(\bar Z_i).
\end{split}
\end{gathered}
\end{equation}
The formula for $\ell_{u_+}(p)$ gives an integer by \cref{lem:mod}.
Furthermore,
the inequality holds by \cref{lem:x_interpol}, using $l \geq \bar Z_i|_\Nd$.
We have therefore shown
that $\ell_u(p) = m_u(\psi(l))$ for all $u\in\V_{\bar v(i)}$ except for
$u_-$. But, since $z_l^\Nd \neq 0$, we have $(\psi(l),E_{\bar v(i)}) = 0$
by \cref{lem:factors}, hence
\[
  -b_{\bar v(i)} m_{\bar v(i)}(\psi(l))
    + \sum_{u\in\V_{\bar v(i)}} m_u(\psi(l))
  = 0 =
  -b_{\bar v(i)} \ell_{\bar v(i)}(p)
    + \sum_{u\in\V_{\bar v(i)}} \ell_u(p).
\]
Cancelling out, we obtain $m_{u_-}(\psi(l)) = \ell_{u_-}(p)$ as well.
This shows that
we could have replaced the third equation in \cref{eq:arm_init_pf} with
a corresponding line with $u_+$ replaced by $u_-$. In particular, we have
$\ell_{u_-}(p) = m_{u_-}(\psi(l)) \geq m_{u_-}(\bar Z_i)$.
We have therefore shown $\ell_u(p) = m_u(\psi(l)) \geq m_u(\bar Z_i)$
for all $u\in\V_{\bar v(i)}$. By \cref{lem:pts_faces} we have
$p\in \bar P_i$. Now, we have
$m_{n_{r+1}}(l) = \alpha_r m_{u_+}(\psi(l)) - m_{\bar v(i)}(l)
= \alpha_r \ell_{u_+}(p) - \ell_{\bar v(i)}(p) = \ell_{n_{r+1}}(p)$,
and $m_{n_{r-1}}(l) = \ell_{n_{r-1}}(p)$ similarly. Since the functionals
$\ell_{n_s}$ with $s=r-1, r, r+1$ form a dual basis of $\Q^3$, uniqueness
follows.
\end{proof}

\begin{lemma} \label{lem:hand_init}
Assume the notation in \cref{block:arm_not} and that either $\bar v(i) = n_j$,
or $\nu = 1$ and $\bar v(i) = n_1$.
For any $l\in S_i$
there is a unique $p\in \bar P_i$ so that $m_{n_{j-1}}(l) = \ell_{n_{j-1}}(p)$.
\end{lemma}
\begin{proof}
We prove the lemma in the case when $\bar v(i) = n_j$, the case
$\bar v(i) = n_1$ is similar.

Let $u_- = u_{n_j,n_{j-1}} \in \V_{\bar v(i)}$ as above and
$u_0 = u_e \in\V_{\bar v(i)}$ for some $e\in\E_{n_j}$.
If $\bar v(i)$ has a
leg group \index{leg group}
with more than
one element, choose $u_0$ from this leg group, otherwise choose $u_0$
arbitrarily. Then there is a unique $u_+\in\V_{\bar v(i)}$ lying on a leg
not in the same leg group as the leg containing $u_0$.
Define $p\in\Z^3$ using \cref{eq:arm_init_pf}, but with $u_+$ and
$n_{r+1}$ replaced with $u_-$ and $n_{r-1}$ in the third line.
Similarly as above, we find $p\in\Z^3$, as well as
$\ell_{u_+}(p) \geq m_{u_+}(\psi(l)) = m_{u_+}(\bar Z_i)$ showing
$p\in P_i$. The equation $m_{n_{r-1}} = \ell_{n_{r-1}}(p)$ now follows
from $m_{u_-}(\psi(l)) = \ell_{u_-}(p)$ as above.

Next we prove uniqueness. Use the notation in \cref{prop:arm_hands}.
We can assume that $\ell_v$ for $v=u_{n_j,n_{j-1}}, \bar v(i), u_{e_1}$
with $e_1\in\E^1_{n_j}$ form a dual basis of $\Z^3$.
Let $L_0,\ldots,L_{j-1}$ be the
associated lines. \index{associated lines}
We have
$p\in L_{j-1} \cap \convx(F_{n_j}(Z_K-E) \cup \{0\}) \cap \R_{\geq 0}^3$.
By \cref{prop:content}, we see that
$\max_{F_{n_j}(Z_K-E)} \ell_1 = \max_{F_{n_j}} \ell_1 - 1 = \alpha_e - 1$.
By the same lemma, the restriction $\ell_1|_{L_{j-1}}$ has
content \index{content}
$\alpha_e$. Therefore, $p$ is determined as the unique point on
$L_{j-1}$ for which $0\leq \ell_1(p) < \alpha_e$.
\end{proof}

\begin{lemma} \label{lem:high_pts}
Assume the same notation as above and assume that $v(i) = n_r$ for some
$\nu < r \leq j$.
Let $p\in \bar P_i$ be as defined in \cref{lem:arm_init} or
\cref{lem:hand_init}, depending on whether $r<j$ or $r=j$.
Define $m_{r-1} = \ell_{n_{r-1}}(p)$ and $m_r = \ell_{n_r}(p)$.
We have $\beta_{r-1} m_{r-1} + m_r \equiv 0 \,(\mod \alpha_{r-1})$, and we
define an
arm sequence \index{arm sequence}
$m_0, \ldots, m_j$ as in \cref{lem:m_exists}, with
associated vertices $p_{\nu+1}, \ldots, p_{j-1}$.
Then $p_s = p$ for all $s \leq r$.
\end{lemma}
\begin{proof}
By definition, it is equivalent to show $m_s = \ell_{n_s}(p)$ for $s\leq r$,
as well as $m_{r+1} = \ell_{n_{r+1}}(p)$ in case $r<j$.

First, assume that $r<j$.
Take $u_\pm = u_{n_r,n_{r\pm1}} \in \V_{n_r}$. We have
\[
  \ell_{u_-}(p)
    = \frac{\ell_{n_{r-1}}(p) + \overline\beta_{r-1}\ell_{n_r}(p)}
	       {\alpha_{r-1}}
    = \frac{m_{r-1} + \overline\beta_{r-1}m_r}
	       {\alpha_{r-1}}.
\]
Similarly as in the proof of \cref{lem:hand_init}, we have
$0\leq \ell_{n_e^*}(p) < \alpha_e$, and so
\begin{equation} \label{eq:high_pts_pf}
  \ell_u(p)
    =
      \frac{\beta_e m_r + \ell_{n_e^*}(p)}
	       {\alpha_e}
    =
	\left\lceil
      \frac{\beta_e m_r}
	       {\alpha_e}
    \right\rceil
\end{equation}
for $u = u_e \in\V_{n_r}$, where $e\in\E_{n_r}$. Hence,
subtracting \cref{eq:m_exists} from \cref{eq:nbr_sum} we get
\[
    \frac{\overline\beta_r\ell_{n_r}(p) + \ell_{n_{r+1}}(p)}
         {\alpha_r}
  = \frac{\beta_r m_r + m_{r+1}}
         {\alpha_r}
\]
showing $\ell_{n_{r+1}}(p) = m_{r+1}$. This shows $p = p_r$.
Next, we prove by descending induction that $p_s = p$ for $s< r$.
Indeed, assuming that $p_{s+1} = p$, we have
$\ell_{n_s}(p) = m_s$ and $\ell_{n_{s+1}} = m_{s+1}$. We can then follow the
same procedure as above, once we prove \cref{eq:high_pts_pf} for
$u=u_e\in\V_{n_s}$ with $e\in\E_{n_s}$.
Since $\alpha_e \ell_u = \beta_e \ell_{n_s} + \ell_{n_e^*}$, it is enough
to prove $0 \leq \ell_{n_s}(p) < \alpha_e$. The first inequality is clear,
since $p\in\Z_{\geq0}^3$. For the second, by permutation of coordinates,
we may assume that the arm $n_1, \ldots, n_j$ goes in the direction of
the coordinate $x_3$, and that $\ell_{n_e^*} = \ell_1$. By construction,
the projection of the sets $\convx(F_{n_t})\cup\{(0,0,0)\})$ to the
$x_1x_2$ plane lie within the triangle with vertices $(0,0)$, $(\alpha_e,0)$
and $(0,a)$ for some $a\in\Z_{> 0}$, for $t\geq s$, by \cref{prop:content}.
In particular, we find,
$\ell_{n_s}(p+(1,1,1)) \leq \min_{p'\in F_{n_r}} \ell_{n_s}(p'+(1,1,1))
\leq \alpha_r + \ell_{n_s}(1,1,1)$.
Equality can only hold if $p+(1,1,1) = (\alpha_r,0,*)$, which is impossible
since $p\in\Z_{\geq 0}^3$.
\end{proof}

\subsection{Plan of the proof}

The proof of \cref{thm:qZ_ident} will be broken into cases in the remaining
subsections of this section, each dealing with various technical issues
that arise. In this subsection we describe some general strategies common
to these cases.

\begin{block} \label{block:SP_bij}
For any $i$ and $p\in \bar P_i$, let
$S_{i,p} = \set{l\in S_i}{\fa{n\in\Nd_{\bar v(i)}}{m_n(l) = \ell_n(p)}}$.
\nomenclature[Sip]{$S_{i,p}$}{Subset of $S_i$}
Also, let
$S'_i = S_i \setminus \cup_{p\in \bar P_i} S_{i,p}$.
\nomenclature[Si']{$S'_i$}{Subset of $S_i$}
By \cref{lem:arm_init,lem:hand_init}, we have
$S_i = \amalg_{p\in \bar P_i} S_{i,p}$ if $\bar v(i)$ is not a central
vertex. 
By \cref{thm:pts}, the
right hand side of \cref{eq:qZ_ident} equals $|\bar P_i|$, while
the left hand side is $\sum_{l\in S_i} z_l^\Nd$. \Cref{thm:qZ_ident} 
is therefore proved as soon as we prove the equations
\begin{equation} \label{eq:count_0}
  \sum_{l\in S'_i} z_l^\Nd = 0
\end{equation}
and
\begin{equation} \label{eq:count_1}
  \sum_{l\in S_{i,p}} z_l^\Nd = 1.
\end{equation}
Although this is not always the case, 
we will follow this course of action in many of the cases.
\end{block}

\begin{lemma} \label{lem:ineq_Zi}
Let $m_\nu,\ldots, m_j$ be an
arm sequence \index{arm sequence}
as in \cref{def:arm_seq} and assume
that for some $r<j$ we have $m_r \geq m_r(\bar Z_i)$ and
$m_{r+1} \geq m_{r+1}(\bar Z_i)$, as well as 
\begin{equation} \label{eq:ineq_Zi1}
  \frac{m_ r             }{m_ r   (Z_K-E)} \leq
  \frac{m_{r+1}          }{m_{r+1}(Z_K-E)}.
\end{equation}
Then $m_s \geq m_s(\bar Z_i)$ for all $s\geq r$.
Similarly, if $r>\nu$ and $m_r \geq m_r(\bar Z_i)$ and
$m_{r-1} \geq m_{r-1}(\bar Z_i)$, as well as 
\begin{equation} \label{eq:ineq_Zi2}
  \frac{m_ r             }{m_ r   (Z_K-E)} \leq
  \frac{m_{r-1}          }{m_{r-1}(Z_K-E)},
\end{equation}
then $m_s \geq m_s(\bar Z_i)$ for all $s\leq r$.
\end{lemma}
\begin{proof}
We give the proof of the first statement, the second one is similar.
We need to prove the inequality $m_s \geq m_s(\bar Z_i)$ for $s> r+1$
as the cases $s=r,r+1$ are assumed.
By \cref{lem:ceil}, it is enough to prove $m_s > \bar r_i m_s(Z_K-E)$
(note that we can have $\epsilon_{i,n_s} \neq 0$ only if
$\bar r_i m_{n_s}(Z_K-E) \in \Z$). But this follows by using
\cref{lem:arm_mono} iteratively to find
\[
  \bar r_i \leq
  \frac{m_{r+1}}{m_{n_{r+1}}(Z_K-E)}
  <
  \frac{m_{r+2}}{m_{n_{r+2}}(Z_K-E)}
  < \ldots <
  \frac{m_j}{m_{n_j}(Z_K-E)}.
\]
\end{proof}

\begin{lemma} \label{lem:trap_pos}
Assume that a face $F_{n_0}\subset \Gamma(f)$ is a
central trapezoid \index{trapezoid}
with a nondegenerate arm $n_1, \ldots, n_j$ in the direction of the $x_1$ axis
as in \cref{block:arm_not}. Let $p_1\in F_{n_0}$ be one of the endpoints
of the segment $F_{n_0}\cap F_{n_1}$, and let $p_2\in F_{n_0}$ be the closest
integral point on the adjacent boundary segment.
Then the vector $p_2-p_1$ has nonnegative $x_2$ and $x_3$ coordinates.
\end{lemma}
\begin{proof}
We can assume that $p_1$ is on the $x_1x_2$ coordinate hyperplane.
Then $\ell_3(p_1) = 0$, thus $\ell_3(p_2 - p_1) = \ell_{p_2} \geq 0$.
Take the remaining vertices $p_3, p_4\in F_{n_0}$ so that
$[p_1,p_4] = F_{n_0} \cap F_{n_1}$. By the same argument as above, we then
have $\ell_2(p_3-p_4) \geq 0$.

If the segment $[p_1,p_4]$ is a top edge,
then $[p_2,p_3]$ is a bottom edge, and so we have $p_2-p_3 = a(p_1-p_4)$
for some integer $a>0$. Thus,
$\ell_2(p_2-p_1) = \ell_2(p_4-p_1) + \ell_2(p_3-p_4) + \ell_2(p_2-p_3)
\geq (a-1)\ell_2(p_1-p_4) = (a-1)\ell_2(p_1) \geq 0$.

If $[p_1,p_4]$ is not the top edge, then the top edge is either
$[p_3,p_4]$ or $[p_1,p_2]$.
In either case, these two edges are parallel, and so 
$\ell_2(p_2-p_1)$ and $\ell_2(p_3-p_4)$ have the same sign and the result
follows since we already proved $\ell_2(p_3-p_4) \geq 0$.
\end{proof}

\begin{lemma} \label{lem:count_empty}
If $(\bar Z_i,E_{\bar v(i)}) > 0$, then $S_i = \emptyset$.
\end{lemma}
\begin{proof}
If $l\in S_i$, then $l\geq \bar Z_i|_\Nd$ and
$m_{\bar v(i)}(l) = m_{\bar v(i)}(Z_i)$ and $\S_Z(l) \neq \emptyset$.
By \cref{lem:Sl_struct}, we then have $\psi(l) \in \S_Z(l)$, and so
$\psi(l) \geq \bar Z_i$, by \cref{lem:psi_Z_ineq}.
We get $(\psi(l),E_{\bar v(i)}) \geq (\bar Z_i,E_{\bar v(i)}) > 0$,
a contradiction.
\end{proof}

\subsection{Case: $\bar v(i)$ is central} \label{ss:case_central}

In this section we will assume that $\bar v(i)$ is a
central node. \index{central node}
We will use the notation given in
\cref{prop:anatomy}\ref{it:anatomy_central_face}.

\begin{block} \label{block:central}
Assume that $\Gamma(f)$ has three nondegenerate
arms. \index{arm}
For $k = (k_1, k_2, k_3)\in\Z^3$, define an element $l_k\in V_Z^\Nd$ as 
follows. Require $m_{n_0^1}(l_k) = m_{n_0^1}(\bar Z_i)$ and
$m_{n_1^\kappa}(l_k) = \min_{p\in \bar P_i}\ell_{n_1^\kappa}(p)
+ k_\kappa\alpha_{n_0^\kappa,n_1^\kappa}$.
Furthermore, require that for each $\kappa=1,2,3$, the sequence
$m_{n_r^\kappa}(l_k)$ is an arm sequence.
Since
$\beta_{n_0^\kappa,n_1^\kappa}\ell_{n_0^\kappa}
+ \ell_{n_1^\kappa}|_{\Z^3} \equiv 0\,
(\mod\alpha_{n_0^\kappa,n_1^\kappa})$, we have
$\beta_{n_0^\kappa,n_1^\kappa} m_{n_0^\kappa}(\bar Z_i)
+ \min_{p\in \bar P_i}\ell_{n_1^\kappa}(p)
\equiv 0\, (\mod\alpha_{n_0^\kappa,n_1^\kappa})$. Thus, $l_k$ with the
required properties exists and is unique by \cref{lem:m_exists}.
Now, by \cref{lem:mod} and \cref{rem:m_exists}, we find that if
$l \in S_i$, then there is a $k$ so that $l = l_k$. Indeed, we find
$k_\kappa = (m_{n_1^\kappa}(l) - \min_{p\in \bar P_i}\ell_{n_1^\kappa}(p))
/\alpha_{n_0^\kappa,n_1^\kappa}$.

In the case when $\Gamma(f)$ has two nondegenerate arms, define
$l_k\in V_Z^\Nd$ for $k\in \Z^2$ as above, and similarly for $k=k_1\in\Z$
if $\Gamma$ contains a single nondegenerate arm.
\end{block}

\begin{lemma} \label{lem:cgeq}
We have $l_k \geq \bar Z_i$ if and only if $k\geq 0$, that is,
$k_\kappa \geq 0$ for $\kappa=1,2,3$.
\end{lemma}
\begin{proof}
Using \cref{lem:ell_min} and \cref{lem:x_interpol} we find
\[
\begin{split}
  m_{n_1^\kappa}(l_k)
  =& \min_{p\in\bar P_i}
    \left(
      \alpha_{n_0^\kappa,n_1^\kappa}
	    (\ell_{u_{n_0^\kappa,n_1^\kappa}}(p)+k_\kappa)
        - \beta_{n_0^\kappa,n_1^\kappa} \ell_{n_0^\kappa}(p)
    \right) \\
  =& \alpha_{n_0^\kappa,n_1^\kappa}
       (m_{u_{n_0^\kappa,n_1^\kappa}}(\bar Z_i) + k_\kappa)
    - \beta_{n_0^\kappa,n_1^\kappa} m_{n_0^\kappa}(\bar Z_i) \\
  =& \alpha_{n_0^\kappa,n_1^\kappa}
     \left\lceil
	   \frac{m_{n_1^\kappa}(\bar Z_i)
	         + \beta_{n_0^\kappa,n_1^\kappa} m_{n_0^\kappa}(\bar Z_i)}
	        {\alpha_{n_0^\kappa,n_1^\kappa}}
     \right\rceil
    - \beta_{n_0^\kappa,n_1^\kappa} m_{n_0^\kappa}(\bar Z_i)
    + \alpha_{n_0^\kappa,n_1^\kappa} k_\kappa
\end{split}
\]
and so $m_{n_1^\kappa}(l_k) \geq  m_{n_1^\kappa}(\bar Z_i)$
if and only if $k_\kappa \geq 0$.

Now, assuming $k\geq 0$, we get $l_k \geq \bar Z_i$ from \cref{lem:ineq_Zi}.
\end{proof}

\begin{lemma} \label{lem:lk_inter}
With $l_k$ as above, we have
$(\psi(l_k), E_{\bar v(i)}) = (\bar Z_i, E_{\bar v(i)}) + \sum_\kappa k_\kappa$.
\end{lemma}
\begin{proof}
By construction we have
$m_{n_1^\kappa}(l_k) = \ell(p_\kappa)
+ k_\kappa \alpha_{n_0^\kappa,n_1^\kappa}$ for each $\kappa$,
where $p_\kappa\in \bar P_i$ minimizes $\ell_{n_1^\kappa}$.
Therefore
$m_{u_{n_0^\kappa,n_1^\kappa}}(\psi(l_k))
= \ell_{u_{n_0^\kappa,n_1^\kappa}}(p_\kappa) + k_\kappa
= m_{u_{n_0^\kappa,n_1^\kappa}}(\bar Z_i) + k_\kappa$.
Furthermore, if $e\in\E_{\bar v(i)}$, then
$m_{u_e}(\psi(l_k)) = m_{u_e}(\bar Z_i)$ by \cref{lem:psi_inters} and
\cref{lem:x_interpol}. Therefore,
$(l_k, E_{\bar v(i)}) = (\bar Z_i, E_{\bar v(i)}) + \sum_\kappa k_\kappa$.
\end{proof}

\begin{lemma} \label{lem:c_0}
Assume that $\bar v(i)$ is a
central node \index{central node}
and that
$(\bar Z_i,E_{\bar v(i)}) = 0$. Then the set $S_i$ consists of a single
element $l$ satisfying $z_l^\Nd = 1$.
\end{lemma}
\begin{proof}
By \cref{block:central} and \cref{lem:cgeq}, we have $l=l_k$ for some
$k\geq 0$ if $l\in S_i$. By \cref{lem:lk_inter},
we have $(l_k,E_{\bar v(i)}) = \sum_\kappa k_\kappa$,
so $z_{l_k,\bar v(i)}^\Nd = 0$
unless $k = 0$ by \cref{lem:factors}.
Thus, to prove the lemma, we must show that, indeed,
$l_0 \in S_i$. For this, we must show $z_{l_0,n_{j^c}^c}^\Nd = 1$ for all
$c$.
By \cref{thm:pts}, we have $|\bar P_i| = 1$, let $p$ be the unique point
in $\bar P_i$. Let $L_s^\kappa$ be the lines associated with the arm data
$m_{n_s^\kappa}(l_0)$ for any $\kappa$. Then $p \in L_0^\kappa$, thus
$z_{l_0,n_{j^\kappa}^\kappa}^\Nd = 1$ by \cref{lem:hand}.
\end{proof}

\begin{lemma} \label{lem:c1}
Assume that $F_{\bar v(i)}$ is a
trapezoid \index{trapezoid}
and that $\bar P_i \neq \emptyset$.
If $k$ is as in \cref{block:central} with $k\geq 0$, then
$z_{l_k,n_{j^\kappa}^\kappa}^\Nd = 1$ if $j^\kappa > 0$.
\end{lemma}
\begin{proof}
Let $L_0^\kappa,\ldots, L^\kappa_{j^\kappa-1}$
be the lines associated with the arm sequence
$m_{n_0^\kappa}(l_k),\ldots,m_{n_{j^\kappa}^\kappa}(l_k)$.
Let $p_1$ be one of the
endpoints of the segment $F_{n_0^\kappa}\cap F_{n_1^\kappa}$, and $p_2$
the closest integral point to $p_1$ on the adjacent edge of $F_{n_0^\kappa}$
with endpoint $p_1$. Take
$p\in\bar P_i$ so that $m_{n_1^\kappa}(l_k) = \ell_{n_1^\kappa}(p)$ and
set $p_0 = p+k_\kappa(p_2-p_1)$. Take $\kappa',\kappa''\in\Z$ so that
$\{\kappa,\kappa',\kappa''\} = \{1,2,3\}$. By \cref{lem:trap_pos}, $p_0$ has
nonnegative $x_{\kappa'}$ and $x_{\kappa''}$ coordinates.
Furthermore, $p_0\in L^\kappa_0$ because $\ell_{n_0^\kappa}(p_2-p_1) = 0$ and
$\ell_{n_1^\kappa}(p_2-p_1) = \alpha_{n_0^\kappa,n_1^\kappa}$.
The lemma now follows from \cref{lem:hand}.
\end{proof}

\begin{lemma} \label{lem:central_node}
If $\bar v(i)$ is a
central node, \index{central node}
then
$\sum_{l\in S_i} = \max\{0, (-\bar Z_i,E_{\bar v(i)}) + 1\}$.
\end{lemma}
\begin{proof}
The case when $(\bar Z_i,E_{\bar v(i)}) \geq 0$ is covered by
\cref{lem:count_empty,lem:c_0}. We start by showing that if $F_{\bar v(i)}$
is a triangle, then this is indeed the case. We have
$(\bar Z_i,E_{\bar v(i)}) \geq \bar r_i ((Z_K-E),E_{\bar v(i)}) = -\bar r_i$,
because $\bar Z_i \geq \bar r_i(Z_K-E)$ and
$m_{\bar v(i)}(\bar Z_i) = \bar r_im_{\bar v(i)}(Z_K-E)$.
If $\bar r_i < 1$, then the statement follows.
If $\bar r_i = 1$, then
$\bar P_i \subset (F_{\bar v(i)} - (1,1,1)) \cap \Z_{\geq 0}^3 = \emptyset$
and so $(\bar Z_i,E_{\bar v(i)}) > 1$ by \cref{thm:pts}.

We therefore assume that $F_{\bar v(i)}$ is a trapezoid and that
$(\bar Z_i, E_{\bar v(i)}) < 0$. In that case, if $k\geq 0$,
we have $z_{l_k,n}^\Nd = 1$ for all $\Nd\ni n\neq \bar v(i)$.
Writing $n = n_r^\kappa$
with $r > 0$, this follows from construction if $r < j^\kappa$, and from
\cref{lem:c1} if $r = j^\kappa$.
We therefore have $S_i = \set{l_k}{k\geq 0,\,z_{l_k,\bar v(i)}^\Nd \neq 0}$.

If $\Gamma(f)$ has exactly one nondegenerate arm, then
$\Nd_{\bar v(i)} = \{n_1^1\}$. By \cref{lem:lk_inter,lem:factors}, we have
$z_{l_k,\bar v(i)}^\Nd = 1$ if $k \leq (-\bar Z_i, E_{\bar v(i)})$, and
$z_{l_k,\bar v(i)}^\Nd = 0$ otherwise. Therefore,
\[
  \sum_{l\in S_i} z_l^\Nd
    =
	  \left|
	    \set{k\in\Z}{0\leq k \leq (-\bar Z_i,E_{\bar v(i)})}
      \right|
	= \max\{0, (-\bar Z_i,E_{\bar v(i)}) + 1\}.
\]

If $\Gamma(f)$ has two nondegenerate arms, then, for $k = (k_1, k_2)$,
we have $z_{l_k,\bar v(i)}^\Nd = 1$ if $k_1+k_2 = (-\bar Z_i,E_{\bar v(i)})$
and $z_{l_k}^\Nd = 0$ otherwise. Therefore,
\[
  \sum_{l\in S_i} z_l^\Nd
    =
	  \left|
	    \set{k\in\Z^2_{\geq0}}{k_1+k_2 = (-\bar Z_i,E_{\bar v(i)})}
      \right|
	= \max\{0,(-\bar Z_i,E_{\bar v(i)}) + 1\}.
\]

If $\Gamma(f)$ has three nondegenerate arms, then
we have $z_{l_k,\bar v(i)}^\Nd = 1$ if
$\sum_\kappa k_\kappa = (-\bar Z_i,E_{\bar v(i)})$,
$z_{l_k,\bar v(i)}^\Nd = -1$ if
$\sum_\kappa k_\kappa = (-\bar Z_i,E_{\bar v(i)})-1$ and
$z_{l_k}^\Nd = 0$ otherwise. Therefore,
\[
\begin{split}
  \sum_{l\in S_i} z_l^\Nd
    =& 
	  \left|
	    \set{k\in\Z^3_{\geq 0}}
		    {\sum_\kappa k_\kappa = (-\bar Z_i,E_{\bar v(i)})}
      \right| \\
	 & -
	  \left|
	    \set{k\in\Z^3_{\geq 0}}
		{\sum_\kappa k_\kappa = (-\bar Z_i,E_{\bar v(i)})-1}
      \right|\\
	=& \max\{0, (-\bar Z_i,E_{\bar v(i)}) + 1\}.
\end{split}
\]
\end{proof}

\begin{rem}
It is simple to verify that in the case when $\Gamma(f)$ has three
nondegenerate
arms, \index{arm}
then, for each $p\in \bar P_i$, there is a unique
element $l_p \in S_{i,p}$ and that $z_{l_p}^\Nd = 1$. Therefore
\cref{eq:count_0,eq:count_1} do indeed hold in this case.
This is, however, not generally true in the case when $\Gamma(f)$
contains a trapezoid, and only one or two arms.
\end{rem}

\subsection{Case: One or two nondegenerate arms}

In this subsection we assume that the diagram $\Gamma(f)$ has one or two
nondegenerate
arms. \index{arm}
We will assume given a fixed step $i$ in the computation sequence
and that $\bar v(i)$ is not the central node.

\begin{lemma} \label{lem:onetwo_arms}
We have $\sum_{l\in S_i} z_l^\Nd = \max\{0,(-\bar Z_i,E_{\bar v(i)})\}$.
\end{lemma}
\begin{proof}
This will follow from \cref{lem:one_arm,lem:two_arms} and
\cref{lem:arm_init,lem:hand_init}, as well as
\cref{thm:pts}.
\end{proof}

\begin{lemma} \label{lem:one_arm}
Assume that the diagram $\Gamma(f)$ contains exactly one nondegenerate
arm and that $F_{\bar v(i)}$ is not a central face. Then, for each
$p\in \bar P_i$, the set $S_{i,p}$ contains a unique element $l_p$
and $z_{l_p}^\Nd = 1$.
\end{lemma}
\begin{proof}
As in \cref{block:arm_not}, assume that $\Nd = \{n_\nu, \ldots, n_j\}$.
We can then assume that $\bar v(i) = n_r$ for some $r>0$.
Fix a $p\in \bar P_i$.
If $r>\nu$, let $m_\nu,\ldots,m_j$ be the arm sequence constructed
in \cref{lem:high_pts}, with the requirement $m_r = \ell_{n_r}(p)$
and $m_{r-1} = \ell_{n_{r-1}}(p)$. If $r = \nu = 1$, let $m_1,\ldots, m_j$
be the arms sequence defined by requiring $m_1 = \ell_{n_1}(p)$
and $m_2 = \ell_{n_2}(p)$, which exists and is unique by \cref{lem:m_exists}.
We then have an element $l_p\in V_Z^\Nd$ with $m_{n_s}(l) = m_s$ for all $s$.

The inequality $l_p \geq \bar Z_i$ follows from \cref{lem:ineq_Zi}.

Let $L_s$ for $s=\nu,\ldots,j-1$ be the lines associated with the arm sequence
$m_\nu,\ldots, m_j$. We then have $p\in L_{r-1}$ if $r>\nu$ and $p\in L_r$
if $r<j$. By \cref{lem:hand} we therefore get $z_{l_p,n_j}^\Nd = 1$.
In order to show $z_{l_p,n_\nu}^\Nd = 1$,
we must, by \cref{lem:factors}, prove $(\psi(l_p),E_{n_\nu}) \leq 0$.
We have $m_{n_\nu}(l_p) = \ell_{n_\nu}(p)$ by \cref{lem:high_pts}.
Since $-b_{n_\nu}\ell_{n_\nu}(p) + \sum_{u\in\V_{n_\nu}}\ell_u(p) = 0$, 
it is enough to show $m_u(\psi(l)) \leq \ell_u(p)$ for $u\in \V_{\bar v(i)}$.
In the case $u = u_{n_\nu,n_{\nu+1}}$, we have
$m_u = \beta_{n_\nu,n_{\nu+1}} m_{n_\nu} + m_{n_{\nu+1}} =
\beta_{n_\nu,n_{\nu+1}} \ell_{n_\nu}(p) + \ell_{n_{\nu+1}}(p)
= \ell_u(p)$ by \cref{lem:mod} and the definition of $\ell_u$.
If, however, $u = u_e$ for some $e\in\E_{n_\nu}$, then
\[
  m_u(\psi(l)) = \left\lceil \frac{\beta_e m_{n_0}}{\alpha_e} \right\rceil
    \leq \frac{\beta_e \ell_{n_0}(p) + \ell_{n_e^*}(p)}{\alpha_e}
	=     \ell_u(p).
\]
by \cref{lem:psi_inters} and the fact that
$\beta_e \ell_{n_0}(p) + \ell_{n_e^*}(p) \equiv 0\,(\mod \alpha_e)$
and $\ell_{n_e^*}(p) \geq 0$ since $p\in\Z_{\geq 0}^3$.
\end{proof}

\begin{lemma} \label{lem:two_arms}
Assume that the diagram $\Gamma(f)$ contains exactly two nondegenerate
arms and that $\bar v(i)$ is not a central face. Then, for each
$p\in \bar P_i$, the set $S_{i,p}$ contains a unique element $l_p$
and $z_{l_p}^\Nd = 1$.
\end{lemma}

\begin{proof}
Use the notation given in \cref{prop:anatomy}.
We can then assume that $\bar v(i) = n_r^1$ for some $r\geq 1$.
Similarly as above, using \cref{lem:m_exists}, we find numbers
$m_0^1,\ldots, m_{j^1}^1\in\Z$ so that if $l\in S_{i,p}$, then
$m_{n_r^1}(l) = m_r^1$ for $0\leq r \leq j^1$. Furthermore,
by \cref{lem:high_pts}, we have $m_s^1 = \ell_{n_s^1}(p)$ for
$s\leq r+1$. If $\Gamma(f)$ contains a central edge, let $c$ be the
number of central edges. If $\Gamma(f)$ contains a central node,
set $c = 0$. In either case, we have $n_s^1 = n_{c-s}^2$ for
$s\leq c$.
Note that in the case of a central edge, we can assume $j^\kappa>c-1$
for $\kappa=1,2$, since otherwise the statement is covered by
\cref{lem:one_arm}. In particular, we have nodes
$n_0^1,n_0^2\in\Nd$.

In the case of a central edge, we therefore have
$m_{n_s^2}(l) = m_{c-s}^1$ for $s=0,1$, for all $l\in S_{i,p}$.
Let $m_0^2,\ldots, m_{j^2}^2$ be the arm sequence with
$m_0^2 = m_c^1$ and $m_1^2 = m_{c-1}^2$.
Then, for any $l\in S_{i,p}$, we have $m_{n_s^2}(l) = m_s^2$.

In the case of a central node, we have a number $m_1^2\in\Z$, uniquely
determined by the equation
\[
  \frac{m_1^2 + \beta_{n_0^2,n_1^2} m_0^1}{\alpha_{n_0^2,n_1^2}}
  - b_{n_0^1} m_0^1 +
  \frac{\beta_{n_0^1,n_1^1} m_0^1 + m_1^1}{\alpha_{n_0^1,n_1^1}}
  + \sum_{e\in\E_{n_0^1}} \frac{\beta_e m_0^1}{\alpha_e}
  = 0.
\]
Setting $m_0^2 = m_0^1$, \cref{lem:m_exists} determines an arm sequence
$m_s^2$ with $m_{n_s^2}(l) = m_s^2$ for all $0\leq s\leq j^2$ and
$l\in S_{i,p}$.

We define $l_p\in V_Z^\Nd$ by $m_{n_s^e}(l_p) = m_s^e$.
We have proved that if $l\in S_{i,p}$, then $l = l_p$.
To prove the lemma, we must show that indeed, $l_p \in S_{i,p}$.
For this, we need to prove that $z_{l_p,n_{j^e}^e}^\Nd = 1$ for $e=1,2$
and that $l_p \geq \bar Z_i$. As in the case of a single nondegenerate
arm, we find $z_{l_p,n_{j^1}^1}^\Nd = 1$, and
$m_{n_s^1}(l_p) \geq m_{n_s^1}(\bar Z_i)$ for all $s$.

Let $L_s^\kappa$ be the lines associated with the arm sequence
$m_0^\kappa,\ldots, m_{j^\kappa}^\kappa$. In the case when $\Gamma(f)$
contains a central edge, note that $L_0^2 = L_{c-1}^1$. In particular,
$p\in L_0^2$, and so $z_{l_p,n_{j^2}^2}^\Nd = 1$ by \cref{lem:hand}.
It is also clear that $m_s^2 \geq m_{n_s^2}(\bar Z_i)$ for $s = c,c-1$
and that 
\[
  \frac{ \ell_{n_{c-1}^2}(p) }{ m_{n_{c-1}^2}(Z_K-E) }
  \leq
  \frac{ \ell_{n_c^2}(p) }{ m_{n_c^2}(Z_K-E) }
\]
since $p\in C_{n_r^2}$. Therefore, by \cref{lem:ineq_Zi}, we
have $m_{n_s^2}(l_p) \geq m_{n_s^2}(\bar Z_i)$ for $s\geq c-1$,
hence $l_p \geq \bar Z_i$.

Next we consider the case when $\Gamma(f)$ contains a central node.
We need to prove $m_s^2 \geq m_{n_s^2}(\bar Z_i)$ for $s\geq 1$
and $z_{l_p,n_{j^2}^2}^\Nd = 1$. The former follows in a similar way
as above as soon as we prove
\begin{equation} \label{eq:58}
  \frac{ m_0^2 }{ m_{n_0^2}(Z_K-E) }
  \leq
  \frac{ m_1^2 }{ m_{n_1^2}(Z_K-E) }.
\end{equation}
Comparing the two equations
\[
  \frac{m_1^2 + \beta_{n_0^2,n_1^2} m_0^2}{\alpha_{n_0^2,n_1^2}}
  - b_{n_0^1} m_0^1 +
  \frac{\beta_{n_0^1,n_1^1} m_0^1 + m_1^1}{\alpha_{n_0^1,n_1^1}}
  + \sum_{e\in\E_{n_0^1}}
    \left\lceil
	  \frac{\beta_e m_0^1}{\alpha_e}
    \right\rceil
  = 0
\]
and
\[
  \frac{\ell_{n_1^2} + \beta_{n_0^2,n_1^2} \ell_{n_0^2}(p)}
       {\alpha_{n_0^2,n_1^2}}(p)
  - b_{n_0^1} \ell_{n_0^1} +
  \frac{\beta_{n_0^1,n_1^1} \ell{n_0^1}(p) + \ell{n_1^1}(p)}
       {\alpha_{n_0^1,n_1^1}}
  + \sum_{e\in\E_{n_0^1}}
	  \frac{\beta_e \ell_{n_0^1}(p) + \ell_{n_e^*}(p)}
	       {\alpha_e}
  = 0,
\]
and the fact that $\ell_{n_0^1}(p) = m_0^1 = m_0^2$ and
$\ell_{n_1^1}(p) = m_1^1$ we find $m_1^2 \geq \ell_{n_1^2}(p)$, hence
\[
  \frac{ m_0^2 }{ m_{n_0^2}(Z_K-E) }
  \leq
  \frac{ \ell_{n_0^2}(p) }{ m_{n_0^2}(Z_K-E) }
  \leq
  \frac{ \ell_{n_1^2}(p) }{ m_{n_1^2}(Z_K-E) }
  \leq
  \frac{ m_1^2 }{ m_{n_1^2}(Z_K-E) },
\]
proving \cref{eq:58}. We observe from these equations that we also have
$m_1^2 \equiv \ell_{n_1^2}(p)\, (\mod \alpha_{n_0^2,n_1^2} )$

Finally, we will prove $z_{l_p,n_{j^2}^2}^\Nd = 1$. By \cref{lem:hand},
it is enough to prove that the line $L_0^2$ contains a point
with nonnegative $x_1$ and $x_3$ coordinates.

We start with the case when $F_{n_0^1}$ is a trapezoid. Let
$p_1$ be one of the endpoints of the segment $F_{n_0^2} \cap F_{n_1^2}$,
and $p_2$ the closest integral point on an adjacent boundary segment
of $F_{n_0^2}$. By \cref{lem:trap_pos}, the vector $p_2-p_1$ has
nonnegative $x_1$ and $x_3$ coordinates. Since $m_1^2 \geq \ell_{n_1^2}(p)$,
as we proved above, the same holds for the point
\[
  p_0 = p + \frac{m_1^2 - \ell_{n_1^2}(p)}{\alpha_{n_0^2,n_1^2}}(p_2-p_1)
\]
which is an integral point by our previous observation.
Since $\ell_{n_0^2}(p_2-p_1) = 0$ and 
$\ell_{n_1^2}(p_2-p_1) = \alpha_{n_0^2,n_1^2}$, by \cref{prop:content},
we find $p_0 \in L_0^2$.

Next, we will prove $z_{l_p,n_{j^2}^2}^\Nd = 1$, assuming
that $F_{n_0^1}$ is a central triangle.
Define a point $p_0$
by requiring $\ell_n(p) = m_n(l_p)$ for $n = n_1^1, n_0^1, n_1^2$.
Using the same proof as in \cref{lem:armseq_points}, we see that
$p_0$ exist, is unique, and $p_0\in\Z^3$. By definition, we also have
$p_0\in L_0^1$ and $p_0\in L_1^2$, so, as in the previous case, it
suffices to show that $\ell_1(p_0) \geq 0$ and $\ell_3(p_0) \geq 0$.
Since $F_{n_0^1}$ is a central triangle, and $\Gamma(f)$ has one
degenerate arm, we have $|\E_{n_0^1}| = 1$. Let $e\in\E_{n_0^1}$ be the
unique element in this set.
We then have
\[
  m_{u_e}(l_p) =
    \left\lceil
      \frac{\beta_e m_{n_0^1}(l_p)}{\alpha_e}
    \right\rceil
    \leq \frac{\beta_e \ell_{n_0^1}(p) + \ell_{n_e^*}(p)}{\alpha_e}
	=    \ell_{u_e}(p).
\]
Furthermore, subtracting \cref{eq:nbr_sum} (with $n = n_0^1$),
evaluated at $p_0$, from
$(l_p,E_{n_0^2}) = 0$, we get $m_{u_e}(l_p) = \ell_{u_e}(p_0)$, thus,
$\ell_{u_e}(p-p_0) \geq 0$.
Evaluating \cref{eq:nbr_sum} at $p-p_0$ gives
$\ell_{u_e}(p-p_0) + \ell_{n_1^2}(p-p_0)/\alpha_{n_0^2,n_1^2} = 0$ and so
$\ell_{n_1^2}(p-p_0) \leq 0$.

Give names $q_1, q_2, q_3$ to the vertices of the triangle $F_{n_0^1}$
as in \cref{fig:shoulders}, that is, $q_1$ lies on the $x_2x_3$ axis, etc.
By definition, we have $p, p_0 \in L_0^1$. Furthermore, the line $L_0^1$ is
parallel to the primitive vector $q_2-q_3$. Therefore, there is a $k\in\Z$ so
that $p_0 = p + k(q_2-q_3)$. By convexity of $\Gamma_+(f)$ we have
$\ell_{n_1^2}(q_2-q_3) \geq 0$. Therefore, by the previous inequality,
we get $k \geq 0$. Since $\ell_3(q_2-q_3) = \ell_3(q_2) \geq 0$, we have
$\ell_3(p_0) \geq 0$.
If $k = 0$, then $p_0 = p$, and we get $\ell_1(p_0) \geq 0$.
Otherwise, we have $p_0 = p + k(q_2-q_3)$ with $k > 0$. Since
$p\in (\cup_{r=1}^{j^1} C_{n_1^r} \cap \Gamma_-(f))-(1,1,1)$, we have
$\ell_2(q_3 - q_2) > \ell_2(p)$, therefore, $\ell_2(p_0) < 0$.
Since the arm in the direction of the $x_3$ axis is assumed degenerate,
we have $a,b\in\Z_{>0}$ so that $\ell_{n_e^*} = a\ell_1 + b\ell_2$.
Furthermore, we have
\[
  \ell_{n_e^*}(p_0) = \alpha_e\ell_{n_0^1}(p_0) - \beta_e\ell_{u_e}(p_0)
    = 
    \left\lceil
      \frac{\beta_e m_{n_0^1}(l_p)}{\alpha_e}
    \right\rceil
	-\beta_e m_{n_0^1}(l_p)
	\geq 0.
\]
All this gives $\ell_1(p_0) \geq 0$, finishing the proof.
\end{proof}

\subsection{Case: Three nondegenerate arms}

In this subsection we will assume that the diagram $\Gamma(f)$ contains
a
central node \index{central node}
and three nondegenerate
arms. \index{arm}
We will assume given a fixed step $i$ in the
computation sequence \index{computation sequence}
and that $\bar v(i)$ is not the central node.

\begin{block}
We use the notation
introduced in \cref{prop:anatomy}\ref{it:anatomy_central_face}.
We can assume that for some $1\leq r \leq j^1$ we have $\bar v(i) = n_r^1$.
By \cref{lem:arm_init,lem:hand_init}, we have $S_i' = \emptyset$,
so in order to prove \cref{eq:qZ_ident}, it is enough to prove
$\sum_{l\in S_{i,p}} z_l^\Nd = 1$ for all $p\in \bar P_i$.
For $\kappa=2,3$, define
\[
  S_{i,p}^\kappa
  = \set{ l\in S_{i,p} }{ m_{n_1^\kappa}(l) < \ell_{n_1^\kappa}(p) }
\] \nomenclature[Sipk]{$S_{i,p}^\kappa$}{Subset of $S_i$}
and set $S_{i,p}^0 = S_{i,p} \setminus (S_{i,p}^2 \cup S_{i,p}^3)$.
\end{block}

\begin{lemma} \label{lem:three_arms}
We have $\sum_{l\in S_i} z_l^\Nd = \max\{0,(-\bar Z_i,E_{\bar v(i)})\}$.
\end{lemma}
\begin{proof} 
This will follow from \cref{lem:3_inters,lem:3_sum_0,lem:3_sum_1} and
\cref{thm:pts}.
\end{proof}

\begin{block} \label{block:mn13}
By \cref{lem:m_exists}, there is an
arm sequence \index{arm sequence}
$m_0^1,\ldots, m_{j^1}^1$
satisfying $m_{r-1}^1 = \ell_{n_{r-1}^1}(p)$ and $m_r^1 = \ell_{n_r^1}(p)$,
and we have $m_{n_s^1}(l) = m_s^1$ for all $l\in S_{i,p}$.
Fix an $l\in S_{i,p}$. We then have
$m_{n_1^2}(l) \equiv - \beta_{n_0^2,n_1^2} m_{n_0^2}(l)
= - \beta_{n_0^2,n_1^2} \ell_{n_0^2}(p) \equiv \ell_{n_1^2}(p)\,
(\mod \alpha_{n_0^2,n_1^2} )$, and so there is a $k\in \Z$ so that
$m_{n_1^2}(l) = \ell_{n_1^2}(p) + k \alpha_{n_0^2,n_1^2}$.
Using the equation
\[
  -b_{n_0^1} \ell_{n_0^1} +
  \sum_{\kappa=1}^3 \frac{\beta_{n_0^\kappa,n_1^\kappa} \ell_{n_0^\kappa}
                          + \ell_{n_1^\kappa}}
                    {\alpha_{n_0^\kappa,n_1^\kappa}}
  +
  \sum_{e\in\E_{n_0^1}} \frac{\beta_e \ell_{n_0^1} + \ell_{n_e^*}}
                             {\alpha_e}
  = 0
\]
and the fact that $(\psi(l),E_{n_0^1}) =: \eta \in \{ 0, -1 \}$ by
\cref{lem:factors}, we find that
\begin{equation} \label{eq:mn13}
  \frac{\ell_{n_1^3}(p) - m_1^{3,k,\eta}}{\alpha_{n_0^3,n_1^3}}
  =
  k - \eta +
  \sum_{e\in\E_{n_0^1}}
    \left\lceil
      \frac{\beta_e m_{n_0^1}}{\alpha_e}
    \right\rceil
	-
    \frac{\beta_e \ell_{n_0^1}(p) + \ell_{n_e^*}(p)}{\alpha_e}.
\end{equation}
where $m_1^{3,k,\eta} = m_{n_1^3}(l)$.
Note that since $(\beta_e \ell_{n_0^1} + \ell_{n_e^*})/\alpha_e$ is an integral
functional, the summand corresponding to $e\in\E_{n_0^1}$
on the right in \cref{eq:mn13} is integral.
Since $\ell_{n_0^1}(p) = m_{n_0^1}$ and $\ell_{n_e^*}(p) \geq 0$,
each such summand is $<1$. Thus, it follows that these summands are nonpositive.
\end{block}

\begin{definition}
For $k\in \Z$, define $m_0^{2,k} = m_0^1$ and
$m_1^{2,k} = \ell_{n_1^2}(p) + k \alpha_{n_0^2,n_1^2}$.
Furthermore, for $\eta = 0,-1$, let $m_0^{3,k,\eta} = m_0^1$ and
define $m_1^{3,k,\eta}$ as the unique solution to \cref{eq:mn13}.
Then, by \cref{lem:m_exists}, there exist unique arm sequences
$(m_s^{2,k})_{s=0}^{j^2}$ and
$(m_s^{3,k,\eta})_{s=0}^{j^3}$ with the given first two initial terms.
Define $l_p^{k,\eta}\in V_Z^\Nd$ by
$m_{n_s^1}(l_p^{k,\eta}) = m_s^1$ for $0\leq s \leq j^1$,
$m_{n_s^2}(l_p^{k,\eta}) = m_s^{2,k}$ for $0\leq s\leq j^2$ and
$m_{n_s^3}(l_p^{k,\eta}) = m_s^{3,k,\eta}$ for $0\leq s\leq j^3$.
\end{definition}

In \cref{block:mn13} we have thus proven
\begin{lemma} \label{lem:lpke}
If $l\in S_{i,p}$, then $l = l_p^{k,\eta}$ for some $\Z$ and $\eta\in\{0,-1\}$.
In fact, we have
\[
  S_{i,p} = \set{l_p^{k,\eta}}
                {
				  l_p^{k,\eta} \geq \bar Z_i,\,
				  z_{l_p^{k,\eta},n_{j^2}^2}^\Nd = 
				  z_{l_p^{k,\eta},n_{j^3}^3}^\Nd = 1
                }.
\]
\qed
\end{lemma}

\begin{definition}
Let $k_0\in \Z$ be the unique number so that $m_1^{3,k_0,0} = \ell_{n_1^3}(p)$.
\end{definition}

It is clear from the remark after \cref{eq:mn13} that $k_0 \geq 0$.

\begin{lemma} \label{lem:3_inters}
We have $S_{i,p}^2 \cap S_{i,p}^3 = \emptyset$.
\end{lemma}
\begin{proof}
If $l_p^{k,\eta} \in S_{i,p}^2 \cap S_{i,p}^3$, then, by definition,
$k < 0$ and $k > k_0+\eta \geq 0$. This is clearly impossible.
\end{proof}

\begin{lemma}
We have $m_{n_s^1}(l_p^{k,\eta}) \geq m_{n_s^1}(\bar Z_i)$ for
$0\leq s\leq j^1$ and $z_{l_p^{k,\eta},n_{j^1}^1}^\Nd = 1$
for any $k,\eta$.
\end{lemma}
\begin{proof}
This follows in exactly the same way as the corresponding statement
in the proof of \cref{lem:one_arm}.
\end{proof}

\begin{lemma} \label{lem:t_geq}
If $k\geq 0$ then $m_{n_s^2}(l_p^{k,\eta}) \geq m_{n_s^2}(\bar Z_i)$
for $0\leq s \leq j^2$. Similarly, if $k \leq k_0+\eta$, then
$m_{n_s^3}(l^{k,\eta}) \geq m_{n_s^3}(\bar Z_i)$
for $0\leq s \leq j^3$.
\end{lemma}
\begin{proof}
We prove the statement for the second arm. The statement for the third arm
follows similarly.
If $k\geq 0$, then
\[
  \frac{m_{n_1^2}(l_p^{k,\eta})}{m_{n_1^2}(Z_K-E)} \geq
  \frac{\ell_{n_1^2}(p)}{m_{n_1^2}(Z_K-E)} \geq
  \frac{\ell_{n_0^2}(p)}{m_{n_0^2}(Z_K-E)} =
  \frac{m_{n_0^2}(l_p^{k,\eta})}{m_{n_0^2}(Z_K-E)},
\]
since $p\in\R^3_{\geq 0} \setminus\cup_{r=1}^{j^2}C_{n_r^2}$ (as in the proof
of \cref{lem:arm_mono}). Thus, the result follows from \cref{lem:ineq_Zi}.
\end{proof}

\begin{lemma} \label{lem:hand_one2}
\begin{enumerate}

\item \label{it:hand_one_trak}
If $F_{n_0^1}$ is a trapezoid and $k \geq 0$, then
$z_{l_p^{k,\eta}, n_{j^2}^2}^\Nd = 1$.

\item \label{it:hand_one_trak0}
If $F_{n_0^1}$ is a trapezoid and $k \leq k_0 + \eta$, then
$z_{l_p^{k,\eta}, n_{j^3}^3}^\Nd = 1$.

\item \label{it:hand_one_tri}
If $F_{n_0^1}$ is a triangle, then
$z_{l_p^{0,0}, n_{j^2}^2}^\Nd = 1$ and
$z_{l_p^{0,0}, n_{j^3}^3}^\Nd = 1$.

\end{enumerate}
\end{lemma}
\begin{proof}
We start by proving \ref{it:hand_one_trak}, the proof of 
\ref{it:hand_one_trak0} is similar.
Let $p_1$ be one of the endpoints of the segment $F_{n_0^2}\cap F_{n_0^2}$
and $p_2$ the closest integral point to $p_1$ on the adjacent
boundary segment of $F_{n_0^2}$. Let $p_0 = p + k(p_2-p_1)$.
By \cref{prop:content} we then have
$\ell_n(p_0) = m_n(l_p^{k,\eta})$ for $n = n_0^2, n_1^2$. Thus, the result
follows from \cref{lem:hand}

\ref{it:hand_one_tri} follows in a similar way, since
$m_n(l_p^{0,0}) = \ell_n(p)$ for $n=n_0^2,n_1^2,n_1^3$.
\end{proof}

\begin{lemma} \label{lem:3_sum_0}
We have $\sum_{l\in S_{i,p}^\kappa} z_l^\Nd = 0$ for $\kappa=2,3$. 
\end{lemma}
\begin{proof}
We prove the lemma for $\kappa=2$, the case $\kappa=3$ follows similarly.

For any $l\in S_{i,p}^2$ we have $z_l^\Nd = \pm 1$. In fact, there are
$k\in \Z_{<0}$ and $\eta \in \{0,-1\}$ so that $l = l_p^{k,\eta}$.
Then $z_l^\Nd = (-1)^\eta$.
Therefore, the lemma is proved as soon as we prove that
for any $k\in\Z_{<0}$ we have $l_p^{k,0} \in S_{i,p}^2$ if and
only if $l_p^{k,-1} \in S_{i,p}^2$. Now,
$m_{n_s^2}(l_p^{k,0}) = m_{n_s^2}(l_p^{k,-1})$ for all $k$. In particular
we have $z_{l_p^{k,0},n_{j^2}^2}^\Nd = z_{l_p^{k,-1},n_{j^2}^2}^\Nd$.
Furthermore, by \cref{lem:t_geq}, we have
$m_{n_s^3}(l_p^{k,\eta}) \geq m_{n_s^3}(\bar Z_i)$ for any $k<0$ (since
$k_0 \geq 0$). It therefore suffices to prove that
if $z_{l_p^{k,0},n_{j^2}^2}^\Nd = 1$ then
$z_{l_p^{k,0},n_{j^3}^3}^\Nd = 1$ if and only if
$z_{l_p^{k,-1},n_{j^3}^3}^\Nd = 1$ for all $k<0$.
But this follow immediately from \cref{lem:all_k}
\end{proof}

\begin{lemma} \label{lem:all_k}
If $F_{n_0^1}$ is a triangle and $\{a,b,c\} = \{1,2,3\}$, define
the points $q_1, q_2, q_3 \in \Z^3$ as the vertices of 
$F_{n_0^1}$ so that $q_a,q_b$ are the end points of the segment
$F_{n_0^c}\cap F_{n_1^c}$.
\begin{enumerate}

\item \label{it:all_k_trap}
If $F_{n_0^1}$ is a trapezoid, then $z_{l_p^{k,\eta},n_{j^3}^3}^\Nd = 1$
for any $k<0$ and $\eta\in\{0,-1\}$.

\item \label{it:all_k_1}
If $F_{n_0^1}$ is a triangle and either $\ell_2(p) \geq \ell_2(q_2)$
or $\ell_1(q_2) \leq \ell_1(q_3)$,
then $z_{l_p^{k,\eta},n_{j^3}^3}^\Nd = 1$
for any $k<0$ and $\eta\in\{0,-1\}$.

\item \label{it:all_k_0}
If $F_{n_0^1}$ is a triangle, $\ell_2(p) < \ell_2(q_2)$, and
$\ell_1(q_2) \geq \ell_1(q_3)$, then
then $z_{l_p^{k,\eta},n_{j^2}^2}^\Nd = 0$
for any $k<0$ and $\eta\in\{0,-1\}$.

\end{enumerate}
\end{lemma}
\begin{proof}
Take $k\in\Z_{<0}$ and $\eta\in\{0,-1\}$.
Let $L_0^3,\ldots, L_{j^3-1}^3$ be the lines associated with the arm sequence
$m_{n_0^3}(l_p^{k,\eta}),\ldots, m_{n_{j^3}^3}(l_p^{k,\eta})$.

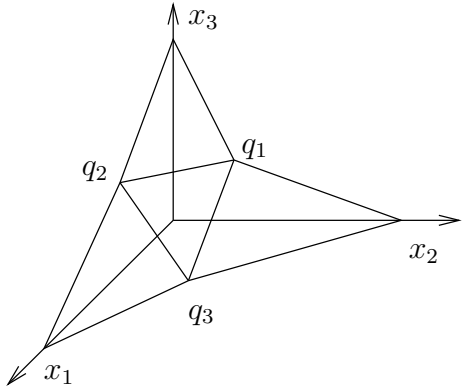
\mynd{shoulders}{ht}{$q_1, q_2, q_3$ are the vertices
of the triangle $F_{n_0^1}$.}{fig:shoulders}

\ref{it:all_k_trap}  Define $p_0$ by
$ p_0 = p + \left( k_0+\eta - k \right) (q_3 - q_1)$.
Using \cref{lem:trap_pos}, and the fact that $k_0+\eta - k \geq 0$, we find
that $p_0$ has nonnegative $x_1$ and $x_2$ coordinates. Furthermore, 
\cref{prop:content} gives $\ell_{n_0^3}(p_0) = m_{n_0^3}(l_p^{k,\eta})$ and
$\ell_{n_1^3}(p_0) = m_{n_1^3}(l_p^{k,\eta})$, that is,
$p_0 \in L_0^3$. The result now follows from \cref{lem:hand}.

\ref{it:all_k_1}
As in the previous case, the result will follow as soon as we prove that
$L_0^3$ contains an integral point with nonnegative $x_1$ and $x_2$
coordinates.

First, we assume $\ell_2(p) \geq \ell_2(q_2)$.
Let $A = \R_{\geq 0}\times \R_{\geq 0} \times \R\subset \R^3$.
We want to show $L_0\cap A \cap \Z^3 \neq \emptyset$.
Now, (for the purposes of this proof only) let $\pi^3$ be the canonical
projection from $\R^3$ to the $x_1x_2$ plane. Furthermore, let
$\ell$ be a linear function on the $x_1x_2$ plane so that
$\ell(\pi(q_1)) = \ell(\pi(q_2)) > 0$. It is then clear that
$\ell(\pi(q_3)) > \ell(\pi(q_2))$.
If we define $p_0$ by the same method as in the previous case, it is not
necessarily true that $p_0 \in A$. We see, however, that
$\ell(\pi(p_0)) \geq \ell(\pi(p))$. Let $L\in \R^3$ be the line which
is parallel to $L_0^3$ and passes through $p$. We find that
the segment $\pi(L)\cap A$ is longer than the segment
$\pi(L_s^3)\cap A$. This implies that the segment $L_0^3\cap A$ is longer
than the segment $L\cap A$. Now, the segment $L\cap A$ contains
$p$, as well as $p + q_1-q_2$, by hypothesis, and so has length at least
one. Thus, $L_0^3\cap A$ has length at least one as well. But a segment 
of length at least one contains an integral point.

\mynd{proj}{ht}{A projection.}{fig:proj}

Now, if $\ell_1(q_2) \geq \ell_1(q_3)$, then we can proceed in a similar
fashion as in \ref{it:all_k_trap}. Indeed, if we define
$p_0 = p + k(q_3-q_2)$, then, by our assumptions, we find that $p_0$
has nonnegative $x_1$ and $x_3$ coordinates. Furthermore,
we have $p_0\in L_0^3$, and so the result follows from \cref{lem:hand}.

\ref{it:all_k_0}
In this case, let $L_0^2,\ldots, L_{j^3}^2$ be the lines associated with
the arms sequence $m_{n_0^2}(l_p^{k,\eta}),\ldots,m_{n_{j^2}^2}(l_p^{k,\eta})$.
Using \cref{prop:content}, we find that
$p_0 = p + k(q_2-q_3) \in L_0^2$. The vector $q_1 - q_3$ is primitive
and we have $\ell_{n_0^2}(q_1 - q_3) = \ell_{n_1^2}(q_1 - q_3) = 0$.
Thus, $L_0^2 \cap \Z^3 = \set{p_0 + h(q_1 - q_3)}{h\in \Z}$.
It is clear that $\ell_3(q_2) > \ell_3(p)$
and $\ell_3(q_3) = 0$. Since $k<0$, we get $\ell_3(p_0) < 0$,
so we find $\ell_3(p_0 + h(q_1 - q_3)) < 0$ for all $h\geq 0$.
If, however, $h>0$, then
\[
\begin{split}
  \ell_1(p_0 + h(q_1 - q_3))
    &=    \ell_1(p_0 + (q_1-q_3)) + (h-1)\ell_1(q_1-q_3) \\
    &\leq \ell_1(p_0 + (q_1-q_3)) \\
    &=    \ell_1(p + k(q_2-q_3) + (q_1-q_3)) \\
    &=    \ell_1(p + (k+1)(q_2-q_3) + (q_1-q_2)) \\
    &=    \ell_1(p + (q_1-q_2)) + \ell_1((k+1)(q_2-q_3)) \\
	&<    0.
\end{split}
\]
Here we use both assumptions in the last inequality.
We have thus proved that no integral point in the line $L_0^2$
has nonnegative $x_1$ and $x_3$ coordinates. By \cref{lem:hand}
we get $z_{l_p^{k,\eta},n_{j^2}^2}^\Nd = 0$.
\end{proof}

\begin{lemma} \label{lem:3_sum_1}
We have $\sum_{l\in S_{i,p}^0} z_l^\Nd = 1$.
\end{lemma}
\begin{proof}
By definition, and \cref{lem:lpke}, we have
\begin{equation} \label{eq:S0_descr}
  S_{i,p}^0 \subset
  \{
    l_p^{0, 0},\ldots,l_p^{k_0  , 0},
    l_p^{0,-1},\ldots,l_p^{k_0-1,-1}
  \}.
\end{equation}
Since
$z_{l_p^{k,\eta},n_0^1}^\Nd = (-1)^\eta$
by \cref{lem:factors}, the lemma is proved as soon as we prove
equality in \cref{eq:S0_descr}.

In the case of a triangle, it follows from definition that  $k_0 = 0$.
Therefore,
\cref{lem:hand_one2} shows that for any element $l$ of
the right hand side of \cref{eq:S0_descr}, we have
$z_{l,n_{j^2}^2}^\Nd = z_{l,n_{j^3}^3}^\Nd = 1$. Furhtermore,
\cref{lem:t_geq} shows that for such an $l$ we have $l\geq \bar Z_i$.
Thus, $l\in S_{i,p}$ and so equality holds in \cref{eq:S0_descr}.
\end{proof}

\newpage
\bibliographystyle{bplain}
\bibliography{bibliography}

\newpage

\printnomenclature

\printindex

\end{document}